\documentclass[12pt]{article}

\usepackage[export]{adjustbox}  
\usepackage{amsmath}
\usepackage{amssymb}
\usepackage{amsthm}
\usepackage[title]{appendix}
\usepackage{array}
\usepackage[backend=bibtex,giveninits=true,style=alphabetic,sorting=nyt,doi=false,isbn=false,url=false,eprint=true]{biblatex}
\usepackage{blindtext}
\usepackage{booktabs} 
\usepackage{caption}
\usepackage{color}
\usepackage{csquotes}
\usepackage{enumitem}
\usepackage{float}
\usepackage[margin=1in]{geometry}
\usepackage{graphicx}
\usepackage[hidelinks,linktoc=all]{hyperref}
\usepackage{nicefrac}
\usepackage{parskip}
\usepackage{rotating}
\usepackage{setspace}   
\usepackage{stackengine}
\usepackage{stmaryrd}   
\usepackage{tikz}
\usepackage{titlesec}
\usepackage{mathtools}
\usepackage{multicol} 
\usepackage{verbatimbox}

\usepackage{lipsum}

\theoremstyle{definition}
\newtheorem{definition}{Definition}
\newtheorem{remark}{Remark}
\newtheorem{convention}{Convention}
\newtheorem{proposition}{Proposition}
\newtheorem{claim}{Claim}
\theoremstyle{plain}
\newtheorem{theorem}{Theorem}
\newtheorem{lemma}{Lemma}
\newtheorem{corollary}{Corollary}

\newcommand{\upK}{^{\mbox{\normalsize $\underset{\scriptstyle n}{\times}K$}}}
\newcommand{\upKnmt}{^{\mbox{\normalsize $\underset{\scriptstyle n-2}{\times}K$}}}
\newcommand{\upU}{^{\mbox{\normalsize $\underset{\scriptstyle n}{\times}U$}}}
\newcommand{\mysym}[1]{{\adjustbox{valign = c}{\includegraphics[]{Figures/sym#1.pdf}}}}

\newcommand{\defw}[1]{{\bf{\textit{#1}}}}
\newcommand{\ordst}{\textsuperscript{st}\,}

\newcommand{\ordth}{\textsuperscript{th}\,}

\newcommand{\ZZ}{{\mathbb{Z}}}
\newcommand{\QQ}{{\mathbb{Q}}}
\newcommand{\RR}{{\mathbb{R}}}
\newcommand{\CC}{{\mathbb{C}}}
\newcommand{\id}{\textnormal{id}}
\newcommand{\cobIV}{\mathcal{C}ob^4}
\newcommand{\cobIVd}{\mathcal{C}ob^{4\bullet}}
\newcommand{\cobIII}{\mathcal{C}hron\mathcal{C}ob^3_{Odd}}
\newcommand{\cobIIId}{\mathcal{C}hron\mathcal{C}ob^{3\bullet}_{Odd}}
\newcommand{\cobIIIM}{\textnormal{Mat}(\mathcal{C}hron\mathcal{C}ob^3_{Odd})}
\newcommand{\cobIIIKM}{\textnormal{Kom}(\textnormal{Mat}(\mathcal{C}hron\mathcal{C}ob^3_{Odd}))}
\newcommand{\cobIIIdKM}{\textnormal{Kom}(\textnormal{Mat}(\mathcal{C}hron\mathcal{C}ob^{3\bullet}_{Odd}))}
\newcommand{\cobIIIKbM}{K^b(\textnormal{Mat}(\mathcal{C}hron\mathcal{C}ob^3_{Odd}))}
\newcommand{\cobIIIKbMpm}{K^b_\pm(\textnormal{Mat}(\mathcal{C}hron\mathcal{C}ob^3_{Odd}))}
\newcommand{\okh}{{\textnormal{Okh}}}
\newcommand{\ff}{{\mathcal{F}}}

\newcommand{\edge}{{\text{edge}}}
\newcommand{\arc}{{\text{arcs}}}
\newcommand{\col}{{\text{Col}}}
\newcommand{\dott}{{\text{dot}}}
\newcommand{\bdbc}{{\Sigma_2}}

\newcommand{\includeMov}[1]{\adjustbox{valign = c}{\includegraphics[scale=0.0375]{Figures/#1.pdf}}}
\newcommand{\includeFig}[1]{\adjustbox{valign = c}{\includegraphics[scale=1]{Figures/#1.pdf}}}

\newcommand{\includeCob}[1]{\adjustbox{valign = c}{\includegraphics[scale=1.25]{Figures/#1.pdf}}}
\newcommand{\includeCobEq}[1]{\adjustbox{valign = c}{\includegraphics[scale=1]{Figures/#1.pdf}}}

\newcommand{\includeTang}[1]{\adjustbox{valign = c}{\includegraphics[scale=0.1]{Figures/#1.pdf}}}

\newcolumntype{C}[1]{>{\centering\let\newline\\\arraybackslash\hspace{0pt}}m{#1}}

\title{The Functoriality of Odd Khovanov Homology up to Sign and Applications}
\author{Jacob Migdail}
\date{21 March 2024}

\newif\ifthesis
\thesisfalse
\titlespacing*{\section}{0pt}{0.5em}{0pt}
\titlespacing*{\subsection}{0pt}{0.5em}{0pt}
\titlespacing*{\subsubsection}{0pt}{0.5em}{0pt}
\titlespacing*{\paragraph}{0pt}{0.5em}{0pt}
\mathtoolsset{showonlyrefs,showmanualtags}

\begin{document}

\pagenumbering{gobble}
\section*{Abstract}
Odd Khovanov Homology is a homological invariant of knots and links that permits a Bar-Natan category presentation.  In this dissertation, we extend the odd Khovanov bracket to link cobordisms and prove that our construction is functorial up to sign.  We then build an odd Khovanov theory for dotted link cobordisms.  Out of the dotted theory, a module structure on the odd Khovanov homology of a diagram over the exterior algebra of the diagram's coloring group arises.  We finish by using our functoriality result to prove that if $n$ is even or if the knot has even framing, then the odd Khovanov homology of the $n$-cable of a knot admits an action of the Hecke algebra $\mathcal{H}(q^2,n)$ at $q=i$.
\newpage
\begin{center}
\phantom{,}
\vspace{0.125in}

THE FUNCTORIALITY OF ODD KHOVANOV HOMOLOGY UP TO SIGN AND APPLICATIONS\\
\vfill

by\\
Jacob A. Migdail-Smith\\
\vfill

B.S. Rutgers University, 2018\\
M.S. Syracuse University, 2021\\
\vfill

Dissertation\\
Submitted in partial fulfillment of the requirements for\\ the degree of Doctor of Philosophy in \textit{Mathematics}.\\
\vfill
Syracuse University\\
May 2024
\vfill
\end{center}
\newpage
\phantom.\vspace{1in}\begin{center}Copyright\,\textcopyright\,Jacob A. Migdail-Smith 2024\\All Rights Reserved\end{center}
\newpage
\pagenumbering{roman}
\setcounter{page}{4}
\pagenumbering{arabic}

\section*{Abstract}
Odd Khovanov Homology is a homological invariant of knots and links that permits a Bar-Natan category presentation.  In this dissertation, we extend the odd Khovanov bracket to link cobordisms and prove that our construction is functorial up to sign.  We then build an odd Khovanov theory for dotted link cobordisms.  Out of the dotted theory, a module structure on the odd Khovanov homology of a diagram over the exterior algebra of the diagram's coloring group arises.  We finish by using our functoriality result to prove that if $n$ is even or if the knot has even framing, then the odd Khovanov homology of the $n$-cable of a knot admits an action of the Hecke algebra $\mathcal{H}(q^2,n)$ at $q=i$.

\newpage

\section{Introduction}

The Jones polynomial is a powerful invariant of links in its own right, and it has continued to impact knot theory today through its categorifications.  In his seminal paper \cite{Kh1999}, Khovanov categorified the Jones polynomial---with what would come to be known as Khovanov homology---using a construction built implicitly over the truncated symmetric algebra.  Shortly after, in \cite{Ja2004}, Jacobsson proved Khovanov's conjecture that Khovanov homology is a functor of link cobordisms up to sign.  Bar-Natan built upon this in \cite{Bn2005} by providing an alternative formulation of Khovanov homology that maps to a category of complexes of planar diagrams with cobordisms for morphisms. He then showed that his alternative formulation is not just an invariant but is also functorial up to sign.  Khovanov homology was later truly realized as a functor by multiple authors (\cite{Ca2007}, \cite{CMW2009}, \cite{Bl2010}, \cite{BHPW2019}, \cite{Vo2020}, \cite{Sa2021}) who addressed the indeterminacy up to sign in various ways.

The functoriality of Khovanov homology proved very useful, even before the sign issues were resolved.  One of the first results that followed was a purely combinatorial proof of the topological Milnor conjecture by Rasmussen \cite{Ra2004}.  Later, Piccirillo used functoriality to prove that the Conway knot is not slice \cite{Pi2020}.  More recently, Hayden and Sundberg \cite{HS2022} showed that Khovanov homology can distinguish smooth surfaces that are topologically---but not smoothly---ambient isotopic, providing a way to find exotically knotted surfaces. 

Parallel to these developments, an alternative categorification of the Jones polynomial was developed by Ozsv\'ath, Rasmussen, and Szab\'o \cite{ORS2007} using the exterior algebra.  This construction is known as odd Khovanov homology (and when disambiguation is required Khovanov's original construction is referred to as even Khovanov homology).  While the two constructions produce identical complexes over coefficients in $\mathbb{Z}_2$, computations by Shumakovitch \cite{Sh2011} showed that---for rational coefficients---the even and odd theories each have pairs of links they can distinguish that the other cannot.  Odd Khovanov homology differs (in practice) from even Khovanov homology in that, to build a valid differential, one has to first construct what is known as a valid sign assignment.  A Bar-Natan category formulation of odd Khovanov homology was later developed by Putyra as part of a larger general theory that can be specialized to either even or odd Khovanov homology.

There are two overall types of sign assignments: type X and type Y. Lemma 2.4 of \cite{ORS2007} states that both types of assignments produce isomorphic odd Khovanov complexes. Practically this result indicates that odd Khovanov homology is not a branched pair of distinct invariants and that researchers can work with either type X or type Y sign assignments. In \cite{Pu2015}, it is pointed out that the author and Cotton Seed observed a hole in the proof of \cite{ORS2007}, and the author provided an alternative proof.  The author of this dissertation believes that there is also a flaw in the proof provided in \cite{Pu2015}, our first major result is \textbf{Proposition \ref{prop:3:XYinv}} which is a corrected proof of Ozsv\'ath, Rasmussen, and Szab\'o's Lemma 2.4.

The main theorem of this dissertation is that odd Khovanov homology is a functor up to sign.  We will start with Putyra's general construction, specialized to odd Khovanov homology, and then define chain maps induced by each of the elementary link cobordisms.  To show that this construction is functorial, we will begin (as in \cite{Bn2005}) by showing that each of Carter and Saito's \cite{CS1997} fifteen movie moves induce chain maps that are homotopic up to an overall sign.  Unlike in the even setting, the odd setting uses chronological planar cobordisms. So we must check that ambient isotopies of link cobordisms---which correspond to changing the order of distant Morse theoretic saddle points---also induce chain maps homotopic up to an overall sign.

In \cite{Bn2005}, Bar-Natan introduced dotted versions of both his source and target categories that are compatible with even Khovanov homology.  In the odd setting, Putyra introduced a dotted version of his target category \cite{Pu2015} and Manion introduced dots as a portion of the differential.  We build an enhanced dotted linked cobordism category and prove \textbf{Corollary \ref{thr:5:dotFunctor}}, that there is a functor up to sign from this dotted link cobordism category to Putyra's dotted target category.  We can then prove our second most important result, \textbf{Theorem \ref{thm:5:colMod}}, that the odd Khovanov homology of a diagram is a module over the exterior algebra of the coloring group of the link diagram.

The functoriality of even Khovanov homology has also been pivotal in answering questions about the structure of even Khovanov homology itself.   The functoriality of even Khovanov homology was used by Grigsby, Licata, and Wehrli \cite{GLW2017} to show that the even Khovanov homology of the $n$-cable of a knot admits an action of the symmetric group $\mathcal{S}_n$, arising from exchanging adjacent strands of the cable.  In this dissertation, we prove an analogous result, \textbf{Theorem \ref{thm:6:HeckeAction}}, that if $n$ is even or if the knot has even framing, then the odd Khovanov homology of the $n$-cable of a knot admits an action of the Hecke algebra of type $\mathcal{A}_{n-1}$ at $q=i$.

We plan to study the following applications of our primary results in future publications.  First, we expect to prove an analog of Levine and Zemke's result in \cite{LZ2019} that ribbon disks induce injective maps from the even Khovanov homology of a knot to that of the unknot.  In the odd setting, ribbon concordances induce injective maps when using rational coefficients. If one is using integer coefficients, then ribbon concordances induce almost injective maps. Meaning, that if a certain geometrically defined odd number $d$ were invertible, then the induced map would be injective.  We also expect to prove that our Hecke algebra action extends to a functor on the odd Temperley–Lieb supercategory at $\delta=0$.  This would allow us to compute invariants of twist-spun knots among other 2-knots, for which we can prove our conjecture that the odd Khovanov Jacobsson number of a dotted 2-knot corresponds to the order of the first homology of the branched double cover of the 4-sphere branched along the given 2-knot.

The remainder of this dissertation is organized in the following manner.  
In \textbf{Section 2} we recall the initial link cobordism category and the terminal planar cobordism category relevant to defining the odd Khovanov homology functor.  
In \textbf{Section 3} we recall the construction of odd Khovanov homology as an invariant, as well as define the chain maps that extend odd Khovanov homology from an invariant into a functor.  
In \textbf{Section 4} we prove the main theorem of this dissertation, the functoriality of odd Khovanov homology up to sign.  
In \textbf{Section 5} we recall the dotted planar cobordism setting, construct a category of dotted link cobordisms that is compatible with the odd Khovanov homology functor, and prove a structural result about odd Khovanov homology.  
In \textbf{Section 6} we prove that the odd Khovanov homology of the $n$-cable of a knot admits the aforementioned Hecke algebra action. 
In \textbf{Appendix A} all the relations and conventions in the planar cobordism category are enumerated together.  Additionally, the commutativity and associativity relations are presented as they would appear in the odd Khovanov differential.  If the reader is not familiar with odd Khovanov homology they are encouraged to use this section as a quick reference guide when reading proofs.  
In \textbf{Appendix B} we provide the work for specific variations of the four-tube relation that otherwise could have been left to the reader.

\section{Cobordisms}
Before we define the functor odd Khovanov homology, we must pin down the initial category that odd Khovanov homology will map from, and the terminal category it will map to.  Our functor will map from the category of link cobordisms to a category comprised of chain complexes of $\mathbb{Z}$-linear combinations of chronological-oriented-planar cobordisms.  We will begin by defining cobordisms in general, then build the particular categories involved, and finally explain what equivalence means in each category.  Our focus will be on building useful definitions of the categories that are sufficient to construct our functor.  Many of the motivations and explanations for these particular settings will be omitted (the reader is encouraged to reference \cite{Pu2015} for details about the particular planar cobordism setting as well as \cite{Bn2005} and \cite{CS1997} for details about the link cobordism setting).
\subsection{Planar Cobordisms}
The objects of the planar cobordism category are planar diagrams, and the morphisms are planar cobordisms.
\begin{definition}
	A \defw{planar diagram} is a disjoint union of (at most) finitely many smoothly embedded unoriented circles in $\RR^2$.
\end{definition}
\begin{definition}
    A \defw{planar diagram cobordism} or \defw{planar cobordism} is a smooth unoriented compact surface with boundary embedded in $\RR^2\times [0,1]$ such that its boundary lies in $\RR^2\times\{0,1\}$.
\end{definition}
The embedding of a cobordism naturally endows it with a height function by projection onto $[0,1]$.  Planar cobordisms are considered to be equivalent if they are ambient isotopic rel boundary.
\begin{definition}
    The \defw{i-frame} is the embedding of the height $i$ part of the cobordism in the plane.  
\end{definition}
Without loss of generality, we can impose the condition that the projection $\RR^2\times\RR\rightarrow\RR$ restricts to a Morse function on the cobordism, and that all critical points occur at different heights. Concretely, this means that all but finitely many frames of the cobordism are planar diagrams, and that frames that fail to be planar diagrams fail as a result of containing---along with circles---a single embedded point or figure-eight.
    \begin{definition}
    Frames that fail the regularity condition are called \defw{critical frames}.\footnote{If the height function were thought of as a Morse function, critical frames would be the heights at which critical points exist.}
    \end{definition}
\subsubsection{Movies}
    We can think of planar cobordisms as an evolution from the 0 frame to the 1 frame.  In this context, the height function is better regarded as a function of time, and the frames (like the frames of a movie) provide us with an alternative presentation of cobordisms.
\begin{definition}
    A \defw{planar movie} is a sequence of selected frames from a planar cobordism in which chosen frames are separated by at most one critical frame. Furthermore, each critical frame contains either a single embedded point or a single embedded figure-eight.
\end{definition}
When drawing movies, we will use the following conventions:
\begin{convention}
	Movies may depict tangles or only a local image of a diagram.  In such cases, the reader should assume all other diagram components are fixed from one frame of a movie to the next.
\end{convention}
\begin{convention}
	For cobordisms, the front is the surface \enquote{closest} to the reader.  Similarly, for movies, the front is the bottom of the frame, as in Figure \ref{fig:movieDirections}.
\end{convention}
\begin{figure}[H]
\centering
    \includeCobEq{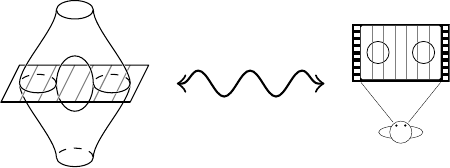}
    \caption{On the left is a cobordism with a distinguished frame.  On the right is you---the reader---looking at the cobordism on the left, now presented as a movie featuring the same distinguished frame}
    \label{fig:movieDirections}
\end{figure}
\begin{convention}
	Movies and cobordisms should be read from the bottom to the top of the page.  The diagram at the bottom is initial with respect to time, and the diagram at the top is final or terminal.
\end{convention}
\subsubsection{Elementary Planar Cobordisms}
    For planar movies, there are a few basic cobordisms of particular interest.  In the following cobordism definitions, assume that all circles not mentioned are assigned an identity cobordism or are fixed.
\begin{definition}
	A \defw{birth} is a cobordism with a critical frame containing an embedded point that expands into a circle in the following frames.
\end{definition}
\begin{definition}
	A \defw{death} is a cobordism with a circle that contracts to a point in the critical frame, then disappears in the frames afterward.
\end{definition}
\begin{figure}[H]
\begin{minipage}{0.5\textwidth}
    \centering
    \includeCobEq{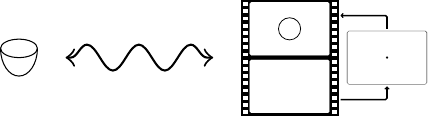}
    \caption{A birth cobordism}
    \label{fig:birth}
\end{minipage}%
\begin{minipage}{0.5\textwidth}
    \centering
    \includeCobEq{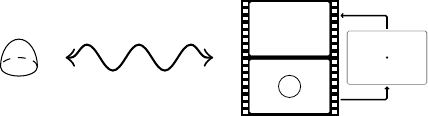}
    \caption{A death cobordism}
    \label{fig:death}
\end{minipage}
\end{figure}
\begin{definition}
	A \defw{merge} is a cobordism where two circles extend toward one another until they meet at a single point where they form a figure eight in the critical frame.  Then, the interiors of the circles push through the point to merge leaving a single circle.
\end{definition}

\begin{figure}[H]
    \centering
    \captionsetup{width=6in}
    \includeCobEq{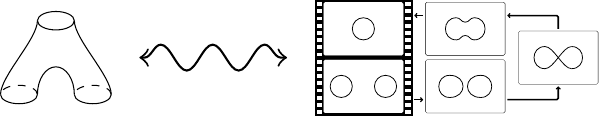}
    \caption{A merge cobordism}
    \label{fig:merge}
\end{figure}
\begin{definition}
	A \defw{split} is a cobordism where the edges of a single circle pinch towards one another until they meet at a single point, where they form a figure eight in the critical frame.  Then, the two curves break away from one another forming two circles.
\end{definition}
\begin{figure}[H]
    \centering
    \includeCobEq{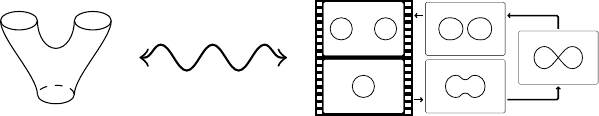}
    \caption{A split cobordism}
    \label{fig:split}
\end{figure}
It should be clear that the death movie is the birth movie played in reverse, and that the split movie is the merge movie played in reverse.  Births and splits increment the number of circles, while deaths and merges decrement the number of circles. All other movies maintain the number of circles in a diagram.  
\begin{definition}
    A \defw{planar ambient isotopy cobordism} is a cobordism that contains no critical frames.
\end{definition}
Cobordisms which are planar ambient isotopies arise from moving circles or deforming the shape of a circle.  An example of a planar ambient isotopy appears in Figure \ref{fig:permutationCobordism}.
\begin{figure}[H]
    \centering
    \includeCobEq{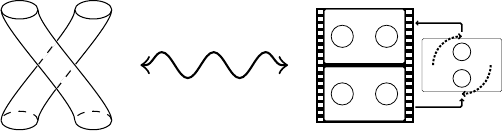}
    \caption{A cobordism \enquote{braiding} two circles}
    \label{fig:permutationCobordism}
\end{figure}
Planar diagrams form a category with (ambient isotopy classes of) planar cobordisms as morphisms. The composition in this category is defined as follows: if the initial diagram of a cobordism $A$ is the same as the terminal diagram of another cobordism $B$, then the composition $A\circ B$ is obtained by stacking the cobordism $A$ on top of $B$ or, in the movie setting, splicing the movie of $A$ at the end of the movie of $B$ (see Figure~\ref{fig:cobordismComposition}).  In situations where it will not cause ambiguity, we will use multiplicative notation and the composition symbol \enquote{$\circ$} will be omitted.
\begin{figure}[H]
    \centering
    \captionsetup{width=4.25in}
    \includeCobEq{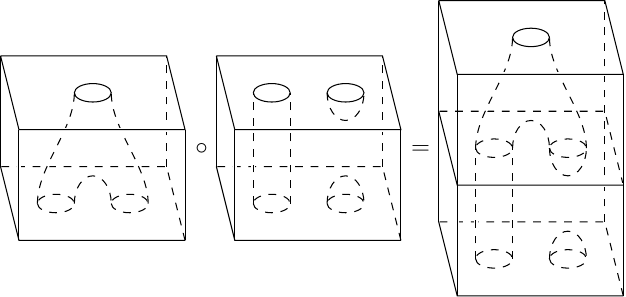}
    \caption{The composition of cobordisms $A$ and $B$, the left and right operands respectively, yields the cobordism $A\circ B$}
    \label{fig:cobordismComposition}
\end{figure}

All planar cobordisms can be decomposed into compositions of the four elementary planar cobordisms and planar ambient isotopy cobordisms.
\subsection{Link Cobordisms}
Link cobordisms are similar to planar cobordisms in that they consist of surfaces with one-dimensional boundaries, except that they are embedded in four-dimensional space rather than three-dimensional space.  The objects of a link cobordism category are links, and the morphisms are link cobordisms.

\begin{definition}
    A \defw{link} is a disjoint union of (at most) finitely many smoothly embedded circles in $\RR^3$.
\end{definition}

A link consisting of only one embedded circle is called a knot.  Throughout this dissertation, we will assume that links are oriented, in the following sense:

\begin{definition}
    An \defw{orientation} on a link is a choice of an orientation on each of the copies of $S^1$ that compose the link.
\end{definition}
We can impose some regularity conditions on links. Namely, that projection along the $z$-axis onto the plane $\RR^2$ is injective at all but finitely many places, and at most two distinct points of the link project down to any point in $\RR^2$.  Under these conditions, we can talk about link diagrams rather than the actual embeddings.
\begin{definition}
    A \defw{link diagram} is an immersion of a link $L\subset\RR^3$ into the plane $\RR^2$ created via projection along the $z$-axis, producing a 4-valent planar graph with a distinguished pair of edges at each vertex corresponding with the strand that was on top prior to projection.
\end{definition}
Links are considered equivalent if they are ambient isotopic. In turn, it is known \cite{Kr1927} that link diagrams represent equivalent links if and only if they differ by a finite sequence of planar ambient isotopies and the three Reidemeister moves described below.

Oriented links allow us to easily assign a direction to crossings (although it can still be defined for unoriented links using a variation on the right-hand rule).  This convention is shown in Figure \ref{fig:2:crossingDirection}

\begin{figure}[H]
    \centering
    \includeFig{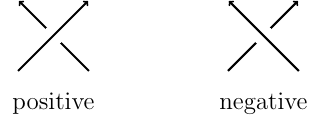}
    \caption{Positive and negative crossings for directed links}
    \label{fig:2:crossingDirection}
\end{figure}

\begin{definition}
    A \defw{link cobordism} is a compact oriented surface smoothly embedded in $\RR^3\times[0,1]$ such that its boundary lies in $\RR^3\times\{0,1\}$.
\end{definition}

In particular, a link cobordism between two links $L_0$ and $L_1$ is a link cobordism $S$ with
\begin{equation}
    \partial S=(-L_0\times\{0\})\cup(L_1\times\{1\})
\end{equation}
where $-L_0$ is the link $L_0$ with reversed orientation.

The notion of a height function and frames introduced for planar cobordisms naturally apply in the link cobordism setting as well.  This allows us to apply some regularity conditions and define movies for link cobordisms as well.  We require all but finitely many of the frames of our link cobordisms to be links, more specifically, links that satisfy the regularity conditions such that their projections are planar diagrams.  

\subsubsection{Elementary Link Cobordisms}

In terms of movies, this gives us six elementary link cobordisms. 

\begin{itemize}
    \item \textbf{Birth}: the creation of a new component
    \item \textbf{Death}: the deletion of an existing component
    \item \textbf{Saddle}: the merging or splitting of two link components
    \item \textbf{Reidemeister I}: the \enquote{twist} move on a strand
    \item \textbf{Reidemeister II}: the \enquote{poke} move on two strands
    \item \textbf{Reidemeister III}: the \enquote{slide} move between a strand and a crossing
\end{itemize}

As in the planar setting, we have the birth and death cobordisms where an unknotted circle either expands from or contracts to a point.  We can also have saddle cobordisms, and, in the oriented setting, these cobordisms always either merge two link components into one or split a link component into two. On the level of link diagrams, a saddle may either merge two components or split a component into two, or it may happen within a component. Here, the word component refers to the underlying four-valent graph.  As in the planar setting, we also have planar ambient isotopy cobordisms. These are the cobordisms that contain none of the six elementary link cobordisms enumerated above.

In addition to the cobordisms that transferred from the planar setting, we also have cobordisms that arise from each of the three Reidemeister moves.

\begin{figure}[H]
    \centering
    \includeCobEq{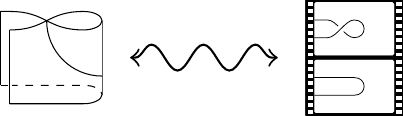}
    \caption{The cobordism arising from a Reidemeister I move}
    \label{fig:RI}
\end{figure}

\begin{figure}[H]
    \centering
    \includeCobEq{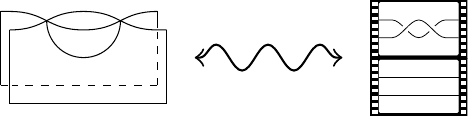}
    \caption{The cobordism arising from a Reidemeister II move}
    \label{fig:RII}
\end{figure}

\begin{figure}[H]
    \centering
    \includeCobEq{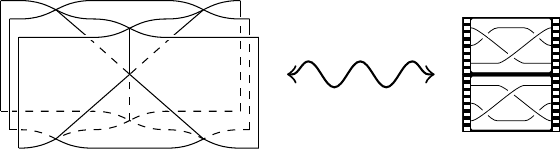}
    \caption{The cobordism arising from a Reidemeister III move}
    \label{fig:RIII}
\end{figure}

\subsection{The Source Category $\cobIV$}

We now have the pieces in place to define the source category of our odd Khovanov functor: the link cobordism category $\cobIV$.  The objects of $\cobIV$ are oriented link diagrams, and the morphisms are link cobordisms compliant with the regularity conditions.  Two such cobordisms are equivalent as morphisms in $\cobIV$ if there exists a smooth ambient isotopy between them.  As with links, this is a difficult notion of equivalence to work with. Carter and Saito (\cite{CS1991},\cite{CS1993}) translated this equivalence into the movie setting.  They developed a set of moves---which function like Reidemeister's moves---but for link cobordisms.

\subsubsection{Movie Equivalence in the Source Category}

In terms of movies, replacing a cobordism with an ambient isotopic one is a splicing operation, reminiscent of editing a movie in the cutting room.

\begin{definition}
    A \defw{movie move} is a pair of link cobordism movies with the condition that if one of the sides of the movie move is present in a movie, the operation of cutting out that side and splicing in the other side of the movie move in the first side's place produces an equivalent movie.
\end{definition}

The following is a type of movie move that will be of particular interest in this dissertation.

\begin{definition}
    A \defw{chronological movie move} is a movie move generated by changing the order of distant critical points in a cobordism.
\end{definition}

Carter and Saito's main result can be summarized in the following manner.
\begin{theorem}
Two movies represent smoothly ambient isotopic link cobordisms if and only if they differ by a finite sequence of chronological movie moves, planar ambient isotopy cobordisms\footnote{In \cite{CS1997} planar ambient isotopy cobordisms are implicit, we explicitly state these as we must also consider if our functor respects these ambient isotopies.}, or the fifteen particular movie moves\footnote{Our collection of movie moves is that in \cite{Bn2005}.  These movie moves differ from \cite{CS1997} in movie moves 6, 7, 8, and 10 where the strip is composed of one side of the move in \cite{CS1997} played forward, and the other side played in reverse.  Showing that our construction respects this collection of movie moves is sufficient to prove functoriality.} depicted in Figures \ref{fig:4:TIMM}, \ref{fig:4:TIIMM}, \ref{fig:4:MM11}, \ref{fig:4:MM12}, \ref{fig:4:MM13}, \ref{fig:4:MM14}, and \ref{fig:4:MM15}.
\end{theorem}

For Carter and Saito's fifteen particular movie moves, the first ten movie moves are \enquote{do nothing moves} in that the given strip is equivalent to the identity cobordism from the first to the final frames.  The final five moves are non-reversible, and must be considered as separate movie moves in the forward and backward directions.  Like Reidemeister moves, where there are multiple versions of some of the moves associated with changing particular crossings, the movie moves also have multiple variants when applicable. In \cite{Bn2005}, the fifteen movie moves are divided into three subgroups, each with five members titled type I, type II, and type III movie moves. We will follow this convention.  The type I movie moves correspond to cobordisms generated by doing and undoing the same Reidemeister move. The type III movie moves are the only ones that involve births or deaths. 

\subsection{The Target Category}

In order to define the target category for our functor, we will first have to introduce some additional structures which can be placed on a planar cobordism category.  Our target category is the category constructed in \cite{Pu2015} specialized to the odd setting.\footnote{In \cite{Pu2015} the category for planar diagrams is general purpose and has variables $X$, $Y$, and $Z$, which can be specified to yield the requisite category for either even or odd Khovanov homology.  We are using Putyra's category with $X$ and $Z$ set to 1 and $Y$ set to $-1$ for the odd Khovanov bracket.}

\subsubsection{Oriented Planar Cobordisms}
	We can build a refined collection of planar cobordisms by placing orientations of deaths and saddle cobordisms.
	\begin{definition}
		An \defw{orientation} on a planar death cobordism is a choice of orientation on the manifold in a neighborhood of the death.
	\end{definition}
	A death can be oriented clockwise or counterclockwise.
    \begin{figure}[H]
    \centering
    \captionsetup{width=4in}
    \includeCob{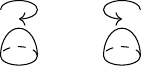}
    \caption{Clockwise and counterclockwise oriented death cobordisms on the left and right respectively}
    \label{fig:orientedDeath}
    \end{figure}
    
Consider a subset of a neighborhood of a saddle point, the subset being those portions of the neighborhood at or below the saddle point. The set resembles a bowtie with two regions separated by a single point.
    \begin{definition}
		An \defw{orientation} on a saddle cobordism is a choice of path from one of its lower regions to the other passing through the saddle point.
	\end{definition}
    \begin{figure}[H]
    \centering
    \captionsetup{width=5.5in}
    \includeCob{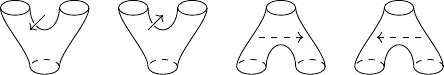}
    \caption{From left to right: the forward-oriented split, the backward-oriented split, the right-oriented merge, and the left-oriented merge}
    \label{fig:orientedSplit}
    \end{figure}
In more intrinsic terms, orientations of deaths and saddles can be viewed as orientations of the descending manifolds of the relevant critical points of the height function. While merges have orientations, they are rarely relevant in our setting and will often be omitted.  To simplify pictures of cobordisms, we will often also omit the orientations of the splits and deaths. In such cases, we will use the following convention:
    \begin{convention}\label{conv:2:orientations}
        The default orientations are clockwise on deaths, forward on splits, and right on merges.
    \end{convention}
\subsubsection{Chronological Cobordisms}
Planar cobordisms gain additional structure in the form of a chronology, which keeps track of the order in which events occur. This is a structure that would normally be ignored under the usual notion of cobordism equivalence.
	\begin{definition}
		The order of critical points in a planar cobordism is called a \defw{chronology}.
	\end{definition}
	\begin{definition}
A cobordism category is \defw{chronological} if equivalence does not include ambient isotopies that change the chronology.
	\end{definition}
\begin{figure}[H]
    \centering
    \includeCob{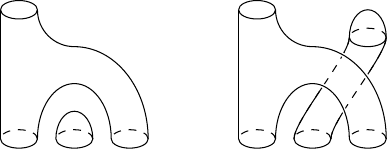}
    \caption{Ambient isotopic cobordisms that differ by a change in chronology}
    \label{fig:chronologyChange}
\end{figure}
{\noindent}In the setting considered in the remainder of this dissertation, the most a change in chronology may incur is multiplication by -1.\footnote{In Putyra's general cobordism category from \cite{Pu2015}, changes in chronology can incur multiplication by $X$, $Y$, $Z$, or $Z^{-1}$.}
\subsubsection{The Odd Khovanov Homology Planar Cobordism Category}
We now have all the machinery in place to define the particular planar cobordism category $\cobIII$ that we will use to build the target category of our functor.\footnote{While $\cobIII$ is defined in \cite{Pu2015}, we are going to use notation more consistent with \cite{Bn2005}.  It is named in this way as it is a category of chronological cobordisms in three dimensions, subject to the specific set of local relations for odd Khovanov homology.}  The objects of this category are planar diagrams, and the morphisms are formal $\ZZ$-linear combinations of planar cobordisms considered up to the relations described below and up to planar isotopies preserving the chronology.  Here we assume that the critical points in these cobordisms occur at different heights, and that all deaths and saddles are equipped with orientations.  The composition of two morphisms is defined by extending the vertical stacking operation from Figure \ref{fig:cobordismComposition} bilinearly to formal linear combinations of cobordisms.  In terms of generators, $\cobIII$ is a preadditive category generated by planar ambient isotopy cobordisms, birth cobordisms, clockwise death cobordisms, counterclockwise death cobordisms, merge cobordisms, and oriented split cobordisms.
\begin{figure}[H]
    \centering
    \includeCob{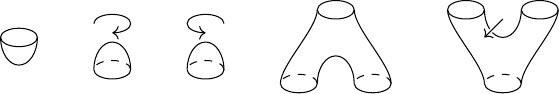}
    \caption{The cobordisms that generate $\cobIII$}
    \label{fig:generatingCobordisms}
\end{figure}

\begin{remark}
    The cobordisms in this dissertation are always embedded.  In \cite{Pu2015} the author works with non-embedded cobordisms except for a brief period where it is strictly necessary that he work with embedded cobordisms.  Working with non-embedded cobordisms simplifies the presentation of relations, in that one does not need to concern themselves with all the ways circles could be embedded in relation to one another.  For the following section, and in Appendix \ref{apdx:asscomm}, the non-embedded versions of relations are presented but should be interpreted as providing the relation to all of the ways circles could be embedded around one another.  For example, a split could be drawn such that the two resulting circles are either lying next to one another in the plane, or concentric.
\end{remark}

While merges have orientations---which play a role in Putyra's more general setting---we only consider them up to the relation that reversing the orientation of a merge leaves the cobordism unchanged.  In contrast, reversing the orientation on a death or a split incurs multiplication by $-1$, yielding the following relations.

\vspace{1em}{\noindent}\begin{minipage}{0.5\textwidth}
    \begin{equation}
        \label{eq:deathSignChange}
        \includeCobEq{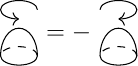}
    \end{equation}
\end{minipage}%
\begin{minipage}{0.5\textwidth}
    \begin{equation}
        \label{eq:splitSignChange}
        \includeCobEq{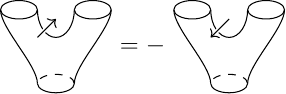}
    \end{equation}
\end{minipage}\vspace{1em}

We also impose the relations in equation \eqref{eq:annihilationCreation} to cancel births and deaths composed with merges and splits.  If we think of these cobordisms (contracted onto some 1-skeleton) from a graph theoretic perspective, these relations would amount to \enquote{pruning} or removing leaves from the graph.
\begin{equation} 
    \label{eq:annihilationCreation}
    \includeCobEq{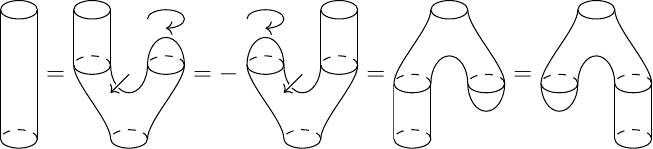}
\end{equation}
For the purposes of tracking chronology changes, our basic cobordisms can be classified as either even or odd.\footnote{These designations arise from the exterior algebra.}  Births and merges are even, while deaths and splits are odd.  If a change in chronology exchanges the order of an even cobordism with any other cobordism, both cobordisms are equal.  When two odd cobordisms are chronologically rearranged, multiplication by $-1$ is incurred.  From this convention, there emerges a pair of exceptional arrangements which, a priori, are neither commuting nor anticommuting as they are annihilated by 2.

\begin{figure}[H]
    \centering
    \includeFig{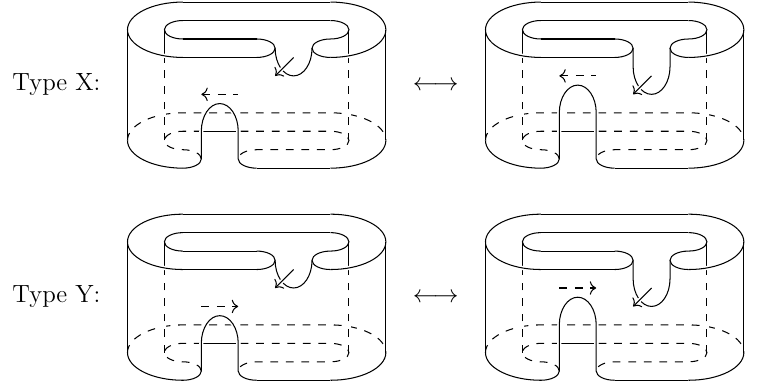}
    \caption{Cobordisms of type X and type Y configurations}
    \label{fig:xytype}
\end{figure}

In order to define an odd Khovanov homology theory, one must artificially decide that one of these configurations commutes and that the other anticommutes.  These two different odd Khovanov homology theories are called type X or type Y based on which pair of cobordisms in Figure \ref{fig:xytype} anticommutes.  It was first observed in \cite{ORS2007} that either choice produces the same final invariant, although their proof was incorrect.  Putyra attempted in \cite{Pu2015} to provide a corrected proof, but his proof also appears to be incorrect.  This is Proposition \ref{prop:3:XYinv}, which we will prove in Section 3.  We will use type Y sign assignments following Putyra in \cite{Pu2015}.  

For a full list of commutativity and associativity relations in our specific target category refer to Appendix \ref{apdx:asscomm}. 

\begin{remark}
     When we refer to relations being associative and commutative, we mean in terms of the odd Frobenius algebra.  When looking at actual cobordisms, almost all of the relations are commutativity relations in a sense, as they relate to some commutation in time of saddles or other elementary cobordisms.
\end{remark}

One particular consequence of the relations described above is the double handle\footnote{In the conventional non-mathematical sense} (2H) relation.  If a cobordism has genus one or higher, and one of its holes can be isolated as a split followed by an immediate merge, that whole region can be rotated one-half turn and the orientation on the split can be reversed to get back to the original cobordism.  This implies that the original cobordism is its own additive inverse, or that multiplying the cobordism by two kills the cobordism.
\begin{equation}
    \includeCobEq{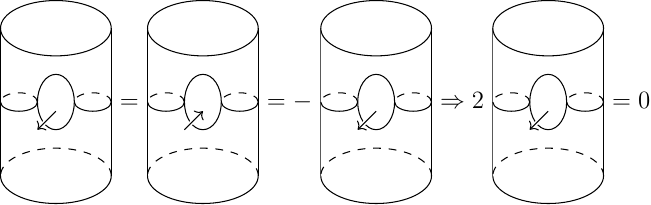}
\end{equation}
The category $\cobIII$ is subject to some additional relations. If a cobordism contains either a sphere or a torus, the cobordism is 0. These are known as the sphere (S) and torus (T) relations, respectively.
\begin{equation}\label{eq:2:SandTrelations}
    \includeCobEq{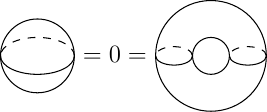}
\end{equation}
The final relation is called the four tube (4Tu) relation and it is given by the following equation:
\begin{equation}
    \includeCobEq{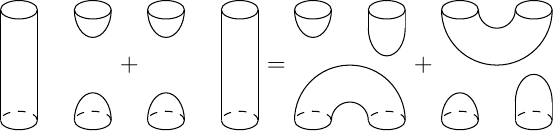}
\end{equation}
Some immediate consequences of the (4Tu) relation that will be used later can be found in Appendix \ref{apdx:4tuvariants}, along with the computations of these variants.

To summarize, $\cobIII$ is a category of chronological planar cobordisms with oriented descending manifolds generated by the following cobordisms:
\begin{multicols}{2}
\begin{itemize}
    \item Birth cobordisms
    \item Clockwise death cobordisms
    \item Counterclockwise death cobordisms
    \item Merge cobordisms
    \item Oriented split cobordisms
    \item Planar ambient isotopy cobordisms
\end{itemize}
\end{multicols}
Modded out by the following local relations:
\begin{multicols}{2}
\begin{itemize}
    \item Orientation reversing relations
    \item \enquote{Pruning} relations
    \item Chronological change relations
    \item Sphere relation
    \item Torus relation
    \item Four tube relation
    \item Associativity relations
    \item Ambient isotopies of planar ambient isotopic cobordisms
\end{itemize}
\end{multicols}

\subsubsection{The Final Target Category $\cobIIIKM$}
Our final goal is to assign a homological algebra type object to each link cobordism.  Objects in the category $\cobIII$ are those we will want to take chain complexes of, but before we do so we must guarantee that the category we are working with is additive.  To this end, we replace $\cobIII$ with $\cobIIIM$, where $\text{Mat}(\mathcal{A})$ denotes the additive closure of the preadditive category $\mathcal{A}$.  Its objects are finite (possibly empty) sequences of objects in $\mathcal{A}$, and its morphisms are given by matrices of morphisms in $\mathcal{A}$.  The composition is modeled on ordinary matrix multiplication.

We can now consider the category $\cobIIIKM$ of bounded chain complexes and chain maps in $\cobIIIM$.  Let $\cobIIIKbM$ denote the corresponding homotopy category in which homotopic chain maps are identified.  It is equivalent to discuss functoriality up to sign when the target category is the bounded homotopy category $\cobIIIKbM$, and functoriality when the target category is the slightly modified category $\cobIIIKbMpm$ in which homotopy classes of chain maps are identified with their negatives.  Note that the latter category is no longer preadditive.

\section{Odd Khovanov Homology}
We will now construct Putyra's version of odd Khovanov homology.  Then we will review portions of the proof of invariance, as they will be needed to extend odd Khovanov homology to link cobordisms and to prove that it is functorial.  We will end this section by defining the chain maps that are assigned to each of the elementary cobordisms in $\cobIV$.
\subsection{Links with Oriented Crossings}
Odd Khovanov homology is not computed from a simple link diagram, but from a link diagram that has been enhanced with orientations on the crossings.  Each crossing has two possible orientations, which are displayed in Figure \ref{fig:orientedCrossing}.

\begin{figure}[H]
    \centering
    \includeFig{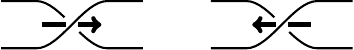}
    \caption{The two possible orientations on a crossing}
    \label{fig:orientedCrossing}
\end{figure}

Note that the arrows specifying the crossing orientation are chosen so that they connect the regions that lie to the left if one approaches the crossing along the overstrand from either side.

\subsection{Crossing Resolution}

A link diagram with $n$ crossings gives rise to $2^n$ planar diagrams corresponding to all possible combinations of replacing each crossing with the vertical or horizontal resolution, where the terms \enquote{vertical} and \enquote{horizontal} refer to the pictures in Figure \ref{fig:orientedCrossing}.  The rule for resolving a crossing is shown in Figure \ref{fig:resolutions}.  If the crossings are labeled from 1 to $n$, then each diagram, $D_\alpha$, can be assigned a binary label $\alpha\in\{0,1\}^n$ with a 0 in the $i$\ordth place if the $i$\ordth crossing was replaced with the vertical resolution, and a 1 if it was replaced with the horizontal resolution.  

\begin{figure}[H]
    \centering
    \captionsetup{width=\textwidth}
    \includeFig{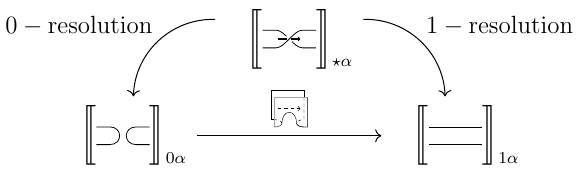}
     \caption{The vertical and horizontal resolution of a crossing}
    \label{fig:resolutions}
\end{figure}

\begin{convention}
    To denote that an index is in superposition we use a $\star$ at that index.
\end{convention}
We assign a cobordism $d_{{\dots}\star{\dots}}:D_{{\dots}0{\dots}}\rightarrow D_{{\dots}1{\dots}}$ to each pair of planar diagrams that differ by a single resolution.  The cobordism $d_{{\dots}\star{\dots}}$ is a morphism in $\cobIII$ whose initial frame is the vertical resolution and whose terminal frame is the horizontal resolution with a saddle between (as depicted in Figure \ref{fig:resolutions}).  The saddle inherits its orientation from the resolved crossing's orientation.

\begin{convention}
    Greek characters will be used to denote binary strings and the concatenation of Greek characters should be read as the concatenation of the binary strings.  Furthermore, $\zeta$ will be used to denote the string consisting of only zeros.
\end{convention}

\begin{definition}
    The \defw{degree} of a resolution is the number of horizontal resolutions in the resulting diagram.
\end{definition}

The degree of a resolution is thus the number of ones in the diagram's binary label, and it is denoted $\deg(\alpha)$.

\subsubsection{The Odd Khovanov Bracket}

\begin{definition}
    The \defw{cube of resolutions} of a link diagram $D$ with $n$ crossings is an $n$ dimensional cube with the $2^n$ planar diagrams $D_\alpha$ for $\alpha\in\{0,1\}^n$ as its vertices and the corresponding cobordisms as its edges.
\end{definition}

Each 2-dimensional face of the cube of resolutions corresponds to two vertices living in the superposition of their resolution, while all others are resolved.  In order to define odd Khovanov homology we need a chain complex. Therefore, all faces must anticommute so that the differential squares to zero.  Each face falls into one of a few basic types and has an inherited sign $\sigma_{i,j}$ of either 1 or $-1$ corresponding with the indices of the non-resolved crossings in $D_{\dots{\star}\dots{\star}\dots}$. The inherent signs of the basic faces are shown in the second section \enquote{Commutivity of Faces in the Odd Khovanov Cube} of Appendix \ref{apdx:asscomm}.\footnote{An unfortunate coincidence arising from our notation is that the type Y configuration is the tenth face type and has been labeled with the roman numeral \enquote{x}, not to be confused with the type X configuration which has received the numeral \enquote{vi}.}  We need to make an assignment of a sign $\epsilon_{\dots{\star}\dots}$ to each edge such that
\begin{equation}\label{eq:3:validFace}
    \displaystyle\sigma_{i,j}\left[\prod_{\{\iota_i,\iota_j\}\in\{\{0,{\star}\},\{1,{\star}\},\{{\star},0\},\{{\star},1\}\}}\epsilon_{\dots{\iota_i}\dots{\iota_j}\dots}\right] = -1
\end{equation}%
Such a sign assignment can always be made (\cite{ORS2007},\cite{Pu2015}), but it is not unique.  For example, one can always get a new sign assignment by swapping the sign of each edge, and there are in general $2^{2^n-1}$ sign assignments satisfying \eqref{eq:3:validFace}.

\begin{definition}
    The \defw{odd Khovanov cube} is the cube of resolutions with each edge multiplied by its assigned sign.\footnote{The odd Khovanov cube normally refers to the cube of resolutions in \cite{ORS2007} which is a cube with truncated exterior algebras at each vertex.  In this dissertation, we bypass this object as we pass to the odd Khovanov bracket before applying the odd Khovanov TQFT.  If we instead passed our cube of resolutions through the odd Khovanov TQFT, we would arrive at the canonical cube of resolutions from \cite{ORS2007}.}
\end{definition}

\begin{definition}
    The \defw{odd Khovanov bracket} is the grading-shifted flattening of the odd Khovanov cube created by taking the direct sum of all vertices with the same degree.
\end{definition}

The odd Khovanov bracket is a chain complex, as all maps in the Khovanov cube raise the degree by one, and before collapsing the cube we ensured that all faces anticommute.  The grading shift in the odd Khovanov bracket is determined by the number of negative crossings in the link diagram. Note that we interpret the odd Khovanov bracket as an object of the category $\cobIIIKM$.

\begin{figure}[H]
    \centering
    \includeFig{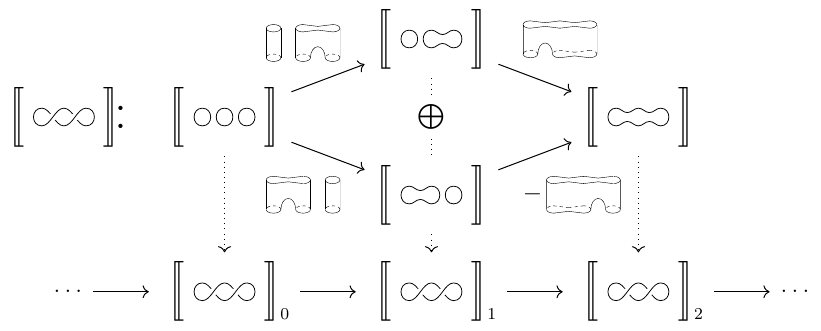}
     \caption{The odd Khovanov bracket of a twice twisted unknot}
    \label{fig:unknot}
\end{figure}

\begin{convention}
    The unsigned cobordisms in the odd Khovanov cube will be denoted with $d$ while the maps in the odd Khovanov bracket are denoted with $\partial$.
\end{convention}
\begin{convention}
    For a link diagram $L$ the odd Khovanov bracket is denoted $\llbracket L\rrbracket$, additionally, the specific $\alpha$\ordth vertex is denoted by $\llbracket L\rrbracket_\alpha$ so that $\llbracket L\rrbracket_\alpha=D_\alpha$.
\end{convention}

\subsection{Odd Khovanov Homology}\label{sec:3:okh}

In \cite{Pu2015}, Putyra showed that the homotopy type of the odd Khovanov bracket as defined above is an invariant of links.  In order to define odd Khovanov homology we must first define a TQFT-type functor $\ff$ that takes our cobordisms to an abelian category.  Putyra defined a very similar functor in \cite{Pu2015} to that in \cite{ORS2007} only using his cobordism language. As both functors are essentially the same, and produce identical invariants, we will use the original functor from \cite{ORS2007} translated into our setting.  Let $R$ be a resolution of a link and $V(R)$ be the free abelian group generated by the circles in the resolution.  We will define $\mathcal{F}(R)$ as the exterior algebra
\begin{equation}
    \ff(R)\coloneqq\Lambda^*V(R)    
\end{equation}
For a cobordism, we will call the initial planar diagram $R$ and the terminal planar diagram $R'$.  If our cobordism is a merge saddle, it takes two circles $c_0$ and $c_1$ in $R$ to a circle $c'$ in $R'$, and if it is a split it takes a circle $c$ in $R$ to circles $c_0'$ and $c_1'$ in $R'$.  Our functor is then defined as follows on merge and split saddles:
\begin{equation}
\begin{array}{cccc}
    \text{Merge} & \{c_0,c_1\} & \rightarrow &\{c'\} \\
    \text{Split} & \{c\} &\rightarrow &\{c_0',c_1'\}
\end{array}
\stackrel{\ff}{\rightarrow}
\begin{array}{c}
    \Lambda^*V(R)\rightarrow \Lambda^*V(R)/(c_0-c_1)\cong \Lambda^*V(R') \\
    \Lambda^*V(R)\rightarrow (c_0'-c_1')\wedge\Lambda^*V(R') \\
\end{array}
\end{equation}

\begin{figure}
    \centering
    \includeFig{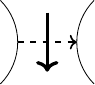} 
    \caption{The arcs after a split cobordism with a downward orientation.  The dashed arrow points from the circle $c_0'$ to $c_1'$}
    \label{fig:circlegenerator}
\end{figure}

Here the map assigned to a split is defined by identifying $\Lambda^*V(R)$ with
\begin{equation}
    \Lambda^*(R')/(c'_0-c'_1)\cong(c_0'-c_1')\wedge\Lambda^*V(R)
\end{equation}
where the isomorphism is given by taking the wedge product from the left with $c_0'-c_1'$. In this definition, it is important that the circles $c_0'$ and $c_1'$ have different roles, and which circle we choose for each role is determined by the orientation on the saddle.

Explicitly, we rotate the orientation on the saddle by ninety degrees clockwise to arrive at an arrow pointing from $c_0'$ to $c_1'$, as in Figure \ref{fig:circlegenerator}.

The map assigned to a birth is induced by the inclusion $V(R)\rightarrow V(R\sqcup\bigcirc)$, and the map assigned to a clockwise death is given by contraction from the left with the dual of the component that gets annihilated.

\begin{definition}
    The \defw{odd Khovanov complex} is the chain complex produced by applying the odd Khovanov TQFT to the odd Khovanov bracket.
\end{definition}
\begin{definition}
    \defw{Odd Khovanov homology} is the homology of the odd Khovanov complex.
\end{definition}

To construct odd Khovanov homology one must fix a diagram, fix an orientation on the crossings, and finally fix a sign assignment.  To show that odd Khovanov homology is well defined it is sufficient to show that these choices do not affect the homotopy type of the odd Khovanov bracket.  We will recall the relevant lemmas and theorems from \cite{ORS2007} and \cite{Pu2015} here for use later.
\begin{lemma}
    \label{lem:3:signinv}
    All valid sign assignments for a particular diagram with oriented crossings of a link produce isomorphic odd Khovanov brackets.
\end{lemma}
Since similar arguments will play a role later, we will briefly recall the proof of Lemma~\ref{lem:3:signinv}.
\begin{proof}
Let $Q$ denote the cube $[0,1]^n$ with its usual cell structure. After identifying $\{\pm 1\}$ with $\mathbb{Z}_2$, we can interpret an assignment of a sign to each edge in the resolution cube as a cellular 1-cochain $\epsilon\in C^1(Q;\mathbb{Z}_2)$. The condition in equation \eqref{eq:3:validFace} then translates to $\delta\epsilon=-\sigma$ where $\sigma$ denotes the 2-cochain given by sending the faces of $Q$ to the signs $\sigma_{i,j}\in\{\pm 1\}$. If $\epsilon$ and $\epsilon'$ are two valid sign assignments, it follows that their difference is a cocycle, and since $Q$ is contractible, this cocycle is a coboundary of a 0-cochain $\eta\in C^0(Q;\mathbb{Z}_2)$. The chain isomorphism $f_{\epsilon'\epsilon}$ relating the complexes corresponding to $\epsilon$ and $\epsilon'$ is now given by $\pm\operatorname{id}_{D_\alpha}$ at each vertex of the resolution cube, where the sign at the vertex $\alpha$ is given by $\eta(\alpha)\in\{\pm 1\}$.
\end{proof}
\begin{lemma}
    \label{lem:3:orientinv}
    All choices of crossing orientation for a particular diagram of a link produce isomorphic odd Khovanov brackets.
\end{lemma}
In \cite{ORS2007} and \cite{Pu2015}, this lemma is proven by showing that for any two crossing orientations $o$ and $o'$ for a given link diagram $D$, there are valid sign assignments $\epsilon$ and $\epsilon'$ such that the complexes constructed from $(D,o,\epsilon)$ and $(D,o',\epsilon')$ are identical.

We emphasize that the chain isomorphisms coming from the proofs of Lemmas~\ref{lem:3:signinv} and \ref{lem:3:orientinv} are essentially canonical. More precisely, let $C$ and $C'$ be the complexes constructed from two crossing orientations $o$ and $o'$ for $D$ and from corresponding sign assignments $\epsilon$ and $\epsilon'$. We then have:

\begin{lemma}\label{lem:3:uniqueness} If $f,g\colon C\rightarrow C'$ are two chain isomorphisms that restrict to $\pm\operatorname{id}_{D_\alpha}$ at each vertex of the resolution cube, then $f=\pm g$.
\end{lemma}
\begin{proof}
This follows because each saddle cobordism that appears in the differentials of these complexes has a well-defined sign, since it is not annihilated by 2 in $\cobIII$. If $f$ is a chain isomorphism of the stated form, the sign of $f|D_{\alpha}$ at the initial vertex of an oriented edge in the resolution cube therefore determines the sign at the terminal vertex (since $f$ must commute with the differentials). Hence, $f$ is uniquely determined by its sign at the leftmost vertex of the resolution cube.
\end{proof}

As a consequence of Lemma~\ref{lem:3:uniqueness}, the isomorphisms $f_{\epsilon'\epsilon}$ from Lemma~\ref{lem:3:signinv} satisfy the coherence conditions $f_{\epsilon\epsilon}=\pm\operatorname{id}$ and $f_{\epsilon''\epsilon'}\circ f_{\epsilon'\epsilon}=\pm f_{\epsilon''\epsilon}$.

For later use, we also note that any subcube of the cube of resolutions of a link diagram $D$ corresponds to the cube of resolutions for a link diagram with fewer crossings. Moreover, any valid sign assignment on the cube of resolutions for $D$ restricts to a valid sign assignment on this subcube. We further have:

\begin{lemma}\label{lem:3:CWsubcomplex} If $Q'$ is a subcube of $Q=[0,1]^n$ of any codimension, then any valid sign assignment on $Q'$ extends to a valid sign assignment on $Q$. More generally, this holds for any CW subcomplex $Q'\subseteq Q$ with $H^1(Q';\mathbb{Z}_2)=0$.
\end{lemma}

Note that the condition \eqref{eq:3:validFace} in the definition of a valid sign assignment still makes sense for an arbitrary CW subcomplex $Q'\subseteq Q$. Abstractly, Lemma~\ref{lem:3:CWsubcomplex} follows from an extension of the proof of Lemma~\ref{lem:3:signinv}, where we observed that any two valid sign assignments on $Q$ differ by an element of $Z^1(Q;\mathbb{Z}_2)$. Under the assumptions on $Q'$ from Lemma~\ref{lem:3:CWsubcomplex}, we have $Z^1(Q;\mathbb{Z}_2)=B^1(Q;\mathbb{Z}_2)$ and $Z^1(Q';\mathbb{Z}_2)=B^1(Q';\mathbb{Z}_2)$, and the lemma now follows because the restriction map $B^1(Q;\mathbb{Z}_2)\rightarrow B^1(Q';\mathbb{Z}_2)$ is always surjective.

\begin{proof}
More concretely, let $\epsilon'$ be a valid sign assignment on $Q'$ and choose any valid sign assignment $\epsilon$ on $Q$. Then, $\epsilon|Q'$ and $\epsilon'$ differ by the coboundary of an element $\eta'\in C^0(Q';\mathbb{Z}_2)$, which we can extend arbitrarily to an element $\eta\in C^0(Q;\mathbb{Z}_2)$. Then, $\epsilon+\delta\eta$ is a valid sign assignment (written additively) which restricts to $\epsilon'$. If we choose $\eta$ to be zero on all vertices of $Q$ that do not belong to $Q'$, then this sign assignment agrees with $\epsilon$ on all edges of $Q$ that have no endpoints in $Q'$. In this case, the complex constructed from $\epsilon+\delta\eta$ can be obtained from the complex for $\epsilon$ by applying the isomorphism $f_{\epsilon'(\epsilon|Q')}$ to vertices of $Q'$ and adjusting differentials accordingly. We will often use these observations, at least implicitly, to argue that we can fix the sign assignments on certain CW subcomplexes $Q'\subseteq Q$.
\end{proof}

The following lemmas were shown by Putyra. Note that in the pictures in equations \eqref{eq:3:RImap} and \eqref{eq:3:RIImap}, deaths are assumed to be oriented clockwise as in Convention~\ref{conv:2:orientations}.
\begin{lemma}\label{lem:3:RIinv}
    The chain maps for positive and negative crossings respectively depicted in diagrams \eqref{eq:3:RImap} and \eqref{eq:3:RImapalt} are strong deformation retractions, thus odd Khovanov homology is invariant under Reidemeister I moves.\footnote{If the orientation of the crossing is flipped, then the orientations on the split maps in the RI map will be flipped.}
\end{lemma}
\begin{equation}
    \label{eq:3:RImap}
    \includeFig{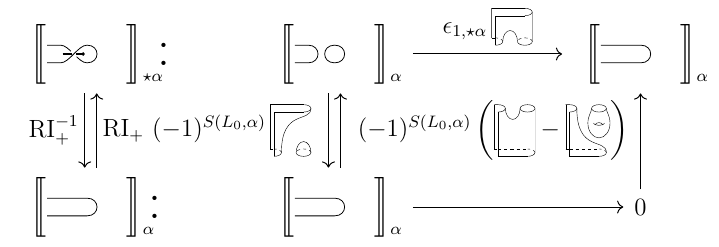}
\end{equation}
\begin{equation}
    \label{eq:3:RImapalt}
    \includeFig{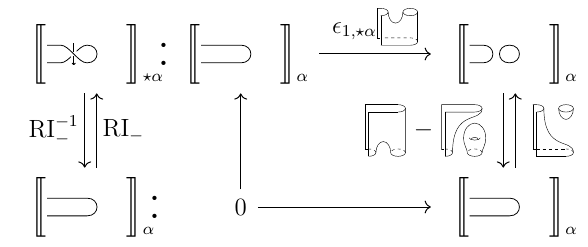}
\end{equation}
\begin{lemma}\label{lem:3:RIIinv}
    The chain map depicted in equation \eqref{eq:3:RIImap} is a strong deformation retraction, thus odd Khovanov homology is invariant under Reidemeister II moves.\footnote{The orientations on the saddles are inherited from the saddles to and from the uncrossed vertex. Thus, $f$ and $g$ inherit orientations from  $\varphi_{0,0{\star}\alpha}$ and $\varphi_{0,{\star}1\alpha}$, respectively.}
\end{lemma}
\begin{equation}
    \label{eq:3:RIImap}
    \hspace{-1.5em}\includeFig{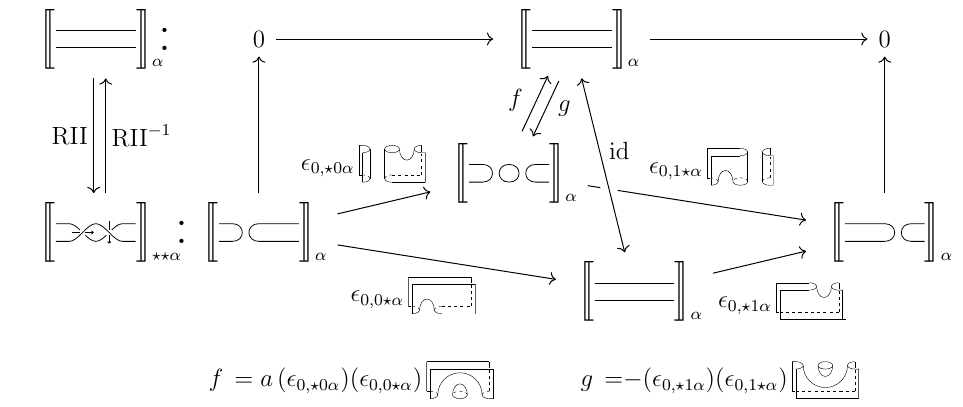}\hspace{-1.5em}
\end{equation}
The constant $a$ takes on the value 1 if the right crossing is oriented down, as in the diagram above, and $-1$ if oriented up or opposite to that in the diagram above.

To show that odd Khovanov homology is invariant under Reidemeister III moves, Putyra first notes that each side of the Reidemeister III move is the cone of the chain map obtained from resolving the crossing that the strand is passed behind during the move.
\begin{equation}
    \includeFig{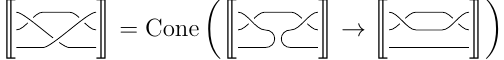}
\end{equation}
\begin{equation}
    \includeFig{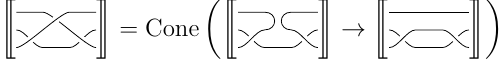}
\end{equation}
Next Putyra notes the following:
\begin{lemma}\label{lem:3:coneStrongDef}
    The homotopy equivalence classes of the cone of a chain map and the cone of that chain map composed with a strong deformation retraction are the same.
\end{lemma}
The combination of Lemmas \ref{lem:3:RIIinv} and \ref{lem:3:coneStrongDef} implies that if the maps in the cones on the right-hand sides of equations \eqref{eq:3:conesL} and \eqref{eq:3:conesR} are homotopic to one another, then so are the complexes associated to each side of the Reidemeister III move.
\begin{equation}
    \label{eq:3:conesL}
    \includeFig{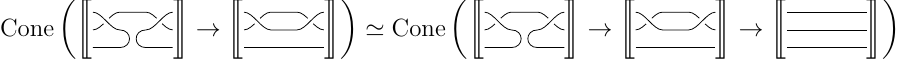}
\end{equation}
\begin{equation}
    \label{eq:3:conesR}
    \includeFig{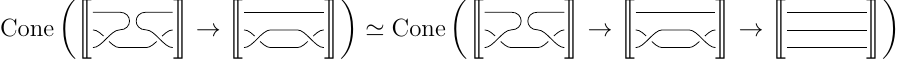}
\end{equation}

The following two diagrams depict the maps whose cones we are considering.  For this argument, consider the maps which travel up the diagrams.  There is a second version of the Reidemeister III move where the crossing that travels over the strand is negative rather than positive.  We only prove the positive crossing case here, but the maps traveling down the cube are those you would use for the other case, and the rest of the argument is essentially identical.

\begin{equation}
    \label{eq:3:RIIIL}
    \hspace{-1em}\includeFig{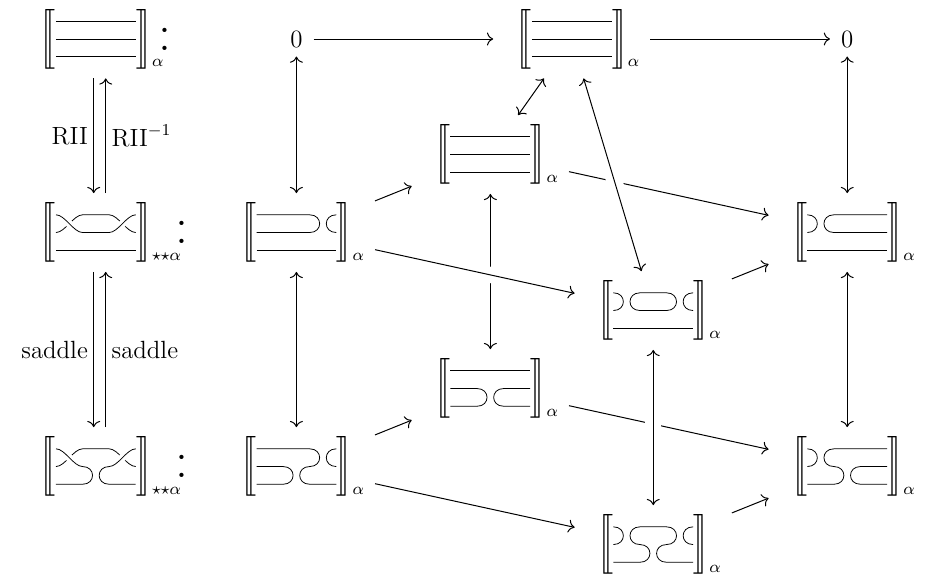}\hspace{-1em}
\end{equation}
\begin{equation}
    \label{eq:3:RIIIR}
    \hspace{-1em}\includeFig{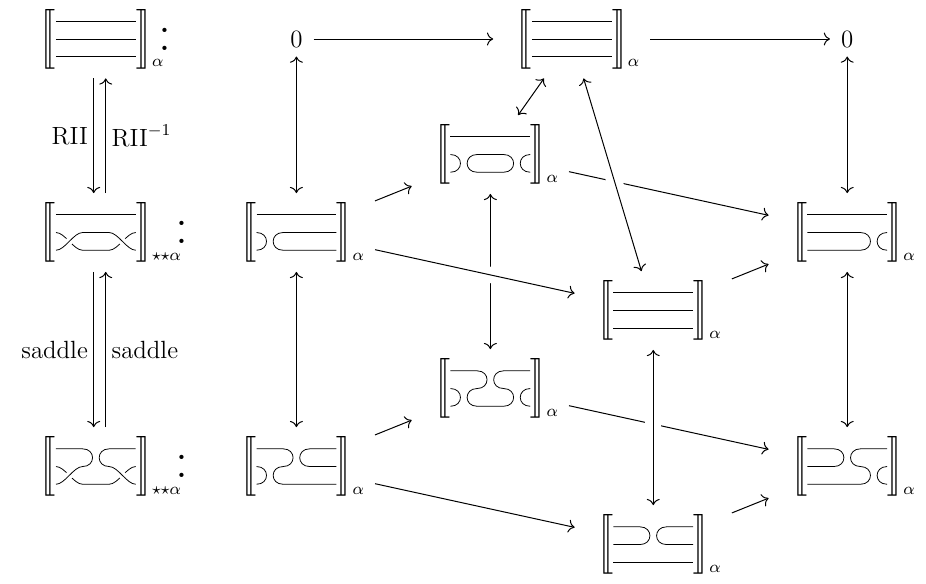}\hspace{-1em}
\end{equation}

Note that for a fixed $\alpha$ the chain map is supported by the maps in degree zero.  A careful investigation of the chain maps from diagrams \eqref{eq:3:RIImap}, \eqref{eq:3:RIIIL}, and \eqref{eq:3:RIIIR} reveals that the nonzero components of both chain maps are given by the direct sum of the following two maps.

\vspace{-1em}{\noindent}\begin{minipage}{0.5\textwidth}
    \begin{equation}
        \includeFig{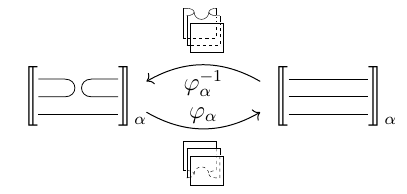}
    \end{equation}
\end{minipage}%
\begin{minipage}{0.5\textwidth}
    \begin{equation}
        \includeFig{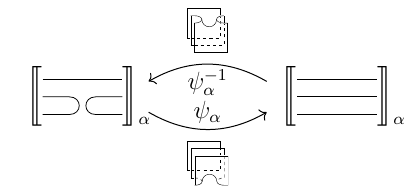}
    \end{equation}
\end{minipage}\vspace{1em}

This shows that the underlying maps are the same when coefficients are stripped away. What remains is to show that the signs in the cone in \eqref{eq:3:conesL} are consistent with those in the cone in \eqref{eq:3:conesR}. To this end, note that the cone on the right-hand side of \eqref{eq:3:conesL} contains two cubical quotient complexes: the first of these is given by the bottom layer of \eqref{eq:3:RIIIL}, and the second one is obtained from the cone complex by removing the rightmost among the four terms in this bottom layer. Corresponding to these two quotient complexes, we consider two cubes $Q\cong[0,1]^{n-1}$ and $Q'\cong[0,1]^{n-1}$, where $n$ denotes the number of crossings in the link diagrams involved in the Reidemeister III move.
Let $P\coloneqq Q\cup Q'$ denote the CW complex obtained by gluing the cubes $Q$ and $Q'$ along the two $(n-2)$-dimensional subcubes that correspond to the two edges on the left side of the bottom layer of  \eqref{eq:3:RIIIL}.

One can check directly that the signs in \eqref{eq:3:conesL} define a valid sign assignment on $P$, which we can view as a cellular 1-cochain $\epsilon\in C^1(P;\mathbb{Z}_2)$. Similarly, there is a cellular 1-cochain $\epsilon'\in C^1(P;\mathbb{Z}_2)$ coming from the signs in \eqref{eq:3:conesR}. Arguing as in the proof of Lemma~\ref{lem:3:signinv}, we see that $\epsilon$ and $\epsilon'$ must differ by a coboundary of a $0$-cochain $\eta\in C^0(P;\mathbb{Z}_2)$, and this 0-cochain gives rise to a chain isomorphism $g_\eta$ between the cones on the right-hand sides of \eqref{eq:3:conesL} and  \eqref{eq:3:conesR}. Finally, an adaption of the proof of Lemma~\ref{lem:3:uniqueness} shows that this chain isomorphism is canonical up to an overall sign.

\begin{remark}\label{rem:3:conemaps} We can equip the bottom layers of \eqref{eq:3:RIIIL} and  \eqref{eq:3:RIIIR} with the same sign assignments, so that these complexes become identical, and likewise for the top terms in \eqref{eq:3:RIIIL} and  \eqref{eq:3:RIIIR}. Lemma~\ref{lem:3:uniqueness} then implies that the chain isomorphism $g_\eta$ from above must be given by $\pm\operatorname{id}$ on both the bottom layer and the top term. This implies that the chain maps in the cones on the right-hand sides of \eqref{eq:3:conesL} and  \eqref{eq:3:conesR} are either identical or negatives of each other. We will use this observation later in the proof of invariance under movie move 15.
\end{remark}

This concludes Putyra's proof of the following lemma.
\begin{lemma}
    \label{lem:3:RIIIinv}
    Odd Khovanov homology is invariant under Reidemeister III moves.
\end{lemma}
From the combination of Putyra's Lemmas \ref{lem:3:signinv}, \ref{lem:3:orientinv}, \ref{lem:3:RIinv}, \ref{lem:3:RIIinv}, and \ref{lem:3:RIIIinv}, his main theorem follows.
\begin{theorem}
    Odd Khovanov homology is an invariant of links.
\end{theorem}
While there could a priori be many different homotopy equivalences between the two sides of a Reidemeister move, the specific homotopy equivalences described above are canonical up to an overall sign. These homotopy equivalences are also natural with respect to changes of the sign assignments, meaning that changing the sign assignments on their source and target complexes corresponds to pre- and postcomposing with the corresponding isomorphisms from Lemma~\ref{lem:3:signinv}. Moreover, changing the orientation of a crossing leaves these homotopy equivalences unchanged, as long as the sign assignments are changed accordingly.

We still need to return to the equivalence of type X and type Y theories.  Above, we reviewed the proof that odd Khovanov homology is an invariant for type Y sign assignments.  What we need to show now is that the two theories are equivalent in that there are not two distinct odd Khovanov homology invariants.

\begin{proposition}\label{prop:3:XYinv}
    Type X sign assignments and type Y sign assignments yield isomorphic odd Khovanov complexes.
\end{proposition}
    
\begin{proof}
Let $L$ be a link with a diagram $D$ with a crossing orientation $o$. Now rotate $L$, including the orientations $o$, by 180 degrees around an axis that is parallel to the plane of the picture. This produces a new diagram $D'$ with a crossing orientation $o'$.

Equivalently, the diagram $D'$ and the orientation $o'$ can be obtained by reflecting the underlying 4-valent planar graph of $D$ and $o$ along a line in the plane and then switching the roles of the strands at each crossing in the reflected diagram.

Now any type $X$ complexes for $(D,o)$ and $(D',o')$ are homotopy equivalent because $D$ and $D'$ represent isotopic links. On the other hand, any type X complex for $(D',o')$ can be seen as a type Y complex for $(D,o)$.

In conclusion, one gets a homotopy equivalence between a type X complex and a type Y complex for $(D,o)$. To get an actual isomorphism, one can use the following fact:

If two bounded complexes of finitely generated free abelian groups have the same chain ranks and the same homology, then they are isomorphic. This follows because one can use base changes to bring the matrices of the differentials into a block matrix that is Smith normal form in the lower left entry and zero elsewhere, after which, the resulting matrices are uniquely determined by the homology and by the ranks of the chain groups.
\end{proof}

\subsection{The Odd Khovanov Homology Functor}
Now that we have reviewed the construction of odd Khovanov homology, we can begin constructing the functor.  For a link cobordism $F$ from the link diagram $L_0$ to the link diagram $L_1$, we will denote the associated chain map on the odd Khovanov brackets of the links as $\Phi$.  The particular component of $\Phi$ going from $\llbracket L_0\rrbracket_\alpha$ to $\llbracket L_1\rrbracket_\alpha$ will be denoted $\Phi_\alpha$, while the underlying cobordism will be denoted $\varphi_\alpha$.  

Unlike in the construction of the odd Khovanov cube, we will want faces that occur as a result of a link cobordism to commute.  By considering cobordisms, we add an additional dimension which we will denote the movie dimension, or movie axis.

\subsubsection{Birth Cobordisms}

The birth is the easiest cobordism to handle.  Let the cobordism $F$ be a four-dimensional birth cobordism. The link $L_1$ is the disjoint union of $L_0$ and a circle.  For each diagram $\llbracket L_0\rrbracket_\alpha$ we define $\varphi_\alpha$ as the identity cobordism from those components derived from $L$ and a planar birth cobordism on the new circle in $\llbracket L_1\rrbracket_\alpha$.  As births are even, the map is already commuting, thus $\Phi_\alpha=\varphi_\alpha$.

\subsubsection{Death Cobordisms}

 Let the cobordism $F$ be a four-dimensional death cobordism. For a death, the link $L_0$ is the disjoint union of $L_1$ and a circle.  For each $\alpha$ we define $\varphi_\alpha$ as the identity cobordism from those components derived from $L_1$ and a planar death cobordism on the lone circle in $\llbracket L_0\rrbracket_\alpha$.  Unlike with births at this stage, some of the faces may commute while others anticommute.  Let $c(L,\alpha)$ be the number of circles in the diagram $\llbracket L\rrbracket_\alpha$. We define 
 \begin{equation}
     S(L,\alpha)=\frac{c(L,\alpha)+c(L,\zeta)+deg(\alpha)}{2}
 \end{equation}
 The function $S(L,-)$ is defined such that, across a merge, the increase in degree and the decrease in the number of components cancel out, while in a split they interfere constructively to increment $S$.  The $c(L,\zeta)$ term is to ensure that $S(L,-)$ is an integer.  If we set $\Phi_\alpha\coloneqq(-1)^{S(L_0,\alpha)}\varphi_\alpha$, then we arrive at a commuting square as our signs change only when a change in chronology between a death and a split occurs in the square.

\subsubsection{Saddle Cobordisms}

Let the cobordism $F$ be a single four-dimensional saddle cobordism.  For each $\alpha$, the planar diagrams $\llbracket L_0\rrbracket_\alpha$ and $\llbracket L_1\rrbracket_\alpha$ are
related by a saddle cobordism in the neighborhood where $L_0$ and $L_1$ differ. 
If we start with the disjoint union of the resolution cubes of $L_0$ and $L_1$ and add in the saddle cobordisms
$\llbracket L_0\rrbracket_\alpha\rightarrow\llbracket L_1\rrbracket_\alpha$ as extra edges, we arrive at a cube which resembles that of a link with one more crossing than $L_0$ and $L_1$ had; we will call this link $L'$.  We can label the original $n$ crossings of $L'$ as in $L_0$ with the extra crossing given the index $n+1$.  After choosing an orientation for the extra crossing, we can find an odd Khovanov cube for $L'$, and by Lemma~\ref{lem:3:CWsubcomplex}, we can assume that the sign assignment on this cube restricts to the given sign assignments on the cubes of $L_0$ and $L_1$. The components $\partial'_{\alpha*}\colon\llbracket L_0\rrbracket_\alpha\rightarrow\llbracket L_1\rrbracket_\alpha$ of the differential in $\llbracket L'\rrbracket$ then anticommute with the differentials in $\llbracket L_0\rrbracket$ and $\llbracket L_1\rrbracket$. To get a chain map $\Phi\colon\llbracket L_0\rrbracket\rightarrow\llbracket L_1\rrbracket$, we will set $\Phi_\alpha\coloneqq(-1)^{\deg(\alpha)}\partial'_{\alpha*}$.

Note that this chain map is defined canonically up to an overall sign because the space of relative 1-cocycles $Z^1(Q',Q_0\cup Q_1;\mathbb{Z}_2)$ is isomorphic to $\mathbb{Z}_2\cong\{\pm 1\}$, where $Q'\coloneqq[0,1]^{n+1}$ denotes the cube corresponding to $L'$, and $Q_0$ and $Q_1$ are the subcubes corresponding to $L_0$ and $L_1$. 
The chain map $\pm\Phi$ is also independent of the choice of the orientation of the last crossing of $L'$. This follows because the chain complex for $L'$ stays the same if one reverses the orientation of the last crossing of $L'$ while also changing the sign assignment on edges of $Q'$ that correspond to split cobordisms $\llbracket L_0\rrbracket_\alpha\rightarrow\llbracket L_1\rrbracket_\alpha$ \cite{ORS2007}. Likewise, changing the orientation of a crossing of $L_0$ or $L_1$ leaves the chain complexes for $L_0$ and $L_1$ and the chain map $\pm\Phi$ unchanged, provided one adjusts the sign assignments accordingly. Finally, the definition of $\pm\Phi$ is compatible with changes of the sign assignments for $L_0$ or $L_1$. Indeed, such changes have the same effect as pre- or postcomposing $\pm\Phi$ with one of the isomorphisms from Lemma~\ref{lem:3:signinv}.

If a saddle cobordism merges two components of the link diagram $L_0$ (viewed as a 4-valent graph), we can use the same sign assignments for $L_0$ and $L_1$. We then do not need to pass to the odd Khovanov cube of $L'$ to construct the map $\Phi$. Instead, $\Phi$ is given by the same saddle cobordisms $\Phi_\alpha$ as in the general case, but without any signs.

If a saddle cobordism splits a component of the diagram $L_0$ into two, then, as in the merge case, the map $\Phi$ can be made more explicit. The $\Phi_\alpha$ themselves are the same as in the general case, but we can define the signs explicitly by using the term $(-1)^{S(L_0,\alpha)}$ that we used for a death cobordism.

\subsubsection{Reidemeister Type Cobordisms}

The remaining three basic link cobordisms are those associated with Reidemeister moves.  We will assign the chain maps used by Putyra \cite{Pu2015} in his proof of the invariance for his general construction, specialized to odd Khovanov Homology to Reidemeister type cobordisms.  To a Reidemeister I move we assign the chain map in diagram \eqref{eq:3:RImap}, and to a Reidemeister II move we assign the chain map in diagram \eqref{eq:3:RIImap}.  For a Reidemeister III cobordism, the maps are those that are induced by the homotopy equivalences between the cones in equations \eqref{eq:3:conesL} and \eqref{eq:3:conesR}, and the homotopy in Lemma \ref{lem:3:RIIIinv}.  While not the full chain map, the following diagram gives the reader the tools to build the final chain map.

\begin{equation}
    \includeFig{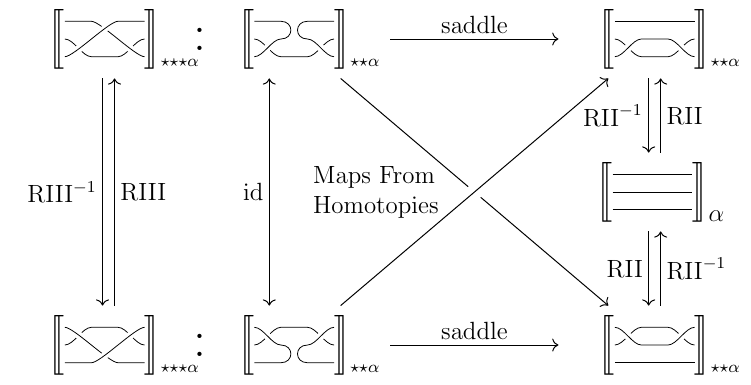}
\end{equation}

\section{The Functoriality of Odd Khovanov Homology up to Sign}

We now have all the elements in place to prove the main theorem of the dissertation.

\begin{theorem}\label{thr:OKHfunctor}
    Odd Khovanov homology extends to a functor from the category $\cobIV$ to the category $\cobIIIKbMpm$.
\end{theorem}

In Section 3 we specified what odd Khovanov homology assigns to movies of link cobordisms.  The entirety of this section is devoted to the proof of Theorem \ref{thr:OKHfunctor}, where we show how our construction respects all the possible basic ambient isotopies of $\cobIV$.  We will begin with a discussion of the type I movie moves, followed by a general argument to show that each side of the type II movie moves produces homotopic chain maps.  Next, we will examine type III movie moves with individual arguments in the forward direction, in the reverse direction, and for alternative variants.  Finally, we will show the functoriality of odd Khovanov homology with respect to chronological movie moves. We do not need to spend additional time ensuring our functor respects equivalences of planar ambient isotopy cobordisms.  Such equivalences are equalities in both $\cobIV$ and $\cobIII$, and are thus respected by our functor.

\subsection{Type I Movie Moves}

\begin{figure}[H]
    \centering
    \begin{tabular}{ c }
        Type \\
        I \\ 
        Movie \\
        Moves
    \end{tabular}
    \(= \left[
    \addstackgap[0.25em]{\begin{tabular}{c}\includeMov{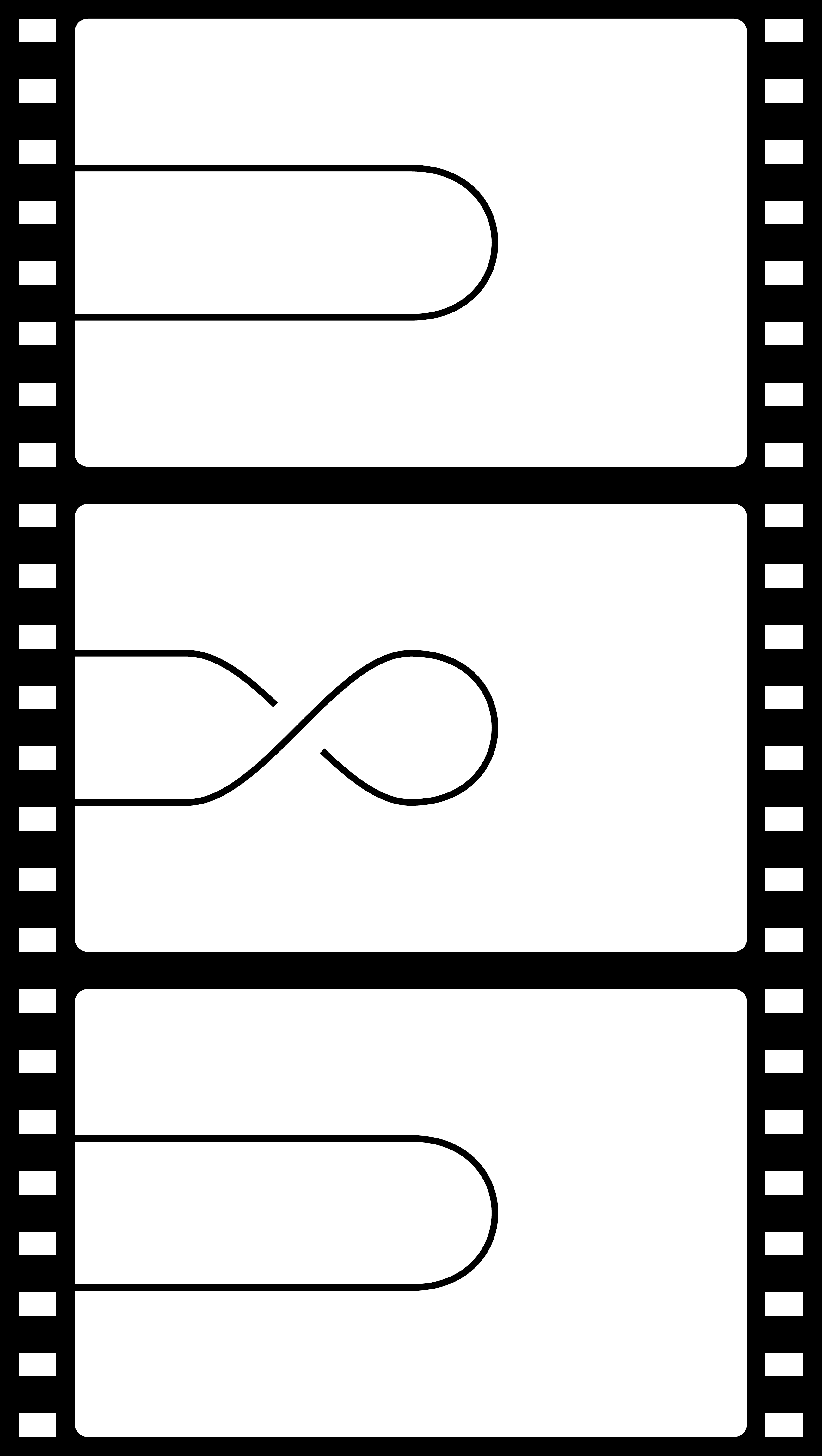}\\MM1\end{tabular}}
    \begin{tabular}{c}\includeMov{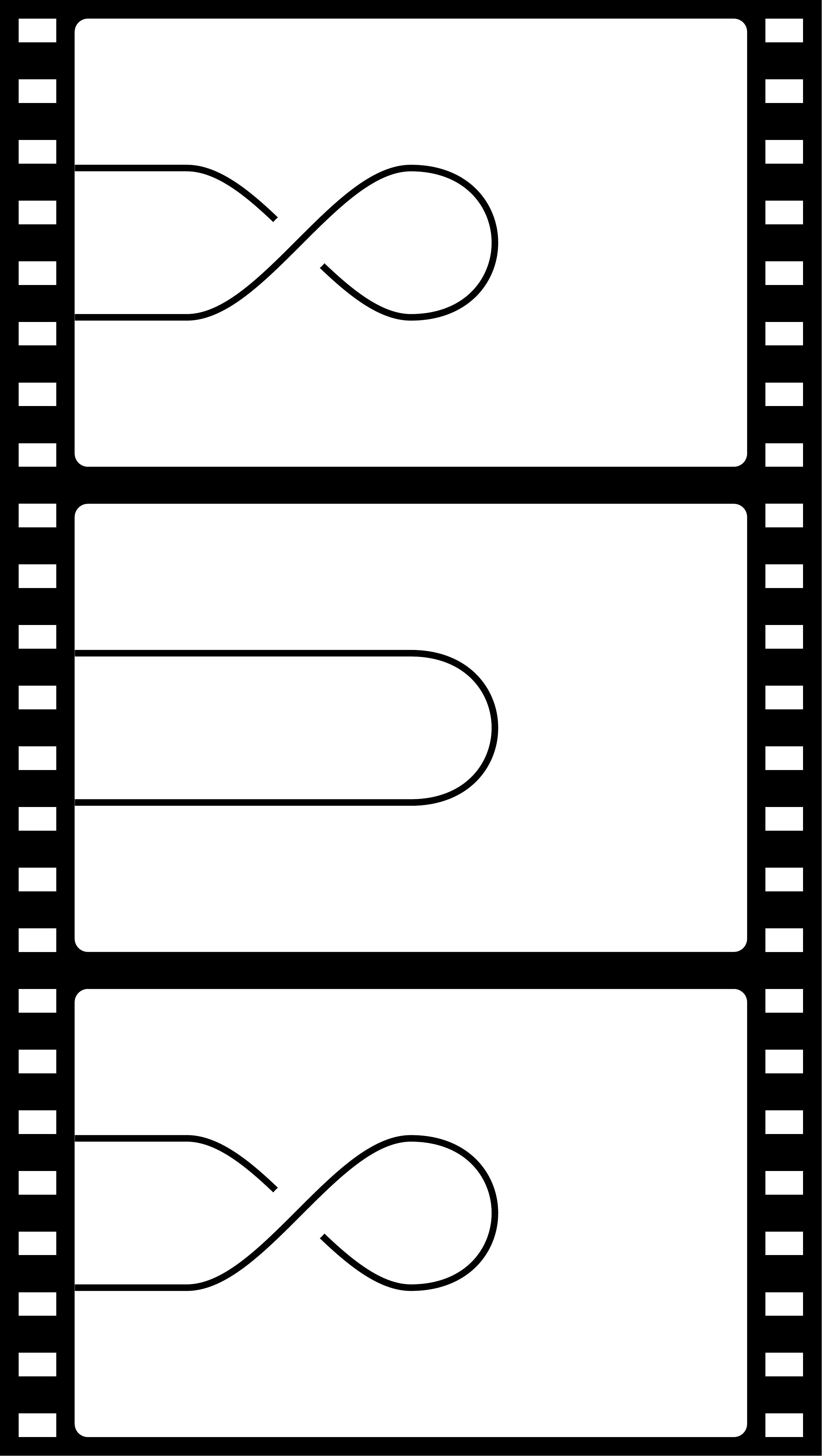}\\MM2\end{tabular} 
    \begin{tabular}{c}\includeMov{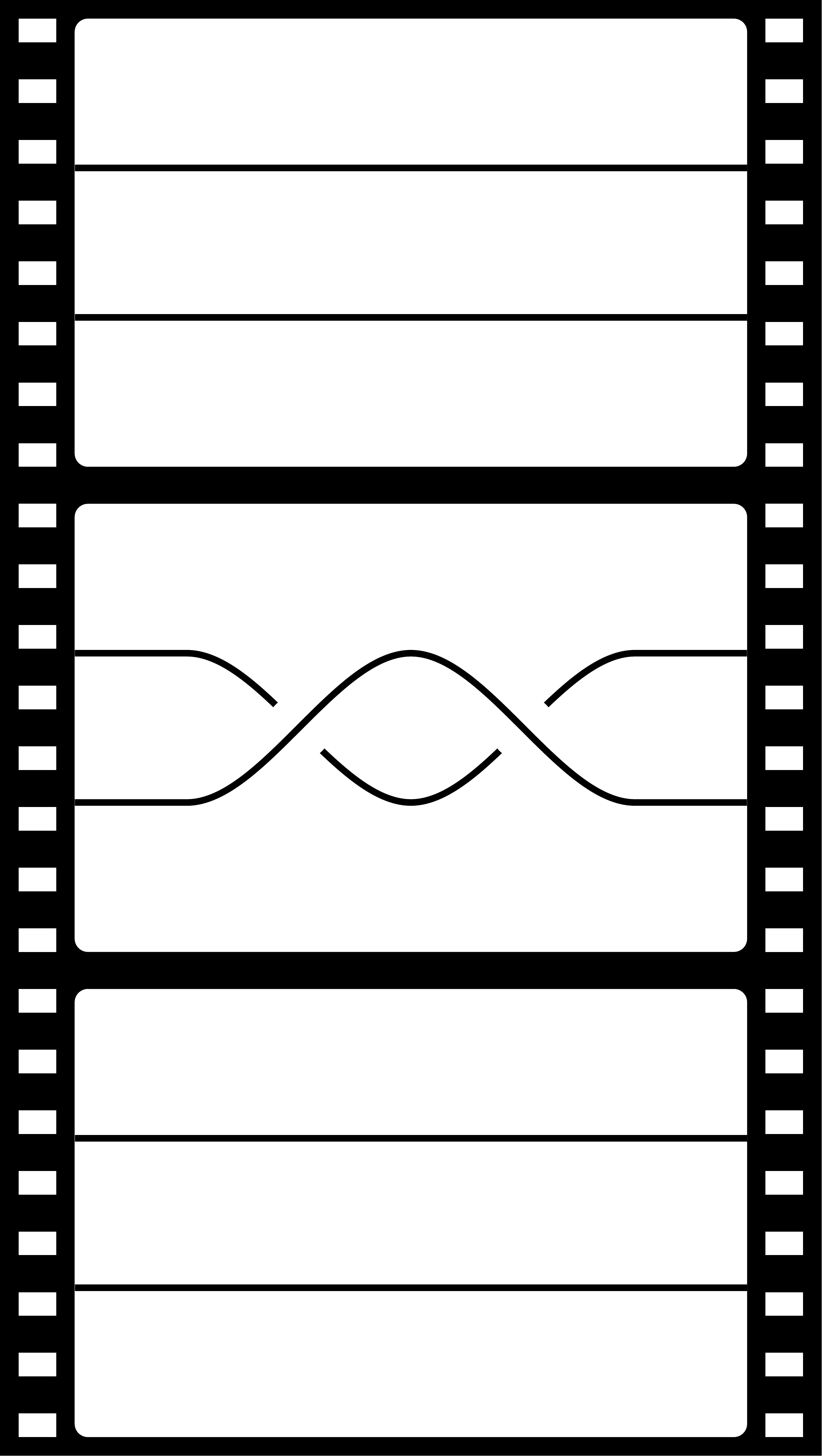}\\MM3\end{tabular} 
    \begin{tabular}{c}\includeMov{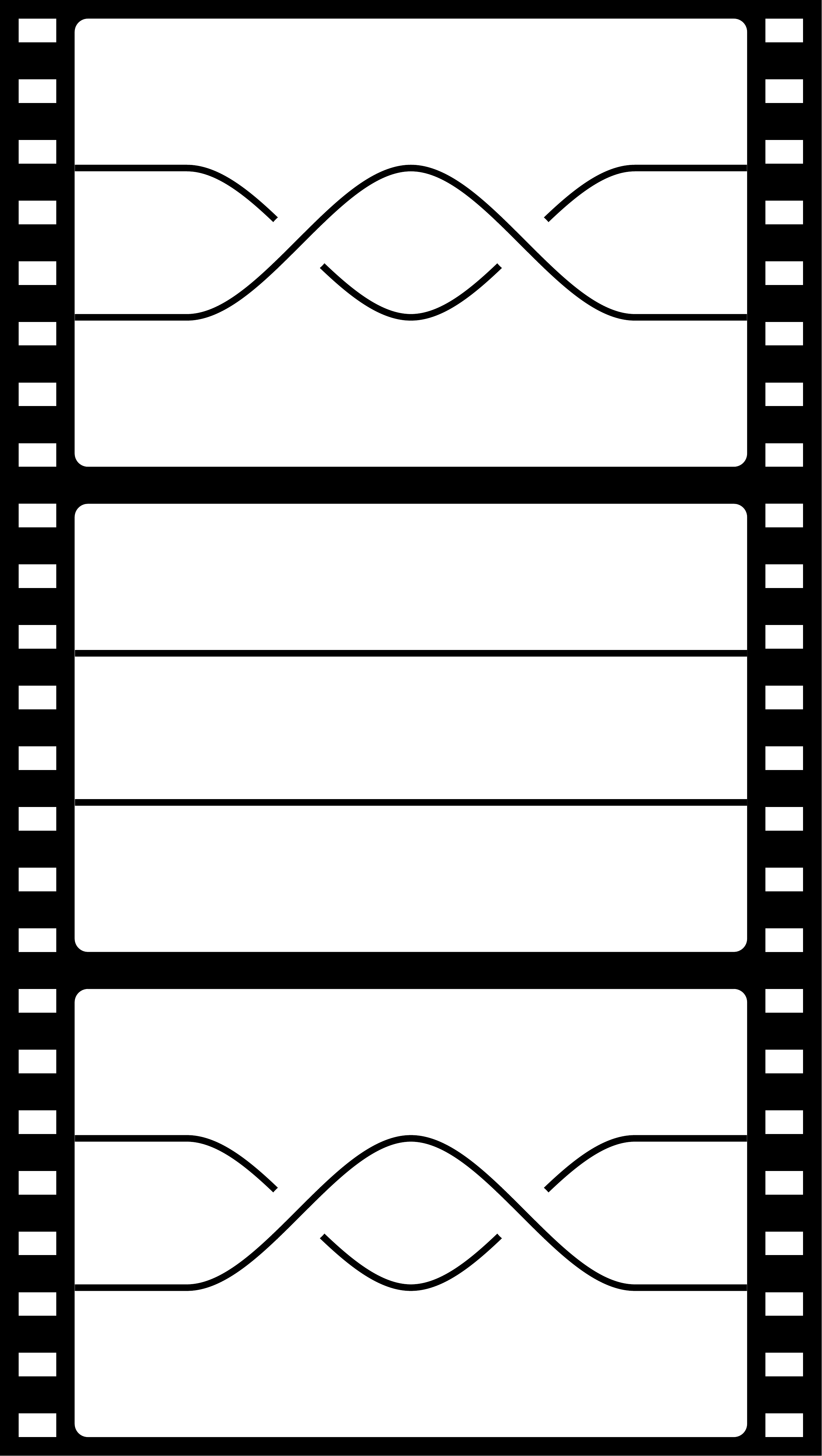}\\MM4\end{tabular} 
    \begin{tabular}{c}\includeMov{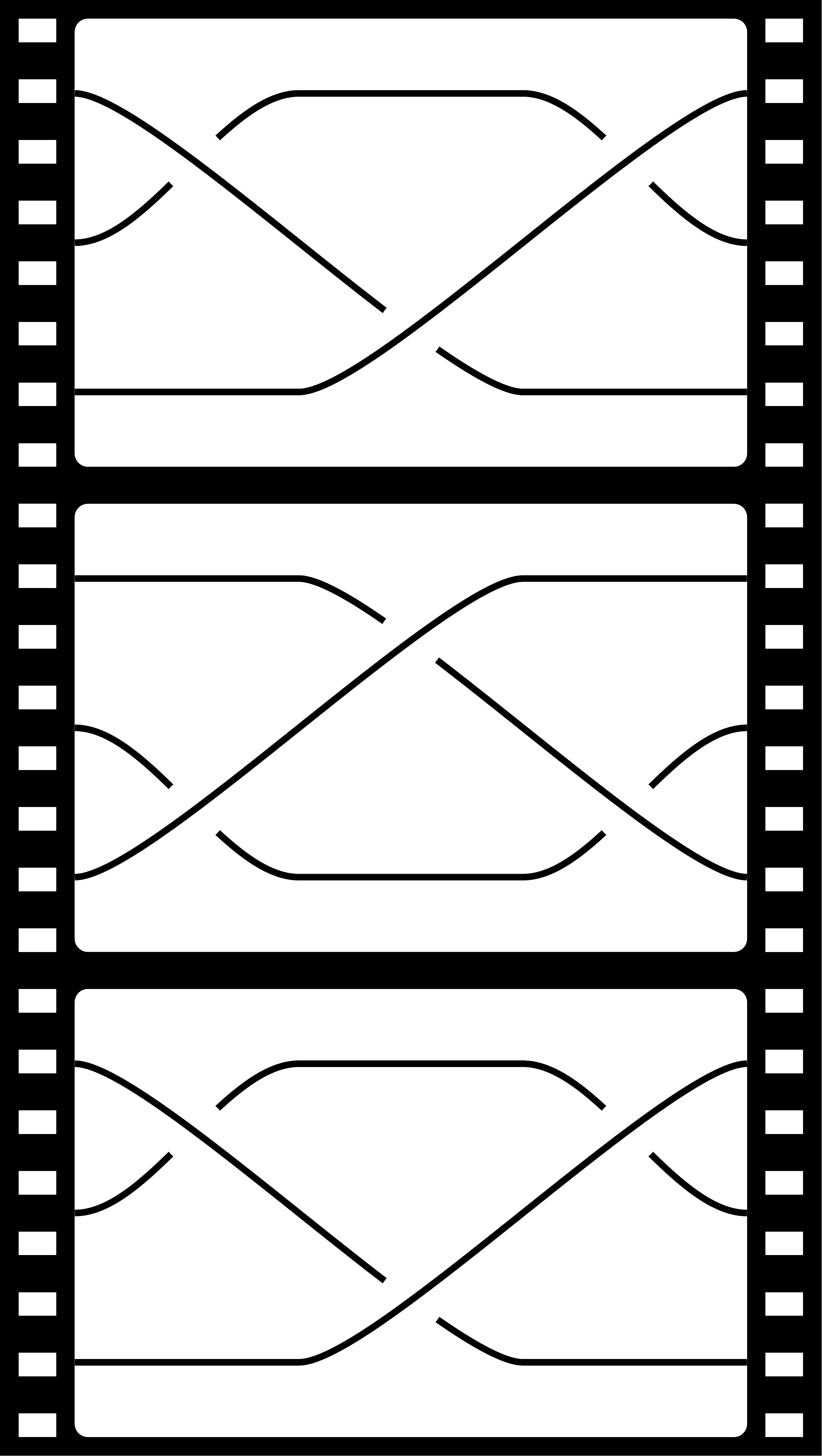}\\MM5\end{tabular} 
    \right]\)
    \caption{Movie moves 1 through 5}
    \label{fig:4:TIMM}
\end{figure}

The left-hand sides of the first five movie moves correspond with doing and undoing a Reidemeister move, while the right-hand sides (which are not shown in Figure \ref{fig:4:TIMM}) are given by trivial movies of identity cobordisms. Our chain maps for Reidemeister moves are precisely those used by Putyra, but specialized to the odd case \cite{Pu2015}.  The proof of the invariance of odd Khovaonov homology with respect to Reidemeister moves precisely implies that the left-hand sides of these movie moves induce chain maps homotopic to identity.

\subsection{Type II Movie Moves}
 
\begin{figure}[H]
    \centering
    \begin{tabular}{ c }
        Type \\
        II \\ 
        Movie \\
        Moves
    \end{tabular}
    \(= \left[
    \begin{tabular}{c}\includeMov{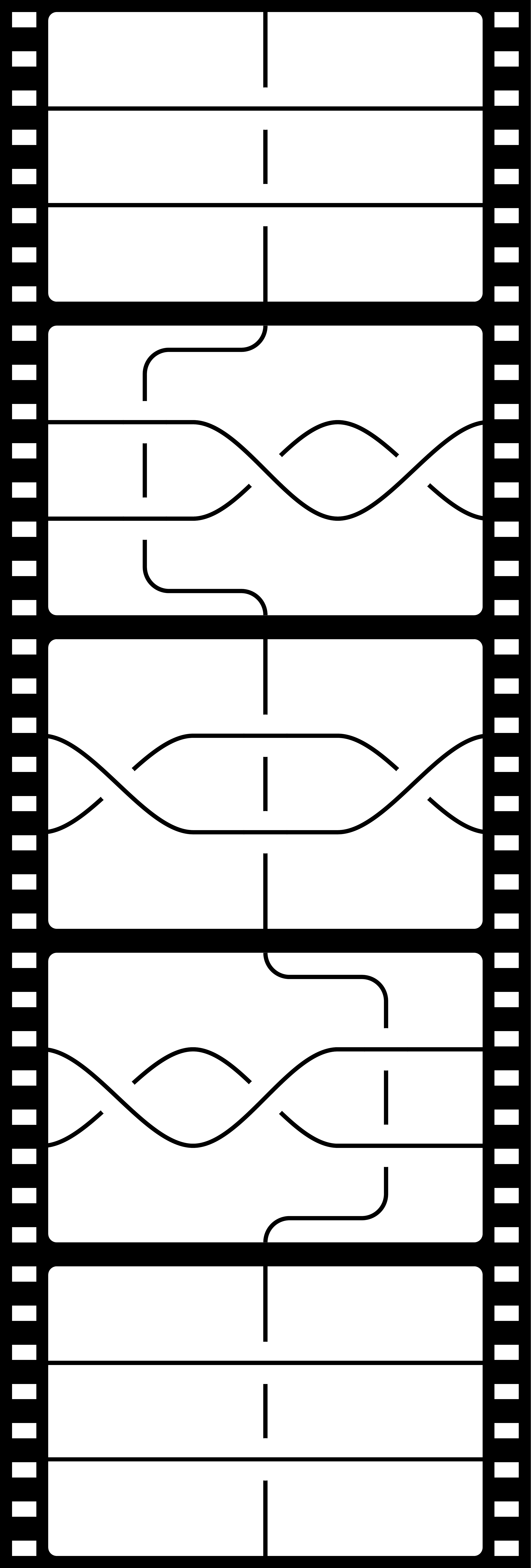}\\MM6\end{tabular}
    \begin{tabular}{c}\includeMov{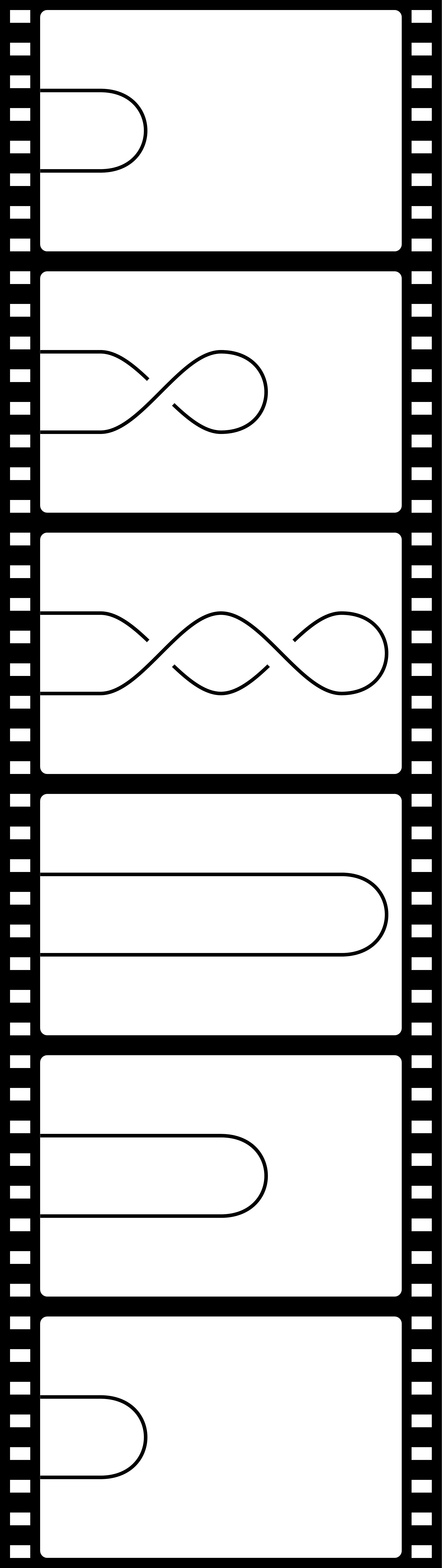}\\MM7\end{tabular} 
    \begin{tabular}{c}\includeMov{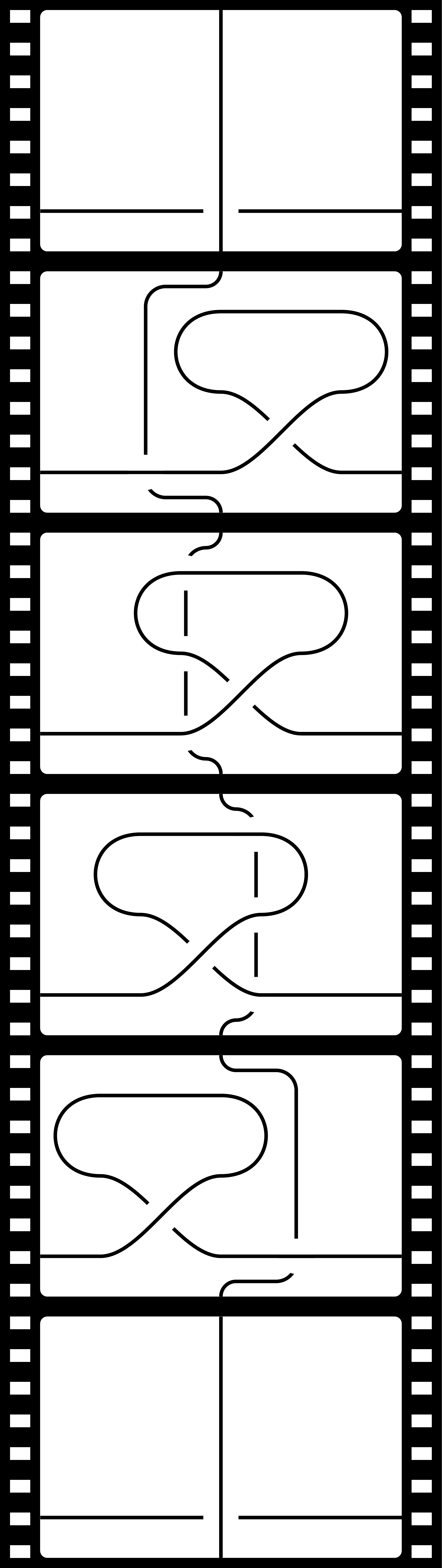}\\MM8\end{tabular} 
    \begin{tabular}{c}\includeMov{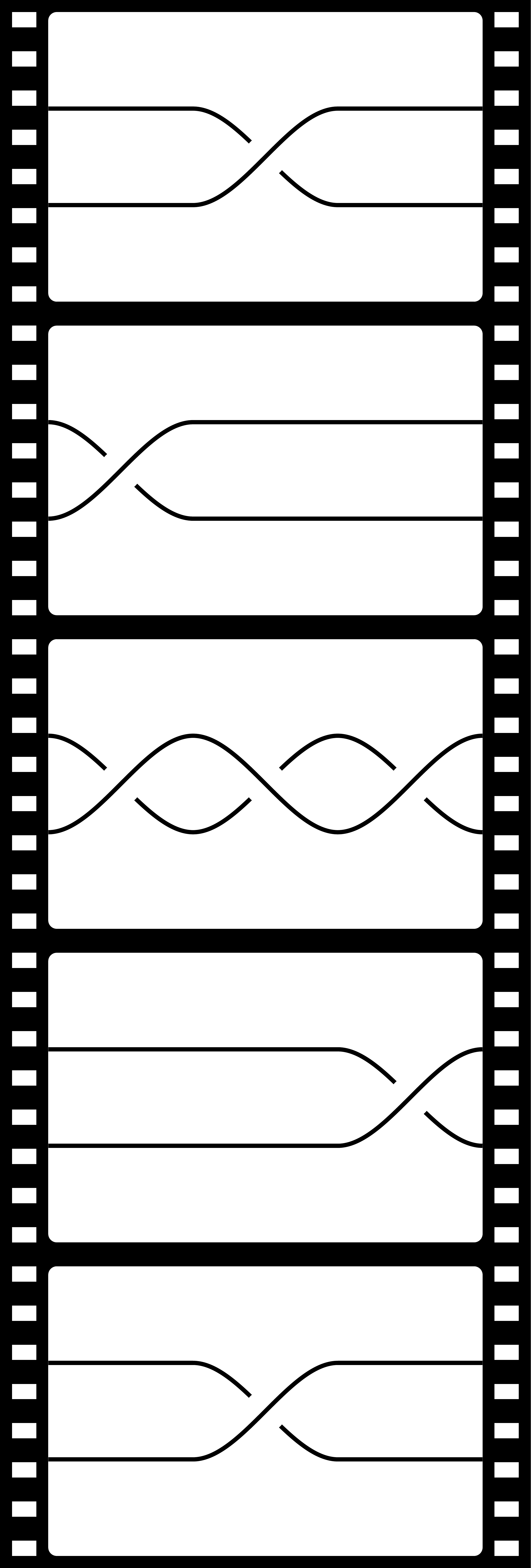}\\MM9\end{tabular} 
    \addstackgap[0.25em]{\begin{tabular}{c}\includeMov{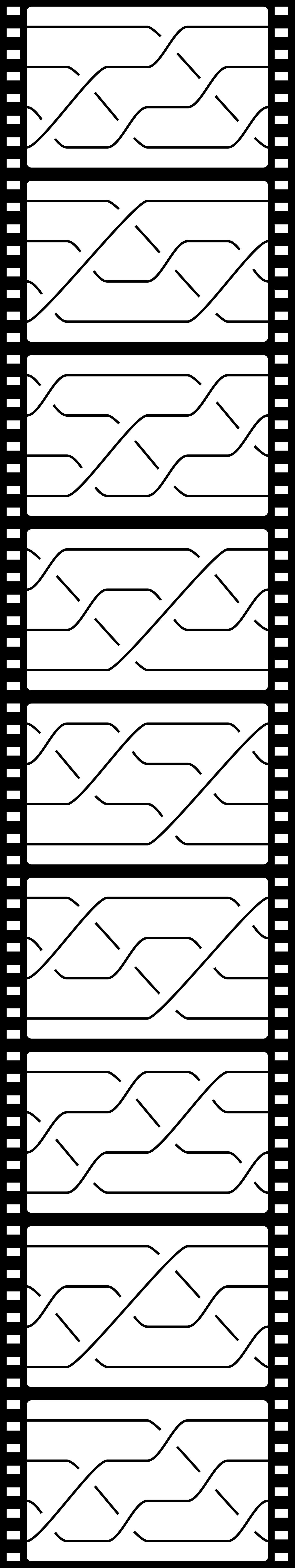}\\MM10\end{tabular}} 
    \right]\)
    \caption{Movie moves 6 through 10}
    \label{fig:4:TIIMM}
\end{figure}

Any link on which a movie move is carried out consists of two tangles glued together: namely the inside part $t$ where the movie move is carried out, and an outside part $T$ which is carried through the movie by identity.  Type II movie moves permit a slightly stricter decomposition wherein the inside tangle can be decomposed into a crossingless tangle $C$---with no closed components---and an annular braid $\beta$.\footnote{$\beta$ is an annular braid in the sense that it projects onto an annulus.}  To simplify our pictures, we will slice our annular region in half; an example of this entire decomposition is shown in Figure \ref{fig:4:2:TIIMMBraidDecomposition}.

The left-hand movie of a type II movie move, or a type I movie move for that matter, is a sequence of Reidemeister moves that start and end on identical frames, while the right-hand side is just the identity cobordism.  
\begin{figure}[H]
    \centering
    \includegraphics[width=\textwidth]{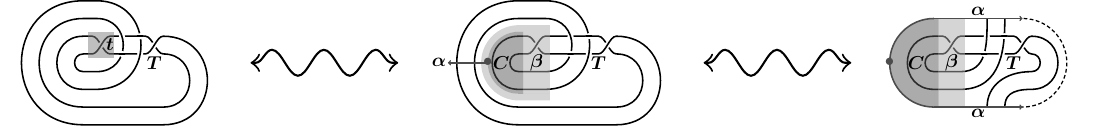}
    \caption{The decomposition of a link diagram into a crossingless tangle $C$---with no closed components---an annular braid $\beta$, and an outside part $T$}
    \label{fig:4:2:TIIMMBraidDecomposition}
\end{figure}
\begin{lemma}\label{lem:t2}
    Let $F$ be a link cobordism with a link diagram $D$ as both its initial and terminal frame.  Furthermore, suppose $F$ is generated by performing a sequence of Reidemeister moves on a tangle in $D$ that can be decomposed into a crossingless tangle with no closed components and an annular braid surrounding it.  Then the chain map that $F$ induces on the odd Khovanov bracket of the link diagram $D$ is homotopic to $\pm\id$.
\end{lemma}
\begin{proof}
Let $C$ be the crossingless tangle, $\beta$ the annular braid, and $T$ the arbitary tangle on the outside.  The link cobordism $F$ is the identity cobordism on the tangle $T$, and on the tangle $C\beta$ it is the cobordism $f$ generated by a sequence of Reidemeister moves. Let $\Phi$ be the map that $F$ induces on the odd Khovanov bracket.
\begin{equation}
\includeFig{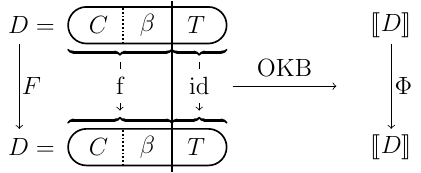}
\end{equation}
To prove the lemma, we must show that $\Phi\simeq\pm\id$.  Consider the alternative diagram $D'$ of the same link obtained by gluing the annular braid $\beta^{-1}\beta$ onto $T$, as shown in \eqref{dig:4:proofII1}. The sequence of Reidemeister moves considered previously gives rise to a cobordism $F'$ from $D'$ to $D'$, and this cobordism induces a map $\Phi'$ on the odd Khovanov bracket of $D'$.
\begin{equation}\label{dig:4:proofII1}
\includeFig{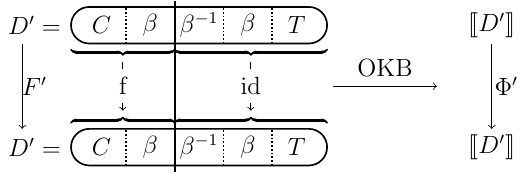}
\end{equation}
We can also consider the cobordism $G$ from $D$ to $D'$, which is the identity cobordism on $C\beta$ and a cobordism $g$ comprised of many Reidemeister II type cobordisms from $T$ to $\beta^{-1}\beta T$.  Let $\Psi$ be the homotopy equivalence that $G$ induces between the odd Khovanov brackets of $D$ and $D'$.
\begin{equation}
\includeFig{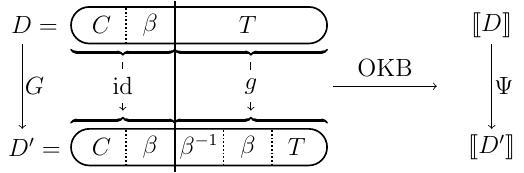}
\end{equation}
The Odd Khovanov bracket induces the following diagram when $F$, $F'$, and $G$ are considered together.
\begin{equation}\label{dig:phiphiprime}
\includeFig{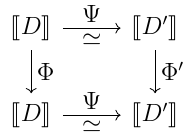}
\end{equation}
\begin{claim}\label{claim:t21}
The preceding diagram commutes up to sign and homotopy.
\end{claim}
\begin{claim}\label{claim:t22}
$\Phi'\simeq\pm\id$.
\end{claim}
We will wait until the end of the proof to prove the first claim.  It follows from Claim \ref{claim:t21} that to prove Lemma \ref{lem:t2} it is sufficient to prove Claim \ref{claim:t22}. \par

For this, we consider an alternative decomposition of $D'$ in which the tangle $\beta\beta^{-1}$ is glued onto the perimeter of $C$ as shown in the bottom half of \eqref{dig:4:proofII2}. Let $G'$ be the link cobordism from $D$ to $D'$ which is given by the identity cobordism of $\beta T$ and a cobordism $g'$ comprised of many Reidemeister II type cobordisms from $C$ to $C\beta\beta^{-1}$.

The link cobordism $G'$ induces a homotopy equivalence $\Psi'$ between odd Khovanov brackets:

\begin{equation}\label{dig:4:proofII2}
\includeFig{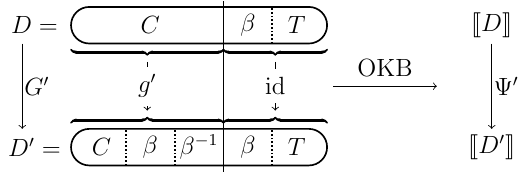}
\end{equation}
Now consider the following automorphism of $\llbracket D\rrbracket$ where $\left(\Psi'\right)^{-1}$ is the homotopy inverse of $\Psi'$:
\begin{equation}
    \Phi''\coloneqq\left(\Psi'\right)^{-1}\circ\Phi'\circ\Psi'
\end{equation}
By the definition of $\Phi''$, the following diagram commutes up to homotopy:
\begin{equation}
\includeFig{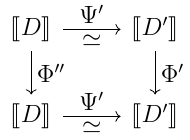}
\end{equation}
\begin{claim}\label{claim:t23}
$\Phi''=\pm\id$.
\end{claim}
The commutativity of the preceding diagram up to homotopy implies that to prove Claim \ref{claim:t22} it is sufficient to prove Claim \ref{claim:t23}.\par
To prove Claim \ref{claim:t23}, we will use that $\Phi''$ is, in a sense localized to the left-hand tangle $C$ in the decomposition of $D$.  More precisely, there is a natural cobordism $F''$---that induces $\Phi''$ on the odd Khovanov bracket---given by composing $G'$ played in reverse with $F'$ and $G'$.  

\begin{equation}
\includeFig{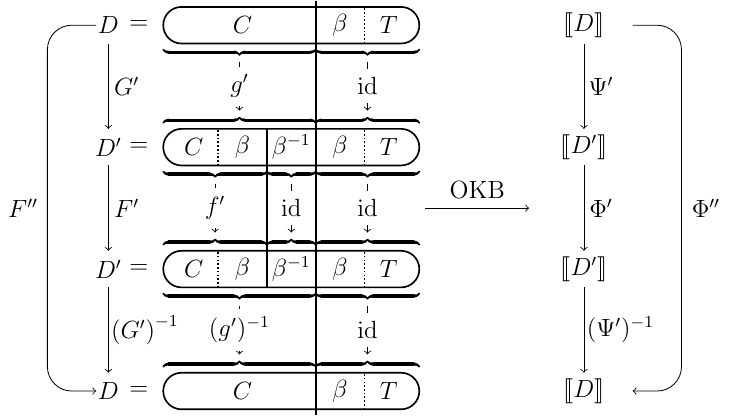}
\end{equation}
Now note that the link cobordism $F''$ acts entirely on the crossingless tangle with no closed components, $C$.  Let $\Phi''_\alpha$ be the component of the map induced by $F''$ on the odd Khovanov cube of $D$ mapping from the $\alpha$ vertex of the odd Khovanov cube.  Furthermore, we can think of the odd Khovanov cube of $D$ as doubly graded with a left degree arising from the number of 1-resolved crossings in the tangle $C$, and the right degree arising from the number of 1-resolved crossings in the tangle $\beta T$. We can make the following conclusions about $F''$, the $\Phi_\alpha''$s, and $\Phi''$.
\begin{itemize}
    \item Any two individual maps $\Phi''_\alpha$ and $\Phi''_\beta$ differ only in the surface above $C$.
    \item As $F''$ is the identity cobordism on $\beta T$, $\Phi''$ preserves the cubical structure arising from the tangle $\beta T$.  That is prior to flattening, $\Phi''$ is an endomorphism of the odd Khovanov cube generated by resolving the crossings in $\beta T$ and has 0 right degree
    \item As $F''$ is a composition of Reidemeister type cobordisms, and the chain maps associated with Reidemeister moves are homotopy equivalences, $\Phi''$ is a homotopy equivalence.
    \item As $F''$ acts entirely on the left tangle, the homotopies between the identity map and the two possible compositions of $\Phi''$ with its homotopy inverse $(\Phi'')^{-1}$ have left degree -1, and right degree 0.
    \item As $C$ is a crossingless tangle, it does not contribute to the cubical structure of the odd Khovanov cube of $D$.  This implies that the entire cubical structure of $D$ arises from $\beta T$.  In turn, as $\Phi''$ preserves the cubical structure of $\beta T$, it must also preserve the cubical structure arising from all of $D$.
    \item It follows that the degree of the vertices is their right degree, and that they are all supported in the same left degree.
    \item It follows that the homotopies between the identity map and the two possible compositions of $\Phi''$ with $(\Phi'')^{-1}$ must be zero maps.
    \item It follows that $\Phi''$ is a chain isomorphism.
    \item It follows that each $\Phi''_\alpha$ is an isomorphism.
    \item The only automorphisms of $D_\alpha$ which are given by a cobordism that is identity except possibly over a crossingless tangle with no closed components, are $\pm\id$. Thus, each $\Phi''_\alpha=\pm\id$.
\end{itemize}
We are left to show that either $\Phi''_\alpha=\id$ or $\Phi''_\alpha=-\id$ for all $\alpha$ uniformly.  Consider a pair of indices $\alpha$ and $\beta$ that differ in a single place, such that $\deg(\beta)=\deg(\alpha)+1$.  We have the following diagram where $\partial$ is the differential in the odd Khovanov bracket.
\begin{equation}
\includeFig{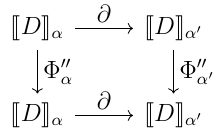}
\end{equation}
The preceding diagram must commute as $\Phi''$ is a chain map, thus $\Phi''_\alpha$ and $\Phi''_\beta$ have the same sign.  It follows that the sign $\Phi''_\zeta$ propagates across the entire cube, and thus $\Phi''=\pm\id$.\\
All that is left is to prove Claim \ref{claim:t21}.  We are going to show that the diagram commutes up to homotopy and sign via induction on the number of generators of $\beta$.  Let $\beta_n$ be the annular braid formed by the first $n$ crossings from $\beta$.  Let $D'_n$ and $\Phi'_n$ be the diagrams and maps defined, respectively, below. 
\begin{equation}
\includeFig{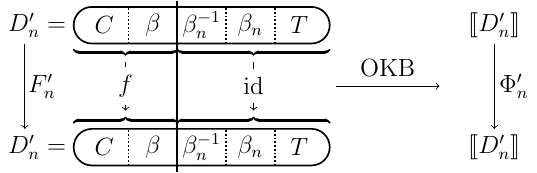}
\end{equation}
For a fixed $n$ less than the number of crossings in $\beta$, we have the following diagram.
\begin{equation}\label{eq:4:RIIcommute}
\includeFig{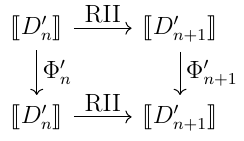}
\end{equation}
Consider the following map where we travel around the diagram clockwise:
\begin{equation}
\includeFig{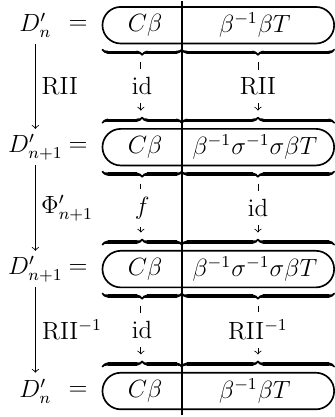}
\end{equation}
The cobordism on the right-hand side features doing and then undoing a Reidemeister II move. An examination of the chain maps associated with Reidemeister II cobordisms shows that they are given by a sum of an identity cobordism and a cobordism containing a saddle and a birth.  By immediately following up a Reidemeister II move with its inverse, we cap off the birth with a death.  This results in a sphere that kills the cobordism (by the sphere relation from \eqref{eq:2:SandTrelations}), leaving behind only a term consisting of composed identities.  Thus, the diagram in \eqref{eq:4:RIIcommute} commutes up to sign and homotopy and, in turn, claim \ref{claim:t21} holds.

Note that in this argument, we used that the maps which we assigned to each individual Reidemeister move are canonical up to an overall sign. In particular, we can assume that the map $\Phi'_{n+1}$ from \eqref{eq:4:RIIcommute} restricts to $\Phi'_n$ on the codimension 2 subcube of the cube of $D'_{n+1}$ that corresponds to replacing both crossings in $\sigma^{-1}\sigma$ by their braidlike resolution.
\end{proof}

It is an immediate consequence of Lemma \ref{lem:t2} that the maps induced by link cobordisms are preserved under type II movie moves up to an overall sign.

\subsection{Type III Movie Moves}

The last five movie moves involve births or deaths in addition to Reidemeister type cobordisms.  Additionally, the movie moves viewed either forward or in reverse must be treated as independent moves.  
\begin{convention} The following conventions will be used throughout the duration of proving the invariance of odd Khovanov homology with respect to movie moves 11 through 15.
    \begin{enumerate}[label=\textnormal{\bf{\alph*.}}]
        \item The initial frame of a movie is the $0$\ordth frame.
        \item Specific diagrams can be referenced by $D_{i,\alpha}$, or the $\alpha$\ordth diagram of the $i$\ordth frame of the cobordism.
        \item $F$ denotes the entire link cobordism, and $F_i$ denotes the link cobordism from the $i$\ordth frame of the movie to the $(i+1)$\ordst.
        \item For intra-frame maps, $d_{i,{\star}\alpha}$, $\epsilon_{i,{\star}\alpha}$, and $\partial_{i,{\star}\alpha}\coloneqq\epsilon_{i,{\star}\alpha}d_{i,{\star}\alpha}$ denote the unsigned planar cobordism, the sign, and the final map from the $0\alpha$ vertex of the cube of resolutions of the $i$\ordth frame of the movie to the $1\alpha$ vertex of the same frame respectively.
        \item For inter-frame maps, $\varphi_{i,\alpha}^{i+1,\beta}$, $\epsilon_{i,\alpha}^{i+1,\beta}$, and $\Phi_{i,\alpha}^{i+1,\beta}\coloneqq\epsilon_{i,\alpha}^{i+1,\beta}\varphi_{i,\alpha}^{i+1,\beta}$ denote the unsigned planar cobordism, the sign, and final map from the $\alpha$ vertex of the cube of resolutions of the $i$\ordth frame of the movie to the $\beta$ vertex of the next frame respectively.  We will allow $\star$s to appear in the intra-frame index if we are referring to a particular collection of the inter-frame maps.
        \item Within a movie move, an arrow can be affixed to any of the following symbols $\mathbin{\Diamond}\in\{D,F,d,\epsilon,\partial,\varphi,\Phi\}$
        to denote if it comes from the left movie $\overleftarrow{\mathbin{\Diamond}}$ or the right movie $\overrightarrow{\mathbin{\Diamond}}$.
    \end{enumerate}
\end{convention}

An immediate consequence of Lemmas \ref{lem:3:signinv} and \ref{lem:3:orientinv} is the following.

\begin{corollary}
To show odd Khovanov homology is invariant with respect to a particular movie move, it is sufficient to show invariance with respect to a particular orientation on the crossings and particular valid sign assignments.
\end{corollary}

\begin{proof}
This follows because changing the sign assignment in the source or target complex has the same effect as pre- or postcomposing the chain map induced by a movie with the corresponding isomorphism from Lemma \ref{lem:3:signinv}. Moreover, changing the orientation of a crossing does not change the source or target complex, or the induced chain map, as long as one changes the sign assignments accordingly.
\end{proof}

\subsubsection{Movie Move 11}

\begin{figure}[H]
    \centering
    \includeFig{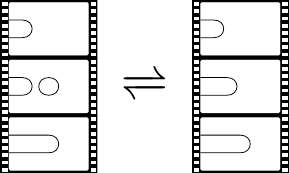}
    \caption{Movie move 11}
    \label{fig:4:MM11}
\end{figure}

Movie move 11 is built from the following elements:
\begin{itemize}
    \item A split/merge cobordism
    \item A death/birth cobordism
\end{itemize}
There is a split in the forward direction, so we will orient it in the canonical direction towards the reader. In the forward and reverse directions, movie move 11 is precisely a \enquote{pruning} relation.  From equation \eqref{eq:annihilationCreation}, the birth followed by a merge relation is an isotopy while the split followed by a death relation may incur an overall sign depending on how the orientations placed on the death and split are related to one another, thus for all $\alpha$, $\overleftarrow{\varphi}_{1,\alpha}^{2,\alpha}\overleftarrow{\varphi}_{0,\alpha}^{1,\alpha}=\pm\overrightarrow{\varphi}_{1,\alpha}^{2,\alpha}\overrightarrow{\varphi}_{0,\alpha}^{1,\alpha}$.  Furthermore, the sign is determined by the orientation placed on the planar cobordism. This is a global decision, thus $\overleftarrow{\Phi}=\pm\overrightarrow{\Phi}$.

\subsubsection{Movie Move 12}

\begin{figure}[H]
    \centering
    \includeFig{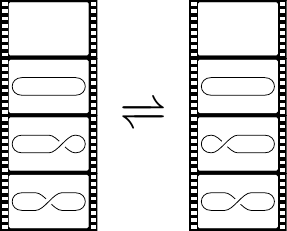}
    \caption{Movie move 12}
    \label{fig:4:MM12}
\end{figure}

Movie move 12 is built from the following elements:
\begin{itemize}
    \item A Reidemeister I cobordism
    \item A death/birth cobordism
\end{itemize}

\paragraph{Movie Move 12 Forward:}\phantom{.}

{\noindent}The left and right sides differ by what side of the circle the Reidemeister I move is carried out on.  In the forward direction, the orientations on the crossings will not be relevant, and we will choose the orientations as in equation \eqref{eq:4:MM12orientation} now for when we need them later in the reverse direction.
\begin{equation}
    \includeTang{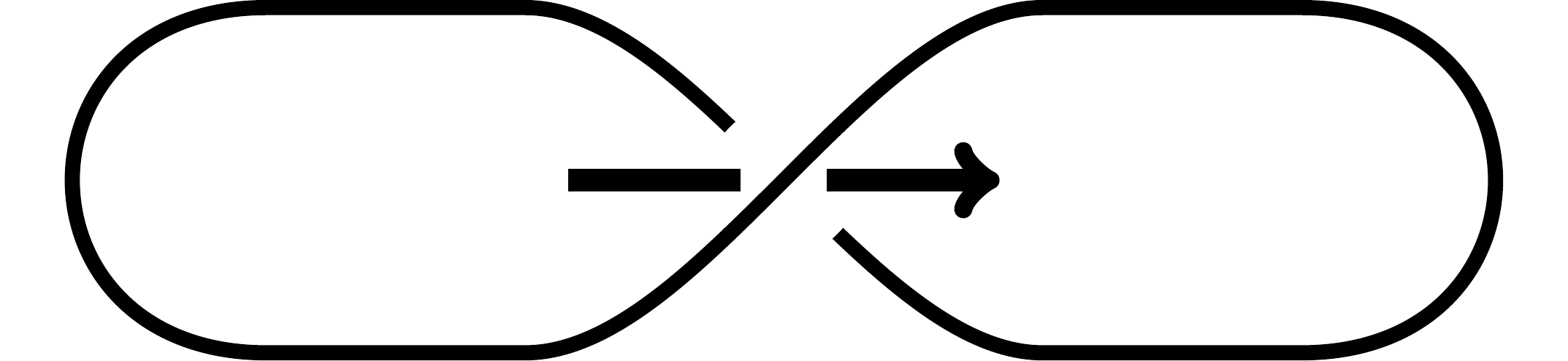}
    \label{eq:4:MM12orientation}
\end{equation}
Diagram \eqref{dig:MM12FL} shows the left-hand side of movie move 12 in the forward direction and diagram \eqref{dig:MM12FR} shows the right-hand side of movie move 12 in the forward direction.

\begin{equation}
    \label{dig:MM12FL}
    \includeFig{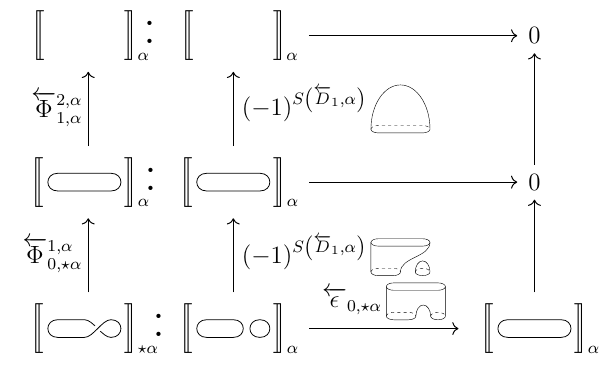}
\end{equation}
\begin{equation}
    \includeFig{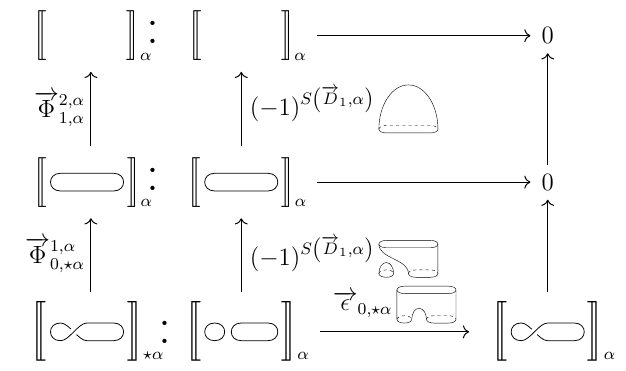}
    \label{dig:MM12FR}
\end{equation}

All of the signs appear squared, thus, for each $\alpha$ we have the following equation:
\begin{equation}
    \includeFig{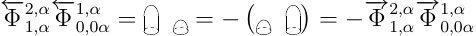}
\end{equation}
It follows for the entire chain maps that $\overleftarrow{\Phi}=-\overrightarrow{\Phi}$.

\paragraph{Movie Move 12 Reverse:}\phantom{.}

{\noindent}Diagram \eqref{dig:MM12RL} shows the left-hand side of movie move 12 in the reverse direction and diagram \eqref{dig:MM12RR} shows the right-hand side of movie move 12 in the reverse direction.

\begin{equation}
    \label{dig:MM12RL}
    \includeFig{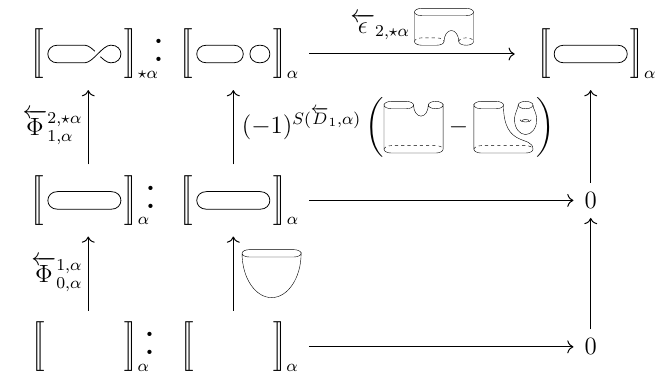}
\end{equation}
\begin{equation}
    \includeFig{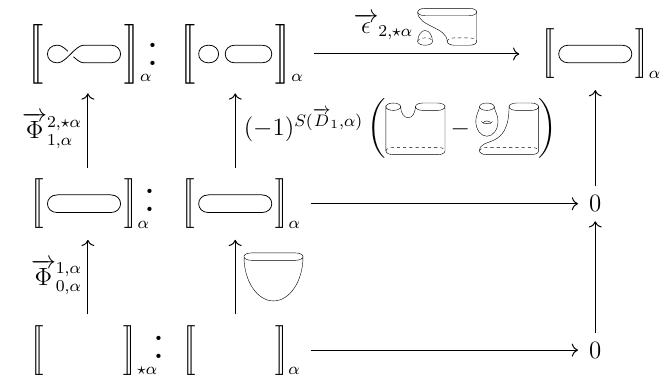}
    \label{dig:MM12RR}
\end{equation}

The orientation on the crossing endows the splits on both sides with canonical orientations on splits. Note that $\overleftarrow{D}_{1,\alpha}$ and $\overrightarrow{D}_{1,\alpha}$ are isotopic and thus produce the same value on the $S$-map.  Thus, by employing the (2H) relation and the variant of the (4Tu) relation in equation \eqref{eq:A:11}, for each $\alpha$ we have the following equation:
\begin{equation}
    \hspace{-0.5em}\includeFig{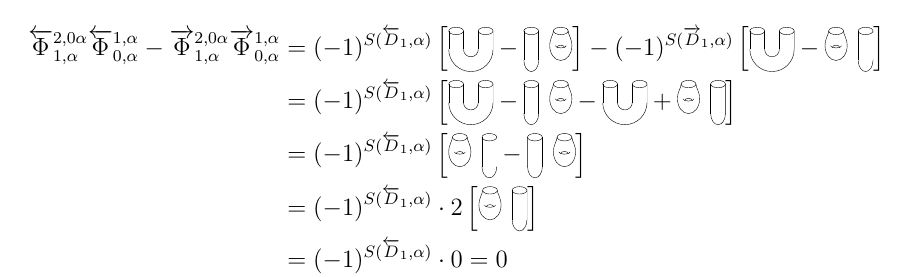}
\end{equation}
It follows for the entire chain maps that $\overleftarrow{\Phi}=\overrightarrow{\Phi}$.

\paragraph{Movie Move 12 Alternative Variant:}\phantom{.}

{\noindent}There is an alternative version of movie move 12 in which the crossing is negative as opposed to the positive one presented.  In that case, for both the forward and reverse directions, the chain map assigned to the Reidemeister I move is supported over the 1-resolution of the crossing. The computations are almost the same as for the right-handed crossing, except that now all planar cobordisms are flipped upside down, and the roles of the forward and reverse directions are exchanged.

\subsubsection{Movie Move 13}

\begin{figure}[H]
    \centering
    \includeFig{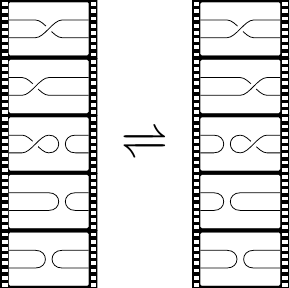}
    \caption{Movie move 13}
    \label{fig:4:MM13}
\end{figure}

Movie move 13 is built from the following elements:
\begin{itemize}
    \item A Reidemeister I cobordism
    \item A saddle cobordism
\end{itemize}

\paragraph{Movie Move 13 Forward:}\phantom{.}

{\noindent}The resolution of the right crossing of the larger link that appears in equation \eqref{fig:MM13FEQ} is present in the transition from the middle to its subsequent frame in the left side of the movie move.  For the right side of the movie move, the corresponding frames also present a resolution of the same larger link, but now of the left crossing.

\begin{equation}
    \label{fig:MM13FEQ}
    \includeFig{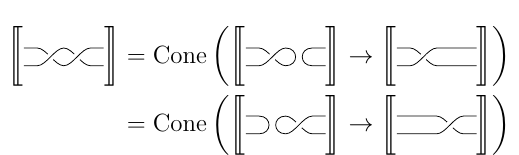}
\end{equation}

We will orient the crossings as in the following diagram. This produces canonical orientations on all saddles.

\begin{equation}
    \includeTang{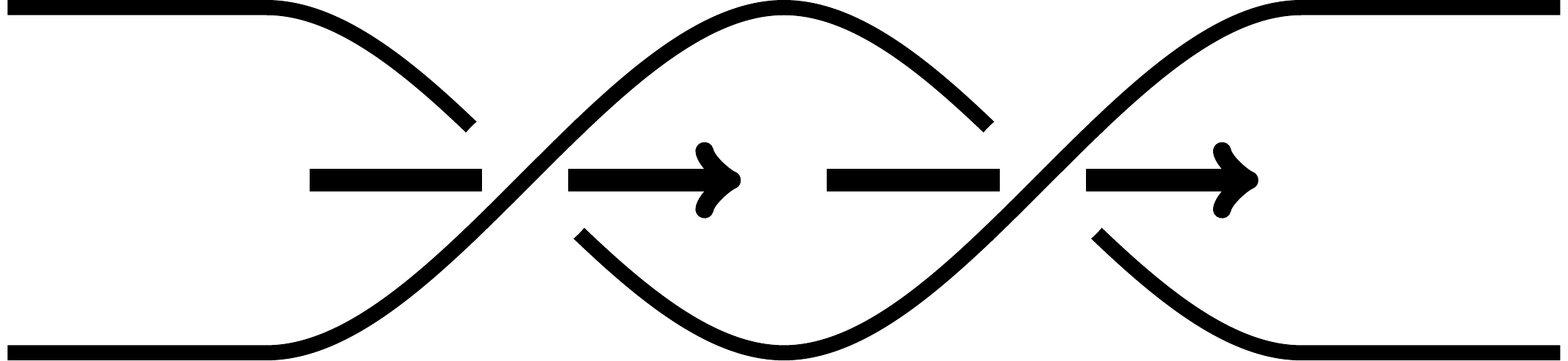}
\end{equation}

Consider the odd Khovanov cube of the larger link in diagram \eqref{fig:MM13Cube} with signs $\psi_1$, $\psi_2$, $\psi_3$, and $\psi_4$ on the edges.  Instead of invoking the existence of maps to build the saddle chain map for the movie move, we will directly construct valid sign assignments for each side from the cube of the larger link, so that the vertical maps (which correspond to the saddle chain maps) have nice relations, and the horizontal sign assignments are the same where the horizontal sign assignments correspond to the assignments on the source and the target of the saddle chain maps.

\begin{equation}
    \includeFig{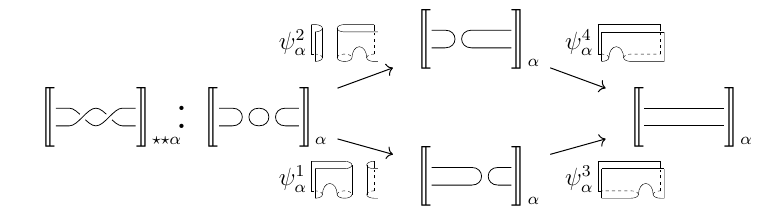}
    \label{fig:MM13Cube}
\end{equation}

Consider the square corresponding to the transition from the second to the third frame of the movie.  The diagrams in corresponding vertices of the left-hand and right-hand cubes are ambient isotopic, so we can endow them with identical internal sign assignments. Here, we use the convention that among the two vertices in the middle column of \eqref{fig:MM13Cube}, the top vertex in the right-hand cube corresponds to the bottom vertex in the left-hand cube, and vice versa.

Since the left pair of maps in \eqref{fig:MM13Cube} are both merges, the faces that can occur by pairing one of these maps with a map from an outside crossing will be either type i, ii, iv, or v. This means they are commuting faces, and thus they can share valid sign assignments.

The second pair of maps in \eqref{fig:MM13Cube} are built from essentially the same saddle cobordisms. If $\psi^1_\alpha,\psi^2_\alpha,\psi^3_\alpha,\psi^4_\alpha$ are part of a valid sign assignment for the left-hand cube, it thus follows that exchanging $\psi^1_\alpha$ with $\psi^2_\alpha$ and $\psi^3_\alpha$ with $\psi^4_\alpha$ gives rise to a valid sign assignment for the right-hand cube.

Note also that in the latter cube the roles of the vertical and horizontal maps in \eqref{fig:MM13Cube} are reversed. Therefore, we can make the following assignments where we sprinkled in terms involving $\deg(\alpha)$ to ensure that the vertical maps commute with the differentials:

\begin{equation}
\begin{array}{c}
\overleftarrow{\epsilon}_{1,{\star}\alpha}=\overrightarrow{\epsilon}_{1,{\star}\alpha}=\psi_\alpha^1\\
\overleftarrow{\epsilon}_{1,0\alpha}^{2,0\alpha}=\overrightarrow{\epsilon}_{1,0\alpha}^{2,0\alpha}=(-1)^\alpha\psi_\alpha^2\\
\overleftarrow{\epsilon}_{1,1\alpha}^{2,1\alpha}=\overrightarrow{\epsilon}_{1,1\alpha}^{2,1\alpha}=(-1)^{\alpha+1}\psi_\alpha^3\\
\overleftarrow{\epsilon}_{2,{\star}\alpha}=\overrightarrow{\epsilon}_{2,{\star}\alpha}=\psi_\alpha^4
\end{array}
\end{equation}

We arrive at the following diagrams.

\begin{equation}
    \includeFig{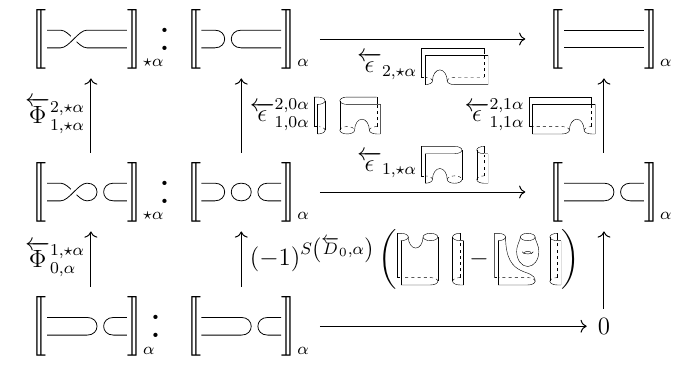}
    \label{fig:MM13FL}
\end{equation}

\begin{equation}
    \includeFig{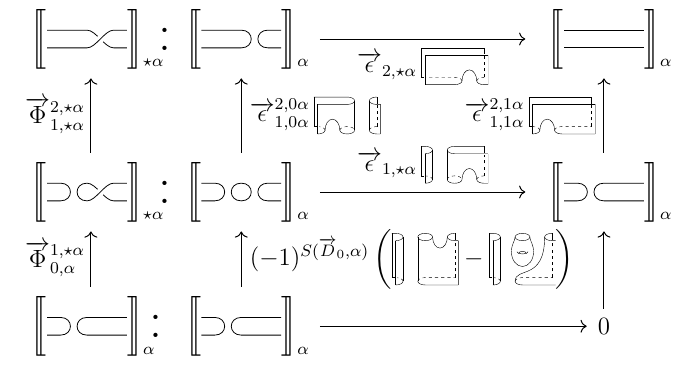}
    \label{fig:MM13FR}
\end{equation}

As in movie move 12 it is the case that $\overleftarrow{D}_{0,\alpha}$ and $\overrightarrow{D}_{0,\alpha}$ are isotopic and thus produce the same values on the $S$-map. Using the (4Tu) and (2H) relations, we arrive at the following equation for a fixed degree $\alpha$

\begin{equation}
    \hspace{-0.125in}\includeFig{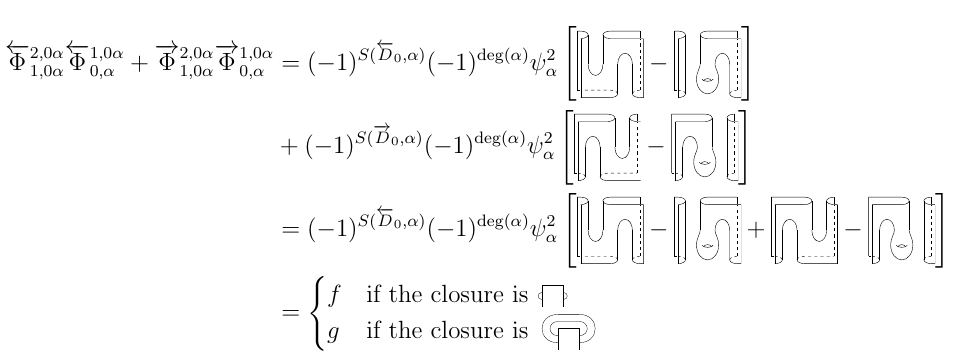}\hspace{-0.25in}
\end{equation}

where

\begin{equation}
    \includeFig{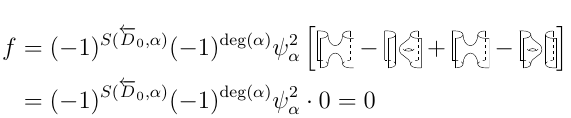}
\end{equation}

and

\begin{equation}
    \includeFig{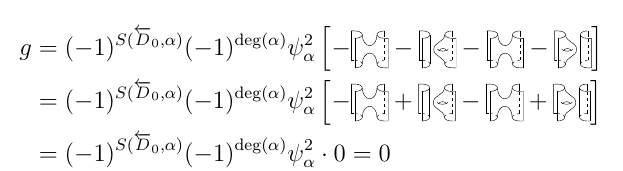}
\end{equation}

We can define the overall sign assignment as we did for the fixed degree $\alpha$, such that $\overleftarrow{\Phi}=-\overrightarrow{\Phi}$.

\paragraph{Movie Move 13 Reverse:}\phantom{.}

{\noindent}In the reverse direction, we can still view the left and right movies as portions of larger link diagrams, but not the same link diagram:

\begin{equation}
    \includeFig{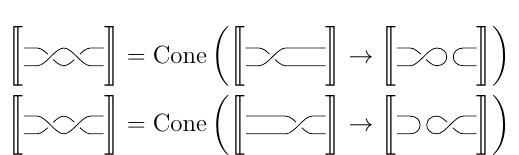}
    \label{fig:MM13REQ}
\end{equation}

We will use the following crossing orientation so that both sides of the movie feature a type x face and the split cobordisms have canonical orientations.

\begin{equation}
    \includeFig{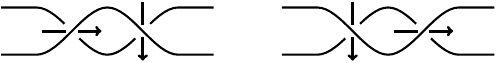}
\end{equation}

As in the forward direction, corresponding vertices feature ambient isotopic diagrams. So we can share internal sign assignments between them.  We will construct a valid sign assignment for the right-hand side using the sign assignment from the left-hand side. 

The cobordisms $\overleftarrow{\varphi}_{0,{\star}\alpha}$ and $\overrightarrow{\varphi}_{0,{\star}\alpha}$ in the bottom rows of \eqref{fig:MM13RL} and \eqref{fig:MM13RR} are essentially the same. Likewise, $\overleftarrow{\varphi}_{0,1\alpha}^{1,1\alpha}$ and $\overrightarrow{\varphi}_{0,1\alpha}^{1,1\alpha}$ are essentially the same.  

The cobordisms $\overleftarrow{\varphi}_{1,{\star}\alpha}$ and $\overrightarrow{\varphi}_{1,{\star}\alpha}$ are not identical, but together with maps from external crossings, they form faces that are either type i, ii, iv, or v, and are thus always commuting, and therefore can share sign assignments.  Similarly, in relation to maps that come from a given external crossing, $\overleftarrow{\varphi}_{1,{\star}\alpha}$ and $\overrightarrow{\varphi}_{1,{\star}\alpha}$ are either both commuting---forming a type iv or v face--or both anticommuting, forming a type vii or viii face.

\begin{equation}
    \includeFig{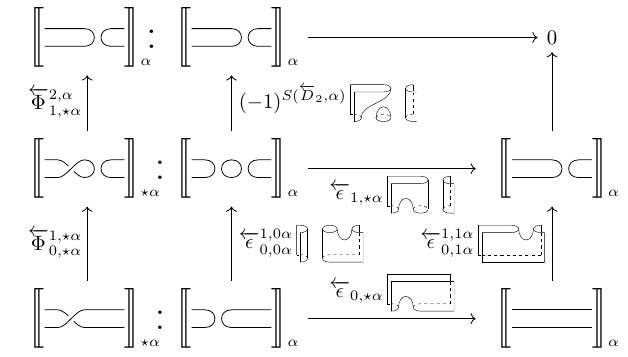}
    \label{fig:MM13RL}
\end{equation}

\begin{equation}
    \includeFig{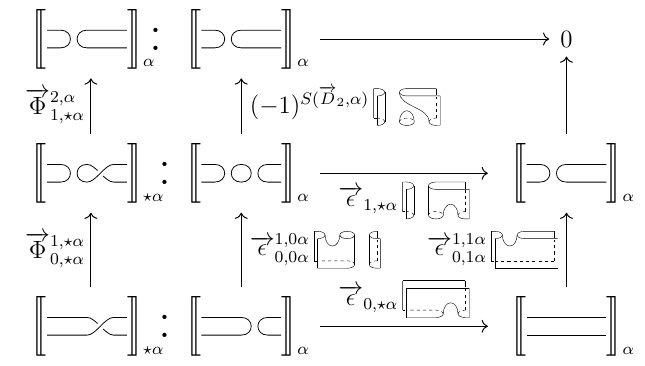}
    \label{fig:MM13RR}
\end{equation}

As all the maps of the square on the left are either essentially identical to their corresponding maps on the right, or at least induce the same sign on faces formed with maps from external crossings, we can build a valid sign assignment on the right-hand side with the same signs.  That is, we can use the following sign assignment for the right-hand side of the move.

\begin{equation}
\begin{split}
\overrightarrow{\epsilon}_{1,{\star}\alpha} &= \overleftarrow{\epsilon}_{1,{\star}\alpha} \\
\overrightarrow{\epsilon}_{1,0\alpha}^{2,0\alpha} &= \overleftarrow{\epsilon}_{1,0\alpha}^{2,0\alpha} \\
\overrightarrow{\epsilon}_{1,1\alpha}^{2,1\alpha} &= \overleftarrow{\epsilon}_{1,1\alpha}^{2,1\alpha} \\
\overrightarrow{\epsilon}_{2,{\star}\alpha} &= \overleftarrow{\epsilon}_{2,{\star}\alpha} \\
\end{split}
\end{equation}

This in turn yields the following equation.

\begin{equation}
    \includeFig{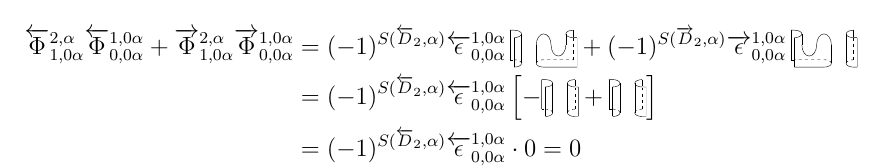}
\end{equation}

The sign assignment propagates to the entire cube, thus $\overleftarrow{\Phi}=-\overrightarrow{\Phi}$.

\paragraph{Movie Move 13 Alternative Variants:}\phantom{.}

{\noindent}There are additional variants of this move that must be considered.  The move is built off of a saddle and a Reidemeister I move, so we need to consider the movie moves shown in \eqref{fig:MM13AF} and \eqref{fig:MM13AR}, where negative crossings appear in place of the positive crossings.  
\begin{equation}
    \includeFig{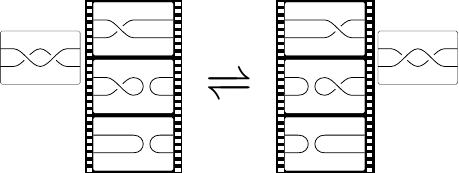}
    \label{fig:MM13AF}
\end{equation}
\begin{equation}
    \includeFig{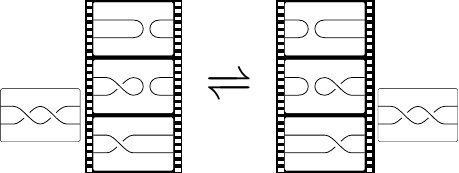}
    \label{fig:MM13AR}
\end{equation}
The arguments will still work as before, but now the maps will be supported in degree one, not zero.  In the forward direction, the larger crossings resemble the two versions of the Reidemeister II move, which were present in the original reverse direction, therefore we use a similar argument to the original reverse direction.  Likewise, in the reverse direction, the larger link closely resembles that from the original forward direction, but with opposite crossings. The argument thus follows that of the original, forward direction.

\subsubsection{Movie Move 14}

\begin{figure}[H]
    \centering
    \includeFig{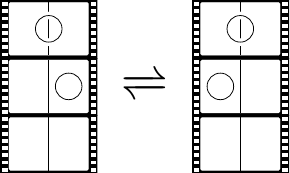}
    \caption{Movie move 14}
    \label{fig:4:MM14}
\end{figure}

Movie move 14 is built from the following elements:
\begin{itemize}
    \item A birth/death cobordism.
    \item A Reidemeister II cobordism.
\end{itemize}
For our computations, we will use a version of this move where all frames are rotated $90^\circ$ clockwise.

\paragraph{Movie Move 14 Forward:}\phantom{.}

{\noindent}The final diagram of both sides decomposes into the following square with the same shared sign assignment on both sides of the move:
\begin{equation}
    \includeFig{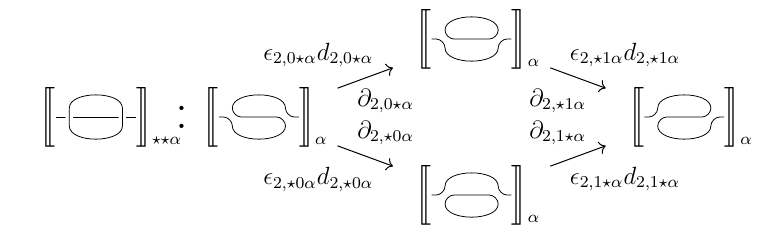}
    \label{fig:MM14topsquare}
\end{equation}
We will endow the crossings with the following orientations so that the top square forms a type x face.
\begin{equation}
    \includeTang{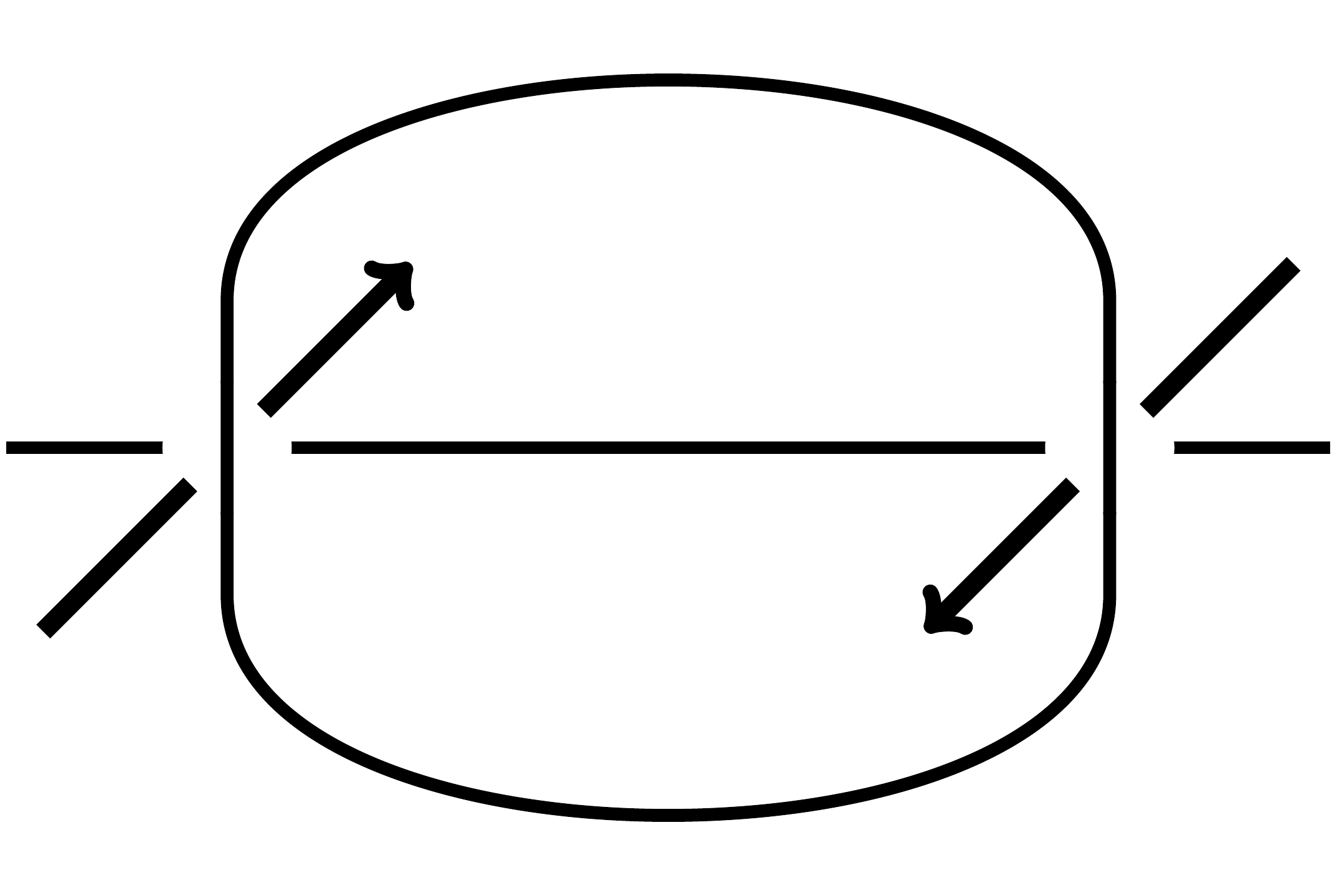}
    \label{dig:MM14squareOrient}
\end{equation}
We will further equip the two vertices in the middle of \eqref{fig:MM14topsquare} with the same internal sign assignments. Note that this is possible because of Lemma~\ref{lem:3:CWsubcomplex}, and because the diagrams corresponding to these vertices represent ambient isotopic tangles up to changing the location of the trivial circle.

As each side of the move begins with no crossings in the tangle, the chain maps assigned to the two sides of the move are supported only in the zero degree.  We will consider the diagrams of each side of the move restricted to the relevant degree.

\begin{equation}
    \includeFig{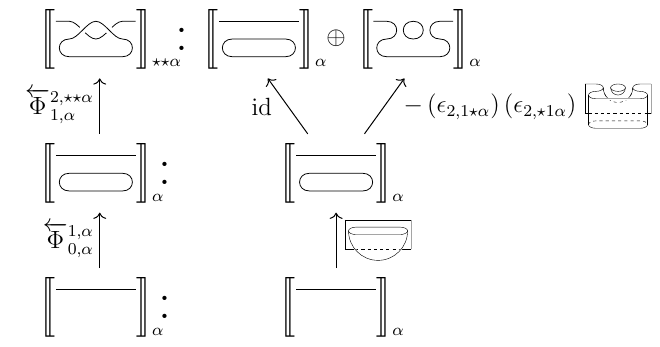}
    \label{fig:MM14FL}
\end{equation}

\begin{equation}
    \includeFig{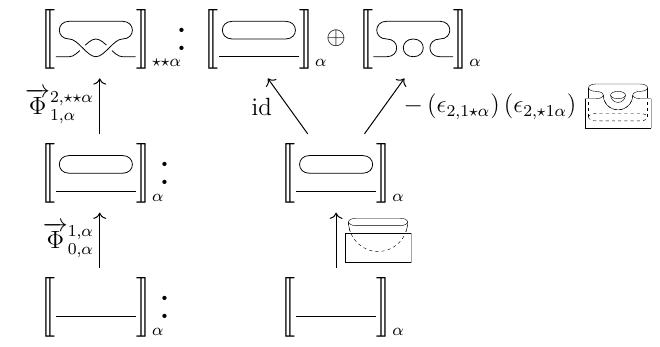}
    \label{fig:MM1FR4}
\end{equation}

Using our orientation conventions, the Reidemeister II chain maps naturally induce the canonical orientation on the saddle on the left-hand side of the move, while they generate a non-canonically oriented saddle on the right-hand side.  We can correct this orientation on the saddle at no cost, as the saddle is always a merge cobordism.  Note that the two supporting maps in each diagram simplify in the following manner.

\vspace{1em}
{\noindent}\begin{minipage}{0.5\textwidth}
    \begin{equation}
    \includeFig{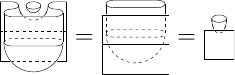}
    \end{equation}
\end{minipage}%
\begin{minipage}{0.5\textwidth}
    \begin{equation}
    \includeFig{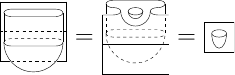}
    \end{equation}
\end{minipage}
\vspace{1em}

Thus, for each $\alpha$ we have the following equation.

\begin{equation}
    \includeFig{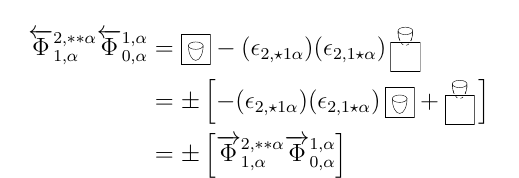}
\end{equation}

We still need the maps to be consistently equal or negative of one another. Our goal will be to show that we can choose
\begin{equation}
\begin{split}
    \epsilon_{2,0{\star}\alpha}&=\epsilon_{2,{\star}0\alpha} \\
    \epsilon_{2,1{\star}\alpha}&=\epsilon_{2,{\star}1\alpha}  
\end{split}
\label{eqn:foursigns}
\end{equation}
for all $\alpha$, so that $\overleftarrow{\Phi}=-\overrightarrow{\Phi}$.

To see this, we first note that the two right maps in \eqref{fig:MM14topsquare} can share sign assignments because they are always merges, and thus form commuting faces with maps coming from external crossings. For similar reasons, the two left maps in \eqref{fig:MM14topsquare} can share sign assignments because they are always splits. Finally, the equations in \eqref{eqn:foursigns} imply that for any fixed $\alpha$, the four signs multiply to $+1$, which is consistent with the fact that the square in \eqref{fig:MM14topsquare} is a type x face and thus anticommutes.

\paragraph{Movie Move 14 Reverse:}\phantom{.}

{\noindent}In the reverse direction, an almost identical argument applies, but there is one additional consideration.

\vspace{1em}
{\noindent}\begin{minipage}{0.5\textwidth}
    \begin{equation}
    \includeFig{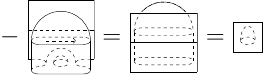}
    \end{equation}
\end{minipage}%
\begin{minipage}{0.5\textwidth}
    \begin{equation}
    \includeFig{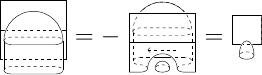}
    \end{equation}
\end{minipage}
\vspace{1em}

In the above equations, the simplifications involve eliminating a split followed by a death, which can induce a sign.  These signs do not impose any issues for our argument as a sign is incurred on both sides of the move.  All other parts of the argument in the forward direction apply to the reverse direction.

\paragraph{Movie Move 14 Alternative Variant:}\phantom{.}

{\noindent}There is an additional variant of movie move 14 in which the circle passes under the strand instead of over. The argument for the initial variant also applies in this alternative setting.

\subsubsection{Movie Move 15}

\begin{figure}[H]
    \centering
    \includeFig{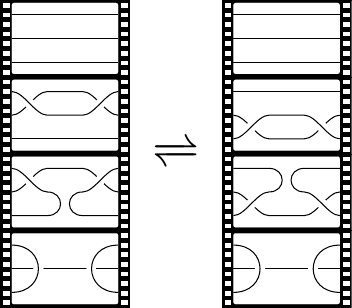}
    \caption{Movie move 15}
    \label{fig:4:MM15}
\end{figure}

The two sides of movie move 15 precisely give rise to the chain maps whose cones we considered in the proof of invariance under Reidemeister III moves. Thus, invariance under movie move 15 follows from Remark~\ref{rem:3:conemaps}.

Note that there are two versions of the Reidemeister III move, one with the strand passing under an overcrossing, and one with the strand passing under an undercrossing.  Reading movie move 15 in the forward direction corresponds with the overcrossing, and the undercrossing corresponds with reading the movie in reverse.  Invariance with respect to both versions of the move can be shown in the same way.

\subsection{Chronological Movie Moves}

The link cobordism category allows ambient isotopies that produce changes in the chronology of the planar cobordisms that appear in the induced chain maps.  Therefore, in order to show the functoriality of odd Khovanov homology up to sign, we must show that exchanging the order of pairs of distant link cobordisms at most induces an overall change in sign.\par

\subsubsection{Commuting Births with Death Cobordisms}

Birth planar cobordisms commute with death cobordisms, but the sign on the death cobordism chain map is determined by the number of circles in the diagram.  The underlying cobordisms are equal, but a sign is induced in the chain maps.  Therefore, odd Khovanov bracket induces a sign when commuting birth link cobordisms with death link cobordisms.

\subsubsection{Commuting Births with Birth, Saddle, or Reidemeister II and III Type Cobordisms}

Birth planar cobordisms commute with all elementary planar cobordisms. Furthermore, the signs that are attached to all cobordisms which are not deaths do not depend on the number of circles in the diagram.  Therefore, odd Khovanov bracket respects commuting birth link cobordisms with all elementary link cobordisms except deaths.

\subsubsection{Commuting Births with Reidemeister I Type Cobordisms}

The signs attached to deaths are also attached to positive Reidemeister I type cobordisms. Therefore, the odd Khovanov bracket up to sign respects commuting births and positive Reidemeister I type cobordisms, as it did with births and deaths.  The negative Reidemeister I type cobordism chain map does not have the same correcting signs, thus births and such cobordisms commute without incurring a sign, as was the case when commuting a birth cobordism with another birth cobordism, a saddle, or a Reidemeister II or III type cobordism.

\subsubsection{Commuting a Pair of Death Cobordisms}

Exchanging the chronology of a pair of death link cobordisms induces a sign. This is sufficient for the odd Khovanov bracket to respect commuting a pair of death link cobordisms.

\subsubsection{Commuting Deaths with Saddle Cobordisms}\label{sec:4:deathsaddle}

We will show in Lemma \ref{lem:deathsadddle} that global factors control whether or not exchanging the order of a death link cobordism and a general saddle induces a sign on the odd Khovanov bracket's chain map.  As a consequence, the exchange introduces at most an overall sign, and the odd Khovanov bracket respects the exchange up to sign.

\begin{lemma}\label{lem:deathsadddle}
    Changing the chronological order of a death and a saddle does not introduce a sign if the resolution of the saddle increases the number of circles in the zero resolution, and introduces a sign otherwise.
\end{lemma}

\begin{proof}
We have essentially the following movie move.

\begin{figure}[H]
    \centering
    \includeFig{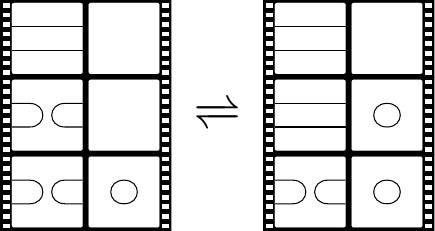}
    \caption{Death and saddle chronological movie move}
    \label{fig:MMDS}
\end{figure}

We get the following chain map in each degree $\alpha$.

\vspace{1em}
{\noindent}\begin{minipage}{0.5\textwidth}
    \begin{equation}
        \includeFig{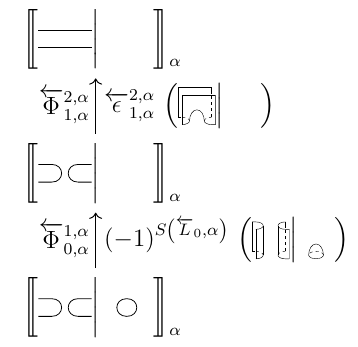}
    \end{equation}
    \end{minipage}%
    \begin{minipage}{0.5\textwidth}
    \begin{equation}
        \includeFig{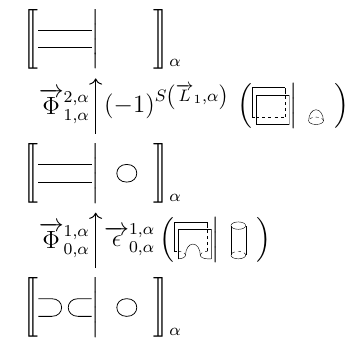}
    \end{equation}
    \end{minipage}
\vspace{1em}
    
Recall that the chain map for a saddle was produced by considering a larger link $L'$.  Once we specify an $\alpha$, the saddle is either a split or a merge.  If the saddle is a split, then the underlying cobordisms anticommute 
$\left(\overleftarrow{\varphi}_{1,\alpha}^{2,\alpha}\overleftarrow{\varphi}_{0,\alpha}^{1,\alpha}=-\overrightarrow{\varphi}_{1,\alpha}^{2,\alpha}\overrightarrow{\varphi}_{0,\alpha}^{1,\alpha}\right)$.
Otherwise, the saddle is a merge and the underlying cobordisms commute
$\left(\overleftarrow{\varphi}_{1,\alpha}^{2,\alpha}\overleftarrow{\varphi}_{0,\alpha}^{1,\alpha}=\overrightarrow{\varphi}_{1,\alpha}^{2,\alpha}\overrightarrow{\varphi}_{0,\alpha}^{1,\alpha}\right)$. 
Note that the saddle maps are built off of the same cube, and the only final map is produced by introducing signs in odd degrees, thus 
$\overleftarrow{\epsilon}_{1,\alpha}^{2,\alpha}=\overrightarrow{\epsilon}_{0,\alpha}^{1,\alpha}$.  To see the change in the other signs, we must consider how the $S$-map that defined the signs on deaths changes when the saddle is a merge or a split with respect to $\alpha$ and the all-zero resolutions.
     
\begin{center}
\begin{tabular}{ c | c | c }
    $S(L_1,\alpha)-S(L_0,\alpha)$ & $C(L',1\alpha)=C(L',0\alpha)+1$ & $C(L',1\alpha)=C(L',0\alpha)-1$ \\ 
    \hline
    $C(L',1\zeta)=C(L',0\zeta)+1$ & 1 & 0 \\ 
    \hline
    $C(L',1\zeta)=C(L',0\zeta)-1$ & 0 & 1    
\end{tabular}
\end{center}

    From the map assigned to deaths, $\overrightarrow{\epsilon}_{1,\alpha}^{2,\alpha}=(-1)^{S(L_1,\alpha)-S(L_0,\alpha)}\overleftarrow{\epsilon}_{0,\alpha}^{1,\alpha}$.  From these relationships it is clear that regardless of a choice of $\alpha$, if a saddle is a split with respect to the zero resolution, $\overleftarrow{\Phi}_{1,\alpha}^{2,\alpha}\overleftarrow{\Phi}_{0,\alpha}^{1,\alpha}=\overrightarrow{\Phi}_{1,\alpha}^{2,\alpha}\overrightarrow{\Phi}_{0,\alpha}^{1,\alpha}$. Thus, we can say deaths commute with such saddles.  If a saddle is a merge with respect to the zero resolution, $\overleftarrow{\Phi}_{1,\alpha}^{2,\alpha}\overleftarrow{\Phi}_{0,\alpha}^{1,\alpha}=-\overrightarrow{\Phi}_{1,\alpha}^{2,\alpha}\overrightarrow{\Phi}_{0,\alpha}^{1,\alpha}$. Thus, we can say deaths anticommute with such saddles.
\end{proof}

\subsubsection{Commuting Deaths with Reidemeister I Type Cobordisms}

There are four versions of the Reidemeister I type cobordism, each of which produces homotopic chain maps up to sign when chronologically rearranged with a death.  Only the Reidemeister I cobordism corresponding with a negative twist commutes with deaths, while all other variants induce an overall sign.

\paragraph{Commuting Deaths with Positive Reidemeister I Type Cobordisms}\phantom{.}

{\noindent}Consider the nonzero portion of the following chain maps. The chain map $\overleftarrow{\Phi}$ is induced by a positive Reidemeister I type cobordism followed by a death cobordism. The chain map $\overrightarrow{\Phi}$ is induced by a death cobordism followed by a positive Reidemeister I type cobordism.

\begin{equation}
    \includeFig{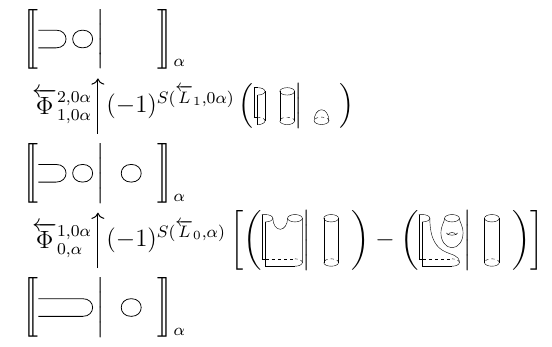}
\end{equation}

\begin{equation}
    \includeFig{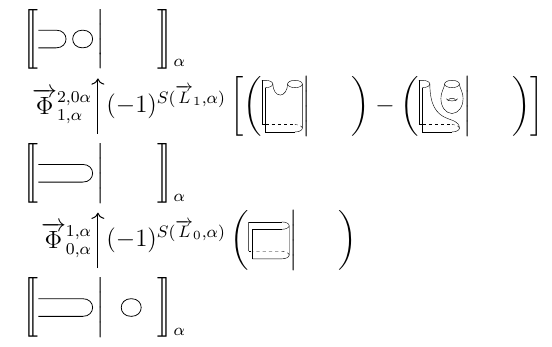}
\end{equation}

The underlying cobordisms in the above chain maps anticommute as the map induced by the Reidemeister type cobordism contains a difference of cobordisms, each of which contains a single odd elementary cobordism when factored.  When building the first sign for both $\overleftarrow{\Phi}$ and $\overrightarrow{\Phi}$ the $S$-map is passed the same diagram and index, producing the same sign.  For the second pair of maps, the degree of the index passed to $S$ is the same, but one of the diagrams, and the associated zero resolution of that same diagram, has an additional two circles, which leads to $S$ producing the same sign.  Thus, in total $\overleftarrow{\Phi}=-\overrightarrow{\Phi}$.

\paragraph{Commuting Deaths with Inverse Positive Reidemeister I Type Cobordisms}\phantom{.}

{\noindent}Consider the nonzero portion of the following chain maps. The chain map $\overleftarrow{\Phi}$ is induced by an inverse positive Reidemeister I type cobordism followed by a death cobordism. The chain map $\overrightarrow{\Phi}$ is induced by a death cobordism followed by an inverse positive Reidemeister I type cobordism.

\vspace{1em}
{\noindent}\begin{minipage}{0.5\textwidth}
    \begin{equation}
        \includeFig{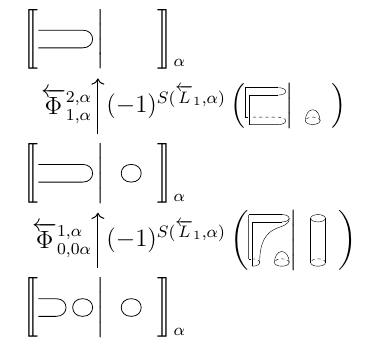}
    \end{equation}
    \end{minipage}%
    \begin{minipage}{0.5\textwidth}
    \begin{equation}
        \includeFig{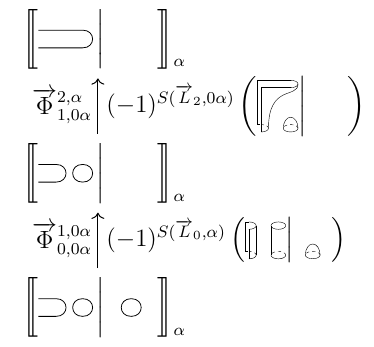}
    \end{equation}
    \end{minipage}
\vspace{1em}

The underlying cobordisms in the above chain maps anticommute as the map induced by the Reidemeister type cobordism contains a single odd elementary cobordism.  For the left-hand side, the signs cancel as the same sign is applied to both terms. For the right-hand side, both signs the index passed to the $S$-map are the same, and as in the prior case there are two fewer circles in both the diagram and the zero resolution for the second map. Therefore, the signs also cancel with each other.  Thus, in total $\overleftarrow{\Phi}=-\overrightarrow{\Phi}$.

\paragraph{Commuting Deaths with Negative Reidemeister I Type Cobordisms}\phantom{.}

{\noindent}Consider the nonzero portion of the following chain maps. The chain map $\overleftarrow{\Phi}$ is induced by a negative Reidemeister I type cobordism followed by a death cobordism. The chain map $\overrightarrow{\Phi}$ is induced by a death cobordism followed by a negative Reidemeister I type cobordism.

\vspace{1em}
{\noindent}\begin{minipage}{0.5\textwidth}
    \begin{equation}
        \includeFig{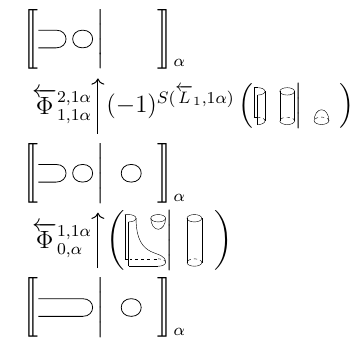}
    \end{equation}
    \end{minipage}%
    \begin{minipage}{0.5\textwidth}
    \begin{equation}
        \includeFig{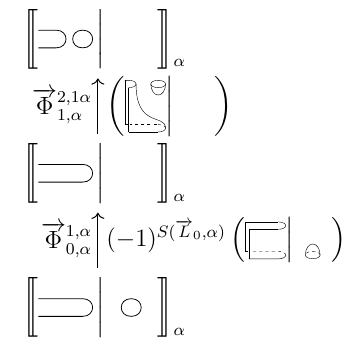}
    \end{equation}
    \end{minipage}
\vspace{1em}

The underlying cobordisms in the above chain maps commute as the map induced by the Reidemeister type cobordism does not contain any odd elementary cobordism when factored.  The sign for the left chain map is based on a diagram with the same number of circles as the diagram used for the sign on the right, but the degree for the sign on the left is one less than the one on the right. Furthermore, the number of circles in the zero resolution of the diagram on the left is one greater than the number of circles in the zero resolution of the diagram on the right, canceling out and producing the same sign from the $S$ map on both the left and right-hand sides.  Thus, in total $\overleftarrow{\Phi}=\overrightarrow{\Phi}$.

\paragraph{Commuting Deaths with Inverse Negative Reidemeister I Type Cobordisms}\phantom{.}

{\noindent}Consider the nonzero portion of the following chain maps. The chain map $\overleftarrow{\Phi}$ is induced by an inverse negative Reidemeister I type cobordism followed by a death cobordism. The chain map $\overrightarrow{\Phi}$ is induced by a death cobordism followed by an inverse negative Reidemeister I type cobordism.

\vspace{1em}
{\noindent}\begin{minipage}{0.5\textwidth}
    \begin{equation}
        \includeFig{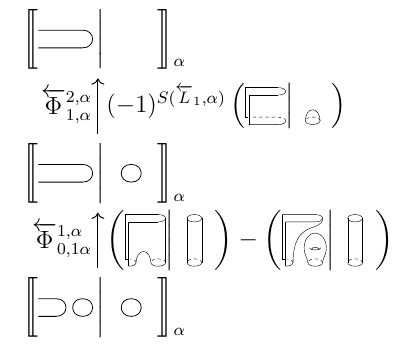}
    \end{equation}
    \end{minipage}%
    \begin{minipage}{0.5\textwidth}
    \begin{equation}
        \includeFig{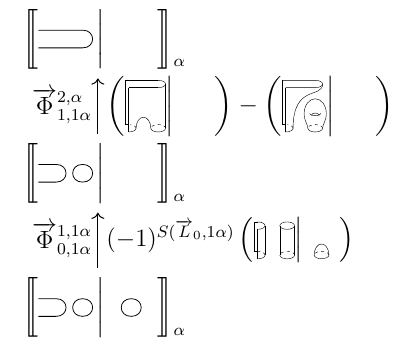}
    \end{equation}
    \end{minipage}
\vspace{1em}

The underlying cobordisms in the above chain maps commute as the map induced by the Reidemeister type cobordism contains a difference of cobordisms, each of which contains an even number of odd elementary cobordism when factored.  The sign for the left chain map is based on a diagram with the same number of circles in its zero resolution as the zero resolution of the diagram on the right, but the degree for the sign on the left is one less than the one on the right. Furthermore, the number of circles in the diagram on the left is one fewer than the number of circles in the zero resolution of the diagram on the right, thus the signs on the two sides are not equal.  Thus, in total $\overleftarrow{\Phi}=-\overrightarrow{\Phi}$.

\subsubsection{Commuting Deaths with Reidemeister II or III Type Cobordisms}

We will show that exchanging the chronology of a death link cobordism and a Reidemeister II or III type cobordism does not change the induced chain map.  

\begin{lemma}\label{lem:deathRII}
    The chain map associated with a death link cobordism commutes with the planar cobordism chain map that consists of a general saddle and a death, or a general saddle and a birth and a sign dependent on the resolution $\alpha$.
\end{lemma}
\begin{proof}
    Fix a vertex $\alpha$. For the death and saddle version we then have the following maps:
\vspace{1em}
{\noindent}\begin{minipage}{0.5\textwidth}
    \begin{equation}
        \includeFig{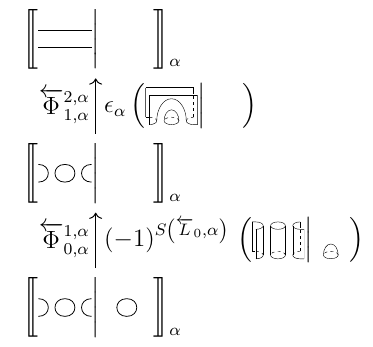}
    \end{equation}
\end{minipage}%
\begin{minipage}{0.5\textwidth}
    \begin{equation}
        \includeFig{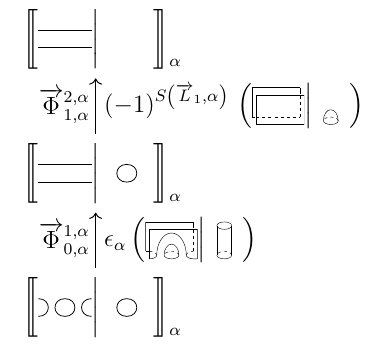}
    \end{equation}
\end{minipage}
\vspace{1em}

If the saddle is a merge, then the underlying cobordisms anticommute as the death moves past a death. Otherwise, the saddle is a split, and the underlying cobordisms commute as a sign is incurred for moving the death past both the split and the death, resulting in no net sign change.  If the saddle is a merge then, after the cobordism, the number of circles is decreased by two as the merge and the death each decrement the number of circles. Meanwhile, if the saddle is a split, the number of circles is maintained.  It follows that the death cobordism introduces a sign only if the saddle is a merge. Therefore, the sign changes cancel out, and the overall cobordisms commute.  A very similar argument shows that deaths would commute with the saddle and the birth map, which is our particular cobordism played in reverse.
\end{proof}

Reidemeister II and its inverse chain maps are supported by identity maps and maps of the variety in Lemma \ref{lem:deathRII}. Thus, deaths commute with Reidemeister II and inverse Reidemeister II cobordisms.  Reidemeister III and its inverse's chain maps are given by identity cobordisms and cobordisms that are either the variety from Lemma \ref{lem:deathRII}, or a composition of two such cobordisms.

\subsubsection{Commuting a Pair of Saddle Cobordisms}

We will show that commuting two saddle cobordisms induces an overall sign.  Consider the odd Khovanov bracket of the larger link with two additional crossings corresponding with the two saddles.  If we fix a resolution of outside crossings, then the square corresponding with the resolution of the two distinguished vertices anticommutes:

\begin{equation}
    \Psi_{{\star}1\alpha}\Psi_{0{\star}\alpha}=-\Psi_{1{\star}\alpha}\Psi_{{\star}0\alpha}
\end{equation}

The two ways we travel around the face correspond with the rearrangement of the saddles' order.  To produce the final chain maps, we need to sprinkle in signs on odd degrees so that they commute.  

\begin{equation}
    \hspace{-0.5in}\includeFig{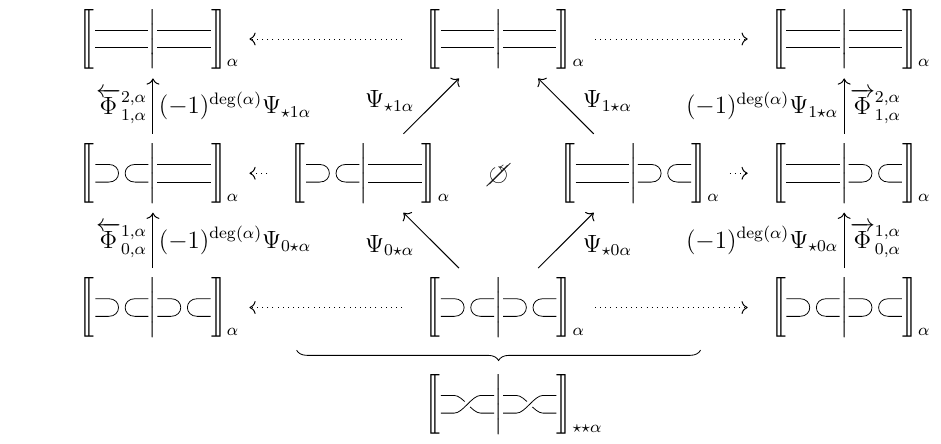}
\end{equation}

This yields the following equation.

\begin{equation}
\begin{split}
    \overleftarrow{\Phi}_{1,\alpha}^{2,\alpha}\overleftarrow{\Phi}_{0,\alpha}^{1,\alpha} &=
    (-1)^{\deg(\alpha)}\Psi_{{\star}1\alpha}(-1)^{\deg(\alpha)}\Psi_{0{\star}\alpha}\\ &=-
    \left[(-1)^{\deg(\alpha)}\Psi_{1{\star}\alpha}(-1)^{\deg(\alpha)}\Psi_{{\star}0\alpha}\right]\\ &=
    -\overrightarrow{\Phi}_{1,\alpha}^{2,\alpha}\overrightarrow{\Phi}_{0,\alpha}^{1,\alpha}
\end{split}
\end{equation}

Therefore, generally for commuting two saddle cobordisms, $\overleftarrow{\Phi}=-\overrightarrow{\Phi}$.

\subsubsection{Commuting Saddles with Reidemeister II and III Type Cobordisms}

We will show that a saddle cobordism and a Reidemeister II or III type cobordism commute.  We will follow a similar strategy as we did with a pair of saddle cobordisms.  Consider the larger link with an additional crossing corresponding with the saddle.  Then look at the odd Khovanov bracket generated by performing the Reidemeister move on a tangle $T$ yielding the tangle $T'$ of this larger link.  Fix a resolution of outside crossings and those in the tangle. We have the following complex:

\begin{equation}
    \includeFig{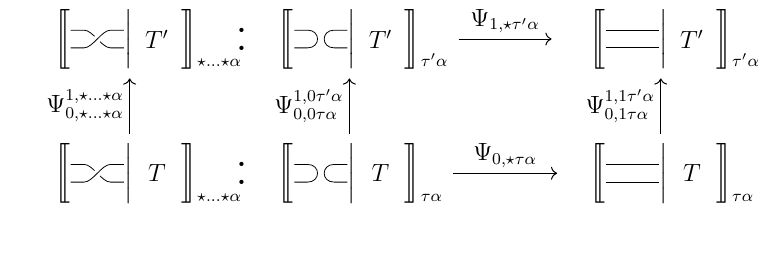}
\end{equation}

Note that many of the maps from Reidemeister type cobordisms map to and from different resolutions (this may be particularly obvious as $T$ and $T'$ are different tangles) but all the maps have zero degree or $\deg(\tau)=\deg(\tau')$.  If we sprinkle in signs to make the maps into chain maps, we arrive at the following maps:

\vspace{1em}
{\noindent}\begin{minipage}{0.5\textwidth}
    \begin{equation}
        \includeFig{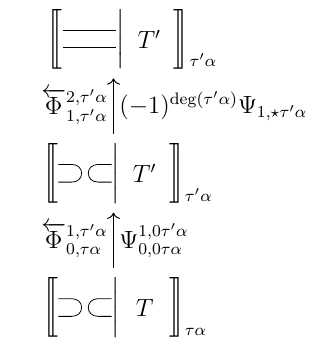}
    \end{equation}
\end{minipage}%
\begin{minipage}{0.5\textwidth}
    \begin{equation}
        \includeFig{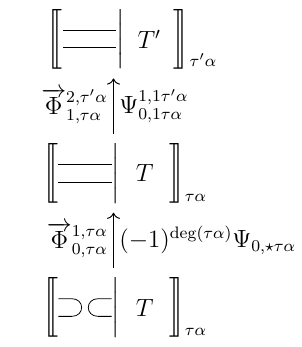}
    \end{equation}
\end{minipage}
\vspace{1em}

Using the commutativity of the initial square, we arrive at the following equation:

\begin{equation}
\begin{split}
    \overleftarrow{\Phi}_{1,\tau'\alpha}^{2,\tau'\alpha}\overleftarrow{\Phi}_{0,\tau\alpha}^{1,\tau'\alpha} &=
    (-1)^{\deg(\tau'\alpha)}\Psi_{1,{\star}\tau'\alpha}\Psi_{0,0\tau\alpha}^{1,0\tau'\alpha}\\ &=
    \Psi_{0,1\tau\alpha}^{1,1\tau'\alpha}(-1)^{\deg(\tau\alpha)}\Psi_{0,{\star}\tau'\alpha}\\ &=
    \overrightarrow{\Phi}_{1,\tau\alpha}^{2,\tau'\alpha}\overrightarrow{\Phi}_{0,\tau\alpha}^{1,\tau\alpha}
\end{split}
\end{equation}

Therefore, generally for commuting a saddle cobordism and a Reidemeister II or III type cobordism, $\overleftarrow{\Phi}=\overrightarrow{\Phi}$.

\subsubsection{Commuting Saddles with Reidemeister I Type Cobordisms}

The argument for commuting saddles with negative Reidemeister I type cobordisms is identical to the argument for commuting saddles with Reidemeister II and III type cobordisms.  With positive Reidemeister I type cobordisms, the sign will be impacted as we transition from the saddle being part of the differential to a chain map in the proof of commuting saddles and Reidemeister II and III type cobordisms.  The degree of the vertices in the 1-resolution of the extra crossing will be reduced by one, but the saddle will either decrement or increment the number of circles in the $\zeta$-resolution, either canceling with the increase in degree or introducing a sign.  Overall, the positive Reidemeister I type cobordisms will commute with saddle cobordisms up to an overall sign.

\subsubsection{Commuting a Pair of Reidemeister Type Cobordisms}

We are left to show that changing the chronological order of a pair of distant Reidemeister type cobordisms at most incurs an overall sign.  Let $F$ and $G$ be two cobordisms arising from Reidemeister moves performed in disjoint disks $D_F$ and $D_G$.  In order to show that, at most, a sign is incurred, we will reuse the machinery from invariance under movie moves 6 through 10.  First, we will prove a small lemma.

\begin{lemma}{(Worm Lemma)}\label{lem:worm}
    Let $L$ be a link diagram embedded in $\mathbb{R}^2$, $D_0$ and $D_1$ be disjoint disks containing tangles $L\cap D_i$, and $x_0$ and $x_1$ be distinguished points on the boundary of the disks away from the ends of the tangles. There exists a disk $D'$ containing both $D_0$ and $D_1$, such that the boundary of each disk $D_i$ is in the boundary of $D'$, except in an arbitrarily small neighborhood around the distinguished point, and such that $D'$ contains a new tangle $L\cap D'$ that adds finitely many extra strands from $L$ and no extra crossings.
\end{lemma}

\begin{proof}
    It is clear that we can produce a piecewise-linear, non-self-intersecting path $p$ from $x_0$ to $x_1$ that avoids the crossings in $L$ and $D_0$ and $D_1$ (except at their endpoints) and intersects $L$ transversely at finitely many places.  We can find $P$ a thickening of $p$ by a sufficiently small amount, such that the $P$ avoids crossings, is homeomorphic to a disk, and picks up sufficiently small portions of the boundaries of $D_0$ and $D_1$ only in a connected neighborhood around the distinguished points.  It follows that $D'=D_0\cup D_1\cup P$ has the desired properties.
\end{proof}

We can choose points $x_F$ and $x_G$ living on the boundaries of $D_F$ and $D_G$, according to the rules of acceptable point placement in Figure \ref{fig:wormrules}.  

\begin{figure}[H]
        \centering
        \captionsetup{width=4.5in}
        \includeTang{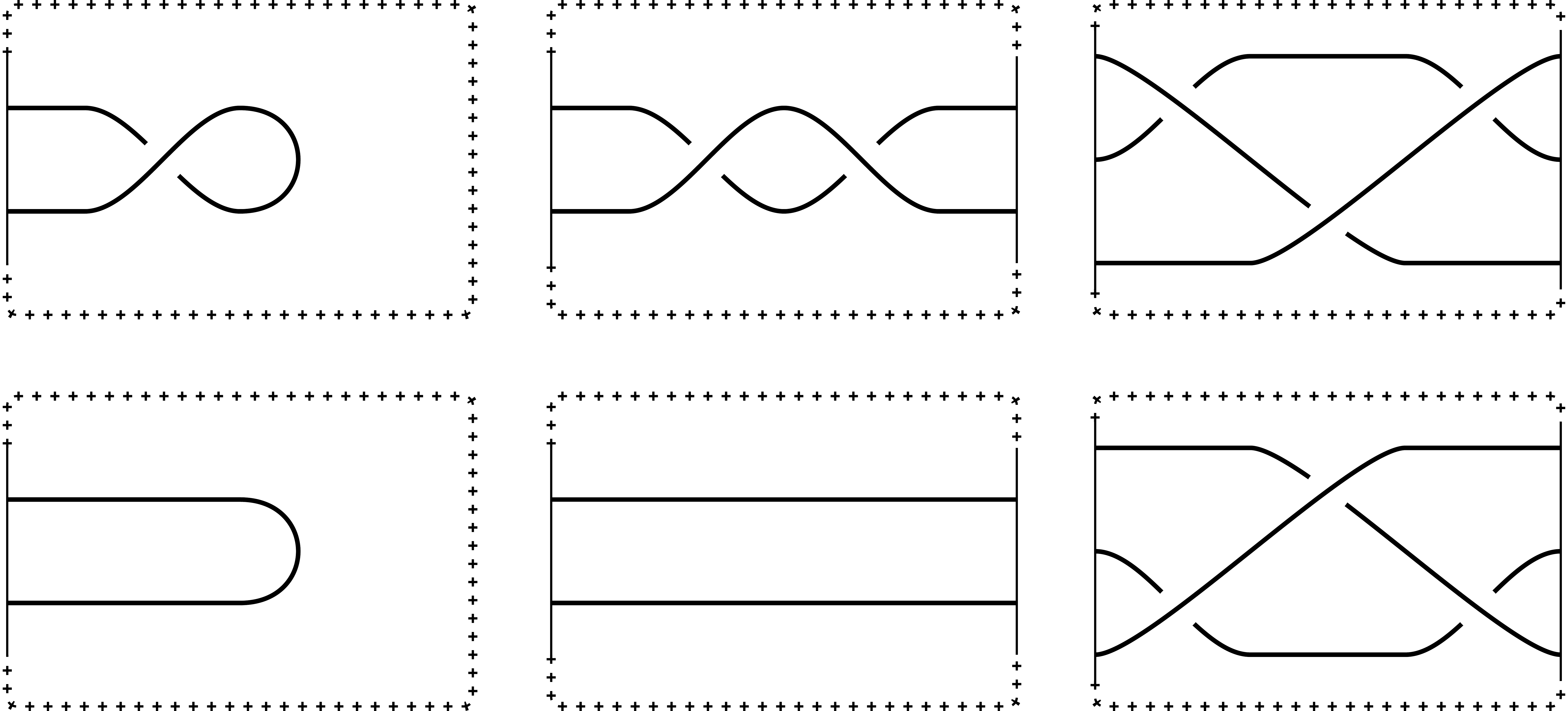}
        \caption{The ticked boundary indicates where the path should connect to the disk containing each Reidemeister move}
        \label{fig:wormrules}
\end{figure}

Let $D'$ be the larger disk produced by applying the Lemma \ref{lem:worm} to our link with disks $D_F$ and $D_G$, and $x_F$ and $x_G$ be distinguished points.  We have ostensibly dug a tunnel from the disk containing $F$ to the disk containing $G$ in order to produce one large tangle, $T'\subset D'$.  Consider the commutator of the cobordisms $H\coloneqq FGF^{-1}G^{-1}$, which begins and ends at identical diagrams.  The tangle $T'$ produced in this process can be decomposed into two parts: a crossingless tangle and an annular braid surrounding the crossingless tangle.  All of the crossings from the original tangle $T'$ appear in the annular braid. By Lemma \ref{lem:t2}, $H$ induces a chain map homotopic to $\pm\id$.  It follows that either order of composing $F$ and $G$ produces homotopic chain maps up to a possible overall sign.

\section{Dotted Odd Khovanov Homology}

In \cite{Pu2015}, Putyra develops a variant of his theory that uses dotted planar cobordisms, where dots should be regarded as \enquote{infintesimal} or \enquote{frozen} handles.  In \cite{Ma2014}, Manion develops a theory where dots are realized as part of the differential.  We will review Putyra's dotted constructions, build a notion of dotted link cobordisms, and reconcile our results with Manion's.  Finally, we will prove a structural result about odd Khovanov homology that follows from dotted cobordisms.

\subsection{Dotted Cobordisms in the Planar Setting}
 Dots are marked points on the surface of cobordisms. The height at which the dot is embedded is critical, in that our regularity conditions require them to live at different levels from other dots, as well as any other critical frames.

To build our new category, $\cobIIId$, we will start with the category $\cobIII$, then add a new basic cobordism and its relevant relations.

\begin{figure}[H]
    \centering
    \includeCob{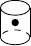}
    \caption{The dotted cobordism generator}
    \label{fig:dot}
\end{figure}

The dot is an odd cobordism yielding the following associativity and commutativity relations:

\vspace{1em}
{\noindent}\begin{minipage}{0.5\textwidth}
    \begin{equation}
        \includeFig{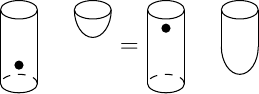}
    \end{equation}
\end{minipage}%
\begin{minipage}{0.5\textwidth}
    \begin{equation}
        \includeFig{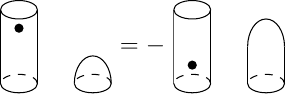}
    \end{equation}
\end{minipage}

\vspace{1em}
{\noindent}\begin{minipage}{0.5\textwidth}
    \begin{equation}
        \includeFig{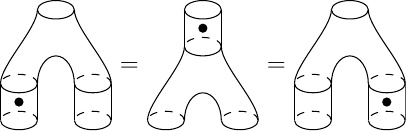}
    \end{equation}
\end{minipage}%
\begin{minipage}{0.5\textwidth}
    \begin{equation}
        \includeFig{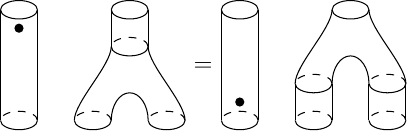}
    \end{equation}
\end{minipage}

\vspace{1em}
\begin{equation}
    \includeFig{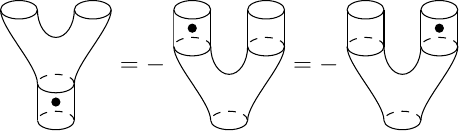}
\end{equation}

\vspace{1em}
\begin{equation}
    \includeFig{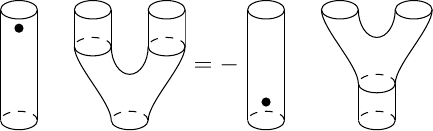}
\end{equation}

\vspace{1em}
{\noindent}\begin{minipage}{0.5\textwidth}
    \begin{equation}
        \includeFig{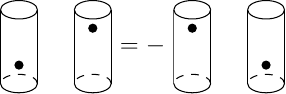}
    \end{equation}
\end{minipage}%
\begin{minipage}{0.5\textwidth}
    \begin{equation}
        \includeFig{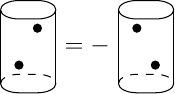}
    \end{equation}
\end{minipage}

Within a particular slice or time, dots are allowed to slide freely around the component they live on.  In the dotted setting, we get an additional sphere relation and a four-tube analog called the neck-cutting relation where, in the pictures below, the orientations are chosen according to Convention \ref{conv:2:orientations}.

\vspace{1em}
{\noindent}\begin{minipage}{0.5\textwidth}
    \begin{equation}
        \includeFig{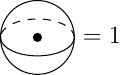}
    \end{equation}
\end{minipage}%
\begin{minipage}{0.5\textwidth}
    \begin{equation}
        \includeFig{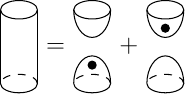}
    \end{equation}
\end{minipage}

One can check that these two relations also imply:

\begin{equation}
    \includeFig{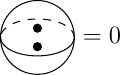}
\end{equation}
\vspace{1em}

From the neck-cutting relation, we can work out a useful horizontal variant, as well as a variant where the relation is rotated 90 degrees (where a solid dot represents a dot on the front sheet, and an empty dot represents a dot on the back sheet):

\vspace{1em}
{\noindent}\begin{minipage}{0.5\textwidth}
    \begin{equation}
        \includeFig{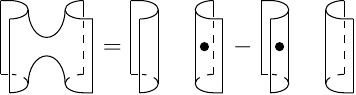}
    \end{equation}
\end{minipage}%
\begin{minipage}{0.5\textwidth}
    \begin{equation}
        \includeFig{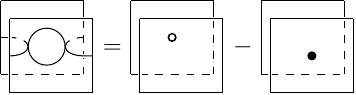}
    \end{equation}
\end{minipage}
\vspace{1em}

The idea of how the horizontal neck-cutting relation can be derived from the vertical neck-cutting relation will be clear from the proof of Lemma \ref{lem:6:horizontalNeck} later in this section.

Note that the TQFT-functor $\mathcal{F}$ from Subsection \ref{sec:3:okh} extends to a functor on $\cobIIId$, where the extended functor sends a dotted identity cobordism to the map $\Lambda^*V(R)\rightarrow\Lambda^*V(R)$ given by taking the wedge product from the left with the component $R_i\subset R$ that contains the dot.

\subsection{Dotted Cobordisms in the Link Setting}

We can enrich the link cobordism setting to one with dots. As in the planar cobordism setting, we require dots to live at different levels from all other non-trivial cobordisms (including other dots).  As dots have essentially identical behavior to deaths in regard to commutativity, the chain map associated with a dot cobordism will be similar to that of a death.  The initial and terminal links $L$ are identical, and in each degree $\alpha$ we associate the map $\varphi_\alpha$---the identity cobordism with a dot fixed on the correct component after resolution---before we correct signs by setting $\Phi_\alpha\coloneqq(-1)^{S(L,\alpha)}\varphi_\alpha$.

Our functor is only defined up to an overall sign, so there is not a well-defined manner in which to discuss linear combinations of link cobordisms.  Instead, because the sign correcting our dots is well-defined (not just up to sign), we can consider linear combinations of link cobordisms---which differ by the location of dots---as long as we first fix the map assigned to all cobordisms that are not dots.

A configuration of dots is a collection of points on the surface $F$ that are all at different heights than critical points and one another.  Let $\dott(F)$ be the free abelian group on the set of configurations of dots on $F$.  We will build a category of dotted link cobordisms $\cobIVd$. The objects of $\cobIVd$ are the same as the objects in $\cobIV$, but the morphisms consist of ordered pairs of a link cobordism $F$ from $\cobIV$, and an element of the group $\dott(F)$.

Up to ambient isotopy, there are finitely many places to put a dot on an undotted link cobordism $F$.  The dot must live between two critical points on some intermediate link diagram $D$.  If we think of a link diagram $D$ as a four-valent multi-graph (essentially a normal projection where we forget which strand was on top), then the dot must live on some edge of $D$.

\begin{convention}
    Cobordisms in $\cobIV$ will be denoted with a dot for example $\dot{F}$. The underlying undotted link cobordism will be represented by the same symbol without the dot.
\end{convention}

If we wish to compose two link cobordisms, $\dot{F}$ and $\dot{G}$, to form $\dot{GF}$, we stack the underlying link cobordism $G$ on top of $F$ to form $GF$, and then we combine the attached elements of $\dott(F)$ and $\dott(G)$ by using the obvious linear map $\dott(G)\otimes\dott(F)\rightarrow\dott(GF)$ that is induced by sending a dot configuration on $G$ and a dot configuration of $F$ to their union, viewed as a dot configuration on $GF$.

It is critical to note that, in this category, it is not true that every link cobordism factors into elementary dotted link cobordisms.  As long as odd Khovanov homology has the overall sign indeterminacy in order to have functoriality, one cannot have a category that permits this nice factoring property and still allows one to move dots around in different terms on the linear combination.  

Before we discuss equivalence in $\cobIVd$, we will define the odd Khovanov bracket for objects in $\cobIVd$.  Given a dotted link cobordism $\dot F$, first consider the chain map generated by odd Khovanov bracket for $F$. This fixes a chain map assigned to each of the non-dot cobordisms.  Then, for each configuration, we consider the larger chain map where we have added in the dot chain maps using the correcting sign defined earlier in this section.  We finish by taking the linear combination of all of these resulting chain maps.  

Recall that equivalences in $\cobIV$ are sequences of ambient isotopies which we divided into three groups: those that arise from planar ambient isotopies, chronological movie moves, and Carter and Saito's fifteen movie moves.  Equivalence in $\cobIVd$ is similar, but we must account for the movements of dots.

Suppose we have two dotted cobordisms with the same underlying cobordism $F$ decorated with different elements of $\dott(F)$.  The natural manner in which to equate configurations of dots on the same underlying cobordisms is if there is an isotopy of the dots which serves as a bijection from the dots in one configuration to the dots in the other.  This isotopy should not move dots under or over strands, but it can move dots past one another or past critical points.  In the odd Khovanov bracket, moving a dot past a critical point may incur a sign.  When moving a dot in the link cobordism setting, we will impose the relation that a sign is incurred if that change incurs a sign on the odd Khovanov bracket.  The dotted link cobordisms are equivalent if, in the elements of $\dott(F)$, each configuration is equivalent to another configuration in the other linear sum, with matching coefficients corrected for the sign that accounts for traveling past critical points.

Dotted link cobordisms are equivalent if they differ by sequences of the equivalences of dots as described above or the types of ambient isotopies for undotted link cobordisms with the following caveats. Consider two dotted cobordisms $\dot F$ and $\dot G$ with underlying ambient isotopic, but non-identical, cobordisms.  Isotopies that arise from planar ambient isotopies are not an issue for our dots, as such isotopies carry the dots along with them and do not even threaten to reorder the dots with one another.  Similarly, if $F$ and $G$ are related by chronological movie moves, we require that any dots that live chronologically between the critical points are first moved out of the way. We can always do this as the commuting critical points are distant from one another.  Finally, if $F$ and $G$ are related by one of Carter and Saito's fifteen movie moves, we will require again that the dots be relocated to different levels before the movie move is carried out.

\begin{remark}
    Difficulties can arise from working in $\cobIVd$. Dots can get trapped (for example in a movie move 3), and be unable to relocate, obstructing one from completing the movie move.  Later in this section, we will see that there is a sensible way to send dots over or under crossings, insofar that the effect on the odd Khovanov bracket is easy to define.  Our current definition of $\cobIVd$ is sufficient for our purposes, but in other settings it may be necessary to incorporate these equivalences into the equivalence of dotted link cobordisms.
\end{remark}

It is important to note that, as the original odd Khovanov homology functor was only defined up to an overall sign, the dotted version cannot distinguish between a link cobordism $F$ equipped with an element of $\dott(F)$ or the negative of the element.

\begin{convention}\phantom{.}
\begin{itemize}
    \item To denote an equivalence of cobordisms in $\cobIVd$ we will use the symbol $\cong$.
    \item To denote cobordisms that induce homotopic chain maps on the odd Khovanov bracket we will use the symbol $\simeq$.
\end{itemize}    
\end{convention}

\begin{theorem}\label{thr:5:dotFunctor}
    Odd Khovanov homology extends to a functor from the category $\cobIVd$ to the category $\cobIIIdKM$ up to sign and homotopy.
\end{theorem}

\begin{proof}
    We defined the product of linear combinations of dotted link cobordisms earlier, such that it corresponds precisely with the composition of chain maps.  Furthermore, the signs incurred by the dot were designed to ensure that our notion of equality would permit this to be functorial with respect to the reordering of dots.  This construction still respects Carter and Saito's fifteen movie moves, as it retains the chain maps associated with each elementary link cobordism.
\end{proof}

In \cite{Ma2014}, Manion considers a dotted odd Khovanov homology, although he works strictly with the link invariant, meaning his dots are part of the differential rather than chain maps.  To translate from his setting to ours, one would have to sprinkle in signs such that Manion's dot maps commute rather than anticommute with the original differential in the odd Khovanov complex. Additionally, to ensure that dots act as a differential, Manion imposes that a cobordism containing two dots on a surface evaluates to zero.  We have only the slightly weaker condition that a cobordism containing two dots on a surface is annihilated by 2.  Additionally, Manion works with type X odd Khovanov homology, whereas we work with type Y odd Khovanov homology.  This difference slightly alters the behavior of dots particularly in their interactions with crossings.

\subsubsection{Link Cobordisms and the Neck-Cutting Relation}

The more powerful results in the dotted link cobordism setting arise from being able to apply the neck-cutting relation to link cobordisms.  We will start by showing that the vertical neck-cutting relation is a chain map, then we will use that result to show that the horizontal neck-cutting relation is as well.  Manion proved a lemma \cite[Lemma 3.2]{Ma2014} analogous to ours about horizontal necks. By working with odd Khovanov homology as a functor from $\cobIVd$, we are able to show that the vertical neck-cutting relation is respected by odd Khovanov homology, then use this to show the same for the horizontal neck-cutting relation.  This proof avoids the case-by-case analysis used by Manion, which is necessitated by the setting he works in.

\paragraph{Link Cobordism Vertical Neck Cutting Relation}\phantom{.}

{\noindent}Let $\dot{F}$ be a dotted link cobordism containing a vertical tube.  Between any two critical frames, one can always artificially choose a frame that will be treated as critical.  This amounts to a choice in every configuration of how to split the corresponding word into two parts.  Let $\dot{F'}$ be the link cobordism where we replace the neck with the vertical neck relation in a neighborhood of the artificial critical level.  We then produce the final linear combination of configurations on $\dot{F'}$ by splicing in the vertical neck-cutting relation, in the manner described in Convention \ref{conv:5:centersplice}, at the height where the artificial critical frame was fixed.  The dotted link cobordisms $\dot{F}$ and $\dot{F'}$ are depicted in diagrams \eqref{dig:5:vertneckI} and \eqref{dig:5:vertneckII}, respectively.

\begin{lemma}[The Vertical Neck-Cutting Relation for Link Cobordisms]
    If dotted link cobordisms $\dot{F}$ and $\dot{F'}$ are related by the vertical neck-cutting relation, then $\dot{F}\simeq\dot{F'}$.
\end{lemma}

\begin{convention}\label{conv:5:centersplice}
    In some of our diagrams of dotted link cobordisms there is a slight abuse of notation where we draw a diagram that makes a dotted cobordism appear to be the composition of multiple dotted cobordisms when, in fact, it does not factor as such.  For these diagrams, we will use dotted lines to indicate when the cobordism should not be assumed to factor.  In such instances, the central cobordism will appear as a linear combination of configurations sandwiched between two arbitrary cobordisms.  With such a diagram we take some arbitrary configuration on the entire cobordism that places no dots in the central component, then we splice in the central component's configurations and take a product between the outside linear combination and the central linear combination.
\end{convention}

\begin{proof}
    If we consider the chain maps living over a fixed resolution, it is clear that $\dot{F}$ and $\dot{F'}$ induce the same maps (as we have the vertical neck-cutting relation) in the dotted planar cobordism setting.  If we now consider the entire complex, it is clear that $\dot{F}$ has no signs associated with the cobordism at the level where we are altering the cobordism, as it is just an identity cobordism.  In the maps induced by $\dot{F'}$, there are signs associated with the death and with the dots.  Critically, in both terms, there is one dot and one death, and each occurs on the same underlying diagram. Therefore, they are given the same sign, and no sign is incurred as the given sign was squared.
\end{proof}

{\noindent}\begin{minipage}{0.5\textwidth}
    {\noindent}\begin{equation}\label{dig:5:vertneckI}
       \includeFig{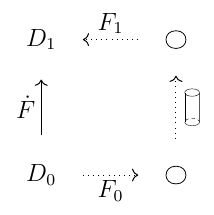}
    \end{equation}
\end{minipage}%
\begin{minipage}{0.5\textwidth}
    {\noindent}\begin{equation}\label{dig:5:vertneckII}
        \includeFig{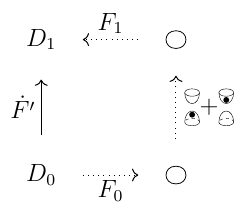}
    \end{equation}
\end{minipage}
\vspace{1em}

\paragraph{Link Cobordism Horizontal Neck-Cutting Relation}\phantom{.}

{\noindent}As a fairly immediate corollary of the vertical neck-cutting relation for link cobordisms, we have the following lemma.

\begin{lemma}[The Horizontal Neck-Cutting Relation for Link Cobordisms]\label{lem:6:horizontalNeck}
    If dotted link cobordisms $\dot{F}$ and $\dot{F'}$ are related by the horizontal neck-cutting relation, then $\dot{F}\simeq\dot{F'}$.
\end{lemma}

\begin{proof}
In the following diagrams, the cobordisms outside of the relevant tangle or distant (with respect to time) from the horizontal neck are omitted.  Each cobordism should be thought of as the product of the shown configuration spliced into the center of some arbitrary outside dotted link cobordism in the manner prescribed in Convention \ref{conv:5:centersplice}.
\begin{equation}\label{dig:5:horzneck}
    \includeFig{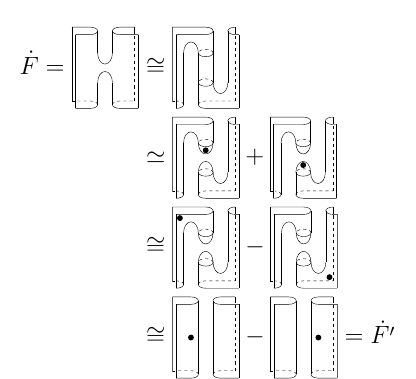}
\end{equation}
\end{proof}

\subsubsection{Link Cobordism Dot Slides}

In \cite{Ma2014}, Manion also considers how dots interact with crossings. Particularly, he relays what impact sliding a dot past a crossing has on the isomorphism class of the chain maps produced by odd Khovanov homology. Manion proved a theorem \cite[Theorem 3.1]{Ma2014} that sliding one of his dots under a crossing produces isomorphic twisted odd Khovanov homology.  If he had used type Y sign assignments, he would have arrived at the same result, but for an overslide.

\begin{theorem}\label{thm:5:dotslide}\phantom{.}

\begin{enumerate}[label=\alph*.]
    \item Let $\dot{F}$ and $\dot{F'}$ be the link cobordisms in diagrams \eqref{dig:5:overdotslideI} and \eqref{dig:5:overdotslideII}, respectively, related by sliding a dot over a crossing.  Then, $\dot{F}\simeq\dot{F'}$.\footnote{We can slide dots through overcrossings as we built our functor using type Y sign assignments.  If we used type X sign assignments, the theorem would hold for undercrossings instead.}

    \vspace{1em}
    {\noindent}\begin{minipage}{0.5\linewidth}
        \begin{equation}\label{dig:5:overdotslideI}
            \includeFig{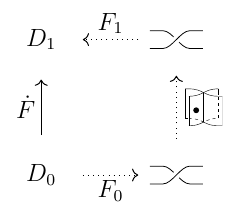}
        \end{equation}
    \end{minipage}%
    \begin{minipage}{0.5\linewidth}
        \begin{equation}\label{dig:5:overdotslideII}
            \includeFig{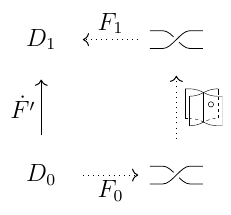}
        \end{equation}
    \end{minipage}
    \vspace{1em}
    
    \item Let $\dot{G}$ and $\dot{G'}$ be the link cobordisms in diagrams \ref{dig:5:underdotslideI} and \ref{dig:5:underdotslideII}, respectively, related by sliding a dot under a crossing.  If we restrict ourselves to knot cobordisms, then for rational odd Khovanov homology, $\dot{F}\simeq\dot{F'}$.

    \vspace{1em}
    {\noindent}\begin{minipage}{0.5\linewidth}
        \begin{equation}\label{dig:5:underdotslideI}
            \includeFig{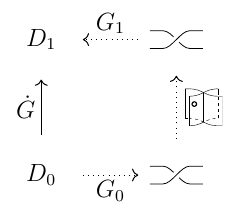}
        \end{equation}
    \end{minipage}%
    \begin{minipage}{0.5\linewidth}
        \begin{equation}\label{dig:5:underdotslideII}
            \includeFig{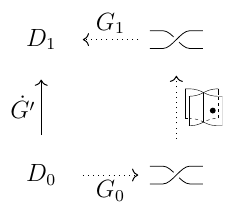}
        \end{equation}
    \end{minipage}
    \vspace{1em}
    
\end{enumerate}
\end{theorem}

Manion's Theorem 3.1 implies Theorem \ref{thm:5:dotslide}.a in our setting.  Roughly, Manion considers the homotopy corresponding with the reversed saddle for the particular crossing that the dot slid past.  When we consider the terms generated by composing this saddle with the saddle differential, we arrive at a pair of horizontal neck cobordisms that then decompose via the horizontal neck-cutting relation into the difference of the two link cobordisms.

\begin{figure}[H]
    \centering
    \captionsetup{width = 4in}
    \includeFig{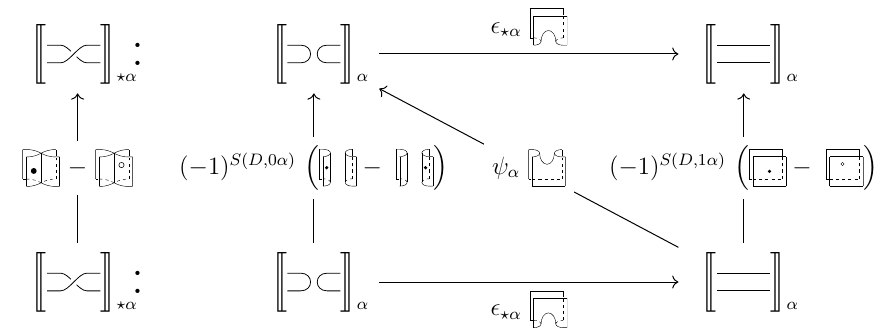}
    \caption{The homotopy map used by Manion.  Note that the particular sign $\psi_\alpha$ is carefully selected in \cite{Ma2014}}
    \label{fig:5:ManionHomotopy}
\end{figure}

In general, one cannot slide a dot under a crossing.  Consider the dotted link cobordism shown in the first picture of~\eqref{eq:5:dottedsphere}. It consists of a sphere passing through a vertical curtain, where it is assumed that the curtain is connected to itself on the outside to form an identity cobordism. The picture itself shows the projection of this cobordism into $\mathbb{R}^3$, where the projection is given by sending a point $(x,y,z,t)$ in $\RR^4$ to the point $(x,y,t)$. Along the circle of intersection, the curtain is assumed to overcross the sphere, so that it has a greater $z$-coordinate. On the level of movies, this means that the link diagram for $t=\nicefrac{1}{2}$ consists of a circle (coming from the sphere) that passes under a strand (coming from the vertical curtain). On either side of the identity cobordism we place a dot on the sphere:

\begin{equation}\label{eq:5:dottedsphere}
    \includeFig{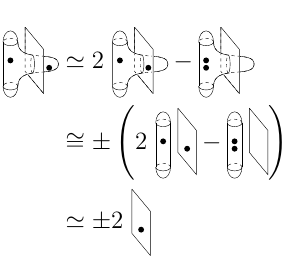}
\end{equation}

In \eqref{eq:5:dottedsphere} above, the first homotopy follows from the discussion in the next subsection.  It is clear that the dotted cobordism in the last line does not induce the zero chain map on odd Khovanov homology.  If we could slide dots under crossings (even up to an overall sign), then we could slide the dots onto the same side of the identity cobordism and unlink the circle.  Then, we would be left with a twice-dotted sphere killing the entire cobordism on the level of odd Khovanov homology.

We will come back to the proof of Theorem \ref{thm:5:dotslide}.b after we explore the implications of Theorem \ref{thm:5:dotslide}.a on a module structure for odd Khovanov homology.

\subsubsection{Link Cobordism and the Coloring Module}

\begin{definition}
    For a given link diagram $D$, let \defw{$\edge(D)$} be the free abelian group generated by the set of edges in $D$ when considered as a graph.
\end{definition}

\begin{definition}
    For a given link diagram $D$, let \defw{$\arc(D)$} be the free abelian group on the arcs in $D$ or $\edge(D)$ with the relation that $e_i-e_j=0$ if $e_i$ and $e_j$ are edges separated by an overcrossing.
\end{definition}

We can now define the coloring group of a knot.\footnote{Fox's $n$-colorings correspond to  homomorphisms from this group to $\ZZ/n\ZZ$.}

\begin{definition}
    For a given link diagram $D$, let \defw{$\col(D)$} be $\arc(D)$ with the additional relation that for each crossing, if the over arc is labeled $e_i$ and the other strands are labeled $e_j$ and $e_k$, respectively.  Then, $2e_i-e_j-e_k=0$.
\end{definition}

Let $\bdbc(L)$ be the branched double cover of  $S^3$ branched along a link $L\subset S^3$. Then, we have the following lemma from \cite{Pr1998} due to Przytycki.

\begin{lemma}
    Let $D$ be a diagram of a link $L$, then $\col(D)\cong\ZZ\oplus H_1\left(\bdbc(L);\ZZ\right)$.
\end{lemma}

Consider the following dotted link cobordism

\begin{equation}
    \includeFig{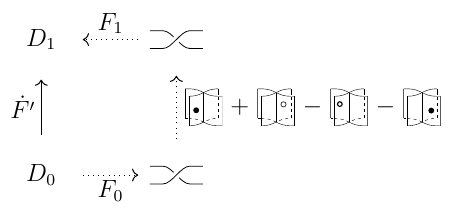}
\end{equation}

With cursory review, it is clear that the above dotted cobordism induces the zero map on odd Khovanov homology.  Indeed, on the vertical resolution, the term with the dot to the left crossing shows up twice, but with opposite signs.  This is the same for the term with the dot to the right of the crossing.  On the horizontal resolution, maps work out similarly for the terms with the dot on either the forward or rear curtain.  As we can slide dots over crossings, we get the following dotted link cobordism that also induces the zero map on odd Khovanov homology.

\begin{equation}
    \includeFig{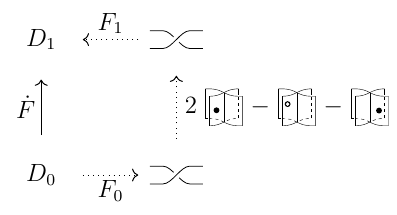}
\end{equation}

From this, a module structure on odd Khovanov homology follows.

\begin{theorem}\label{thm:5:colMod}
    Let $D$ be a diagram of a link $L$, then $\okh(L)$ is a module over the exterior algebra of $\col(D)$.
\end{theorem}

\begin{proof}
    From the discussion above, it is clear that elements of $\col(D)$ act on the odd Khovanov homology of $L$, where the action is induced by dotted identity cobordisms.    
    As dots anticommute with one another, the maps induced by dots are compatible with the exterior algebra product.  It follows that we can realize odd Khovanov homology as a module over the ring, which is the exterior algebra of $\col(D)$.
\end{proof}

This allows us to think of dots more geometrically, particularly because we can realize the horizontal neck-cutting relation as classes in the first homology of the branched double cover.  We can now return to Theorem \ref{thm:5:dotslide}.b.

\begin{proof}[Proof of Theorem \ref{thm:5:dotslide}.b]  Let $D$ be the diagram of a knot.  Consider the surjective homomorphism 

\begin{equation}
    \phi:\text{Col}(D)\rightarrow\ZZ
\end{equation}

which sends each generator of $\col(D)$ to 1. Note that this homomorphism is well-defined because it sends each relation on $\col(D)$ to $0$.  As $\col(D)$ is isomorphic to the first homology of the branched double cover together with an extra $\ZZ$-summand, and the branched double cover of a knot is always a rational homology sphere, $\phi$ induces an isomorphism

\begin{equation}
    \Tilde{\phi}:\text{Col}(D)\otimes\QQ\rightarrow\QQ
\end{equation}

Since $\Tilde{\phi}$ sends each generator to 1, this shows that all generators are identified in $\text{Col}(D)\otimes\QQ$.  Hence all identity cobordisms decorated by a single dot induce the same map on the rational odd Khovanov homology of the given knot.
\end{proof}

\section{Hecke Algebra Action}

In \cite{GLW2017}, the functoriality up to sign of even Khovanov homology is used to construct an action of the symmetric algebra $\mathcal{S}_n$ on the even Khovanov homology of the $n$-cable of a link.  We will construct a similar action for odd Khovanov homology where we will begin by building the proper setting.

\subsection{Framed Links and Cables}

There is a structure that can be endowed on a knot or link so that the knot or link more closely models the tying of ribbons.

\begin{definition}
    A \defw{framed} knot or link is a smooth embedding of a solid torus or solid tori into $\RR^3$.
\end{definition}

If we think of the solid torus as being $D^2\times S^1$---where $D^2$ is the unit disk in $\CC$---then we can see this ribbon-like behavior if we consider the induced embedding of the annulus $I\times S^1$, $[-1,1]\times S^1$, where $[-1,1]$ is the real part of $D^2$.  One can also view a framing as a choice of a continuous nowhere vanishing normal vector field on an ordinary, unframed knot or link---where the unframed knot corresponds to the induced embedding of the core circle of the solid torus $D^2\times S^1$.  In the following, we write

\begin{equation}
    \mathfrak{F}_K:D^2\times S^1\rightarrow \RR^3
\end{equation}

for the embedding corresponding to a framed knot $K$.

One can still consider projection-based diagrams of knots and links if we endow the diagram with the following natural framing.

\begin{definition}
    For a given diagram of a knot or link, the \defw{blackboard framing} is the framing of the link in which the chosen normal vector field lies parallel to the plane of the picture.
\end{definition}

The blackboard framing's name is derived from the idea that it is the framing which we see when a knot is projected on a blackboard and drawn with actual (non-infinitesimal) chalk lines.  Framings of knots are interesting up to isotopy.  There are countably many non-isotopic framings for every knot. 

\begin{definition}
    For a given framed knot, the \defw{framing number} is the linking number between the two boundary components of the ribbon obtained by embedding the annulus $[-1,1]\times S^1$, where here the two boundary components are oriented parallel to each other.
\end{definition}

In other words, the framing number is the signed number of twists (including coils) in the ribbon before the knot is closed back up.  For a blackboard-framed link, one can directly read off the framing number of a diagram by computing the writhe.  Framing numbers can be affixed to links by assigning a framing number to each of its components. However, for the remainder of this section, we will restrict the scope of our inquiry to knots.  

Diagrams of framed knots are known to represent isotopic framed knots if they can be transformed into each other using the normal Reidemeister II and III moves, and a modified Reidemeister I move that is a pair of twists with opposite orientations in place of the normal Reidemeister I move.  If we look at the normal Reidemeister I move, it adds a twist, and thus either increments or decrements the writhe. Thus, it could never preserve the framing. Framing also provides us with a method to build a more complicated link out of a knot or a link.

\begin{definition}
    The \defw{n-cable} of a framed knot is the link produced by the induced embedding of $n$ circles in the solid torus, where the circles are the products of $n$ evenly spaced points on the real part of $D^2$ with $S^1$.
\end{definition}

For a framed knot $K$ we will denote the $n$-cable of $K$ by $K^n$.

\begin{figure}[H]
    \centering
    \includeFig{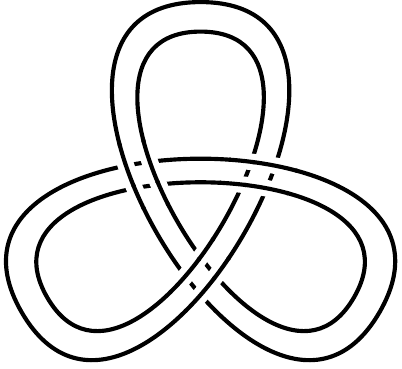}
    \caption{The 2-cable of the trefoil}
\end{figure}

\subsection{Tangle Cobordisms}

There is a natural way in which to build cobordisms from $K^n$ to $K^m$.

\begin{definition}
    An \textbf{(}\textbf{\textit{m,n}}\textbf{)}\textbf{\textit{-tangle}} is a smooth compact 1-manifold with boundary embedded in $D^2\times I$, such that the boundary of the 1-manifold consists of $n$ points in $D^2\times\{0\}$ and $m$ points on $D^2\times\{1\}$, and no other points of the 1-manifold lie in the boundary of $D^2\times I$.
\end{definition}

Let the following map be the embedding map for a tangle $T$, which embeds disjoint copies of $S^1$ and $I$ in $D^2\times I$.

\begin{equation}
    \mathfrak{t}_T:\left(\bigsqcup_{1\leq i\leq p}I_i\right)\sqcup\left(\bigsqcup_{1\leq j\leq q}S^1_j\right)\rightarrow D^2\times I
\end{equation}

We want to build the cobordism that corresponds with \enquote{sweeping} the tangle around a torus.  Instead of embedding one-dimensional disks and circles, now we will be embedding annuli and tori. We will also be mapping to a thickened solid torus.  The following map is the obvious product map, where we embed the product of our domain with $S^1$ into the product of our range with $S^1$

\begin{equation}
    \mathfrak{t}_T\times \id_{S^1}:\left[\left(\bigsqcup_{1\leq i\leq p}I_i\right)\sqcup\left(\bigsqcup_{1\leq j\leq q}S^1_j\right)\right]\times S^1\rightarrow D^2\times I\times S^1
\end{equation}

Now we can build the tangle cobordism that corresponds to \enquote{sweeping} the tangle along a knot.  For a tangle $T$ and a framed knot $K$, we we will denote the tangle cobordism by $T\times K$, as it is essentially the product of the tangle and the knot.  The cobordism $T\times K$ lives in $\RR^3\times I$, and we start by considering the map $\mathfrak{t}_T\times S^1$, which lands in $D^2\times I\times S^1$. We map the $I$ component of $D^2\times I\times S^1$ to the $I$ component of $\RR^3\times I$, and the $D^2\times S^1$ component into $\RR^3$ using $\mathfrak{F}_K$, which is the map that embeds the framed knot $K$ in $\RR^3$.

\paragraph{Notation and Braid Generators}\phantom{.}

{\noindent}Of particular interest will be the tangle cobordism where the tangle arises from a generator of the $n$-strand braid group.  The $i\ordth$ generator of the $n$-strand braid group is the tangle on $n$ strands numbered $1$ through $n$ that crosses the $i\ordth$ and the $(i+1)\ordth$ strands.

\begin{figure}[H]
    \centering
    $\underset{1\hspace{1.1em}2}{\mysym{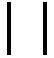}}\dots\underset{\hspace{0.5em}i\hspace{0.9em}i+1}{\mysym{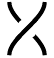}}\dots\underset{n}{\mysym{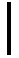}}$
    \caption{The positive $i\ordth$ generator of the braid group on $n$ strands}
\end{figure}

We will denote the cobordism associated with the positive $i\ordth$ generator of the braid group on $n$ strands using the convention depicted in the following diagram.

\begin{equation}\label{dig:6:tangleNotation}
    \mysym{01}_i\upK
\end{equation}

The notation in \eqref{dig:6:tangleNotation} is the general notation we will use, where the subscript $i$ denotes that the left-most strand of the presented tangle is the $i\ordth$ strand, and the superscript notates that we are sweeping around the knot $K$, and that this tangle contains $n$ strands such that all those not depicted are identity strands (they go straight from the bottom to the top of the tangle).  One tangle that will be of particular interest is the \enquote{cap-cup} tangle, which is essentially the one-dimensional version of a death followed by a birth.

\begin{figure}[H]
    \centering
    $\underset{1\hspace{1.1em}2}{\mysym{03}}\dots\underset{\hspace{0.5em}i\hspace{0.9em}i+1}{\mysym{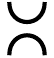}}\dots\underset{n}{\mysym{34}}$
    \caption{The $i\ordth$ cap-cup tangle on $n$ strands}
\end{figure}

\subsection{The Hecke Algebra}

The Hecke algebra of type $\mathcal{A}_{n-1}$ is an associative unital algebra $\mathcal{H}(q^2,n)$ with generators and relations listed below, where $q$ is a fixed invertible element of the ground ring.

\singlespacing
\begin{equation}
    \mathcal{H}(q^2,n)=
    \resizebox{\width}{1.2\height}{\Bigg\langle}
    \left. g_1,\dots,g_n\middle|
    \begin{array}{cc}
    g_i^2=(q^2-1)g_i+q^2 &  1\leq i\leq n\\
    g_ig_{i+1}g_i=g_{i+1}g_ig_{i+1} & 1\leq i\leq n \\
    g_ig_j=g_jg_i & 1\leq i,j\leq n\text{ and }|i-j|\geq 2
    \end{array}
    \right.
    \resizebox{\width}{1.2\height}{\Bigg\rangle}
\end{equation}
\doublespacing

The algebra $\mathcal{H}(q^2,n)$ can be seen as a quotient of the group algebra of the braid group on $n$ strands with the additional relation $g_i^2=(q^2-1)g_i+q^2$.  This relation is related to the skein relation for the Jones polynomial and, more generally, the HOMFLY-PT polynomial \cite{Jo1987}. One can also think of the Hecke algebra $\mathcal{H}(q^2,n)$  as a deformation of the group algebra of the symmetric group on $n$ elements, to which it specializes when $q=1$.  In this dissertation, we will be particularly concerned with the case when $q=i$. Thus, we are considering the following particular algebra, where coefficients are assumed to be in $\mathbb{Z}$:

\singlespacing
\begin{equation}
    \mathcal{H}(-1,n)=
    \resizebox{\width}{1.2\height}{\Bigg\langle}
    \left. g_1,\dots,g_n\middle|
    \begin{array}{cc}
    g_i^2=-2g_i-1 &  1\leq i\leq n\\
    g_ig_{i+1}g_i=g_{i+1}g_ig_{i+1} & 1\leq i\leq n \\
    g_ig_j=g_jg_i & 1\leq i,j\leq n\text{ and }|i-j|\geq 2
    \end{array}
    \right.
    \resizebox{\width}{1.2\height}{\Bigg\rangle}
\end{equation}
\doublespacing

\subsection{Swapping Strands in the $n$-Cable}

In \cite{GLW2017}, it was shown that on even Khovanov homology, the map induced by exchanging two adjacent strands of the $n$-cable of a knot decomposes into plus or minus the identity map and plus or minus the map induced by the cobordism that is the product of the cap-cup tangle over the two strands with the knot.  We will prove a similar result for odd Khovanov homology, which is a porism or corollary of the proof from the result in \cite{GLW2017}.  We will outline the proof here in order to highlight that there are no significant changes incurred by transferring to the odd setting.

In the following movies, we will only show two strands. As long as $n$ is bigger than two, there are more strands either surrounded by or surrounding the shown strands.  The shown strands are the two that are going to be exchanged.  The movie of the strand-swap begins with one of the strands crossing over the other with a Reidemeister II move.  Next, one of the crossings \enquote{slides} around the knot, passing over and under strands along the way in a sequence of Reidemeister III moves.  To finish, another Reidemeister II move completes the swap.  This process is shown in diagram \eqref{dig:6:strandSwap}, where the box labeled $K$ represents the knotted part of the $n$-cable (more explicitly, it represents the $n$-cable of a framed $(1,1)$-tangle diagram whose closure is $K$).

\vspace{1em}
{\noindent}\begin{minipage}{0.5\linewidth}
    \begin{equation}\label{dig:6:strandSwap}
    \includeMov{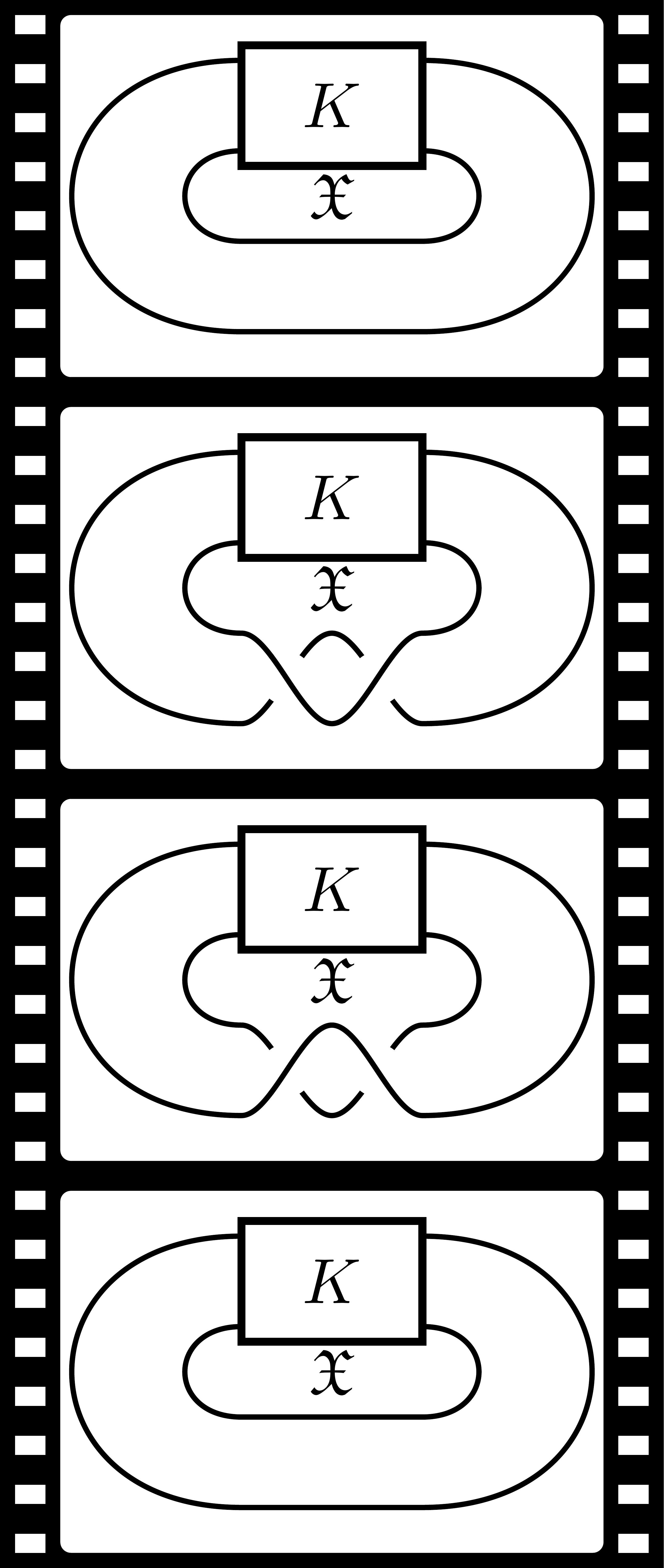}
    \end{equation}
\end{minipage}%
\begin{minipage}{0.5\linewidth}
    \begin{equation}\label{dig:6:strandSwapUnknot}
    \includeMov{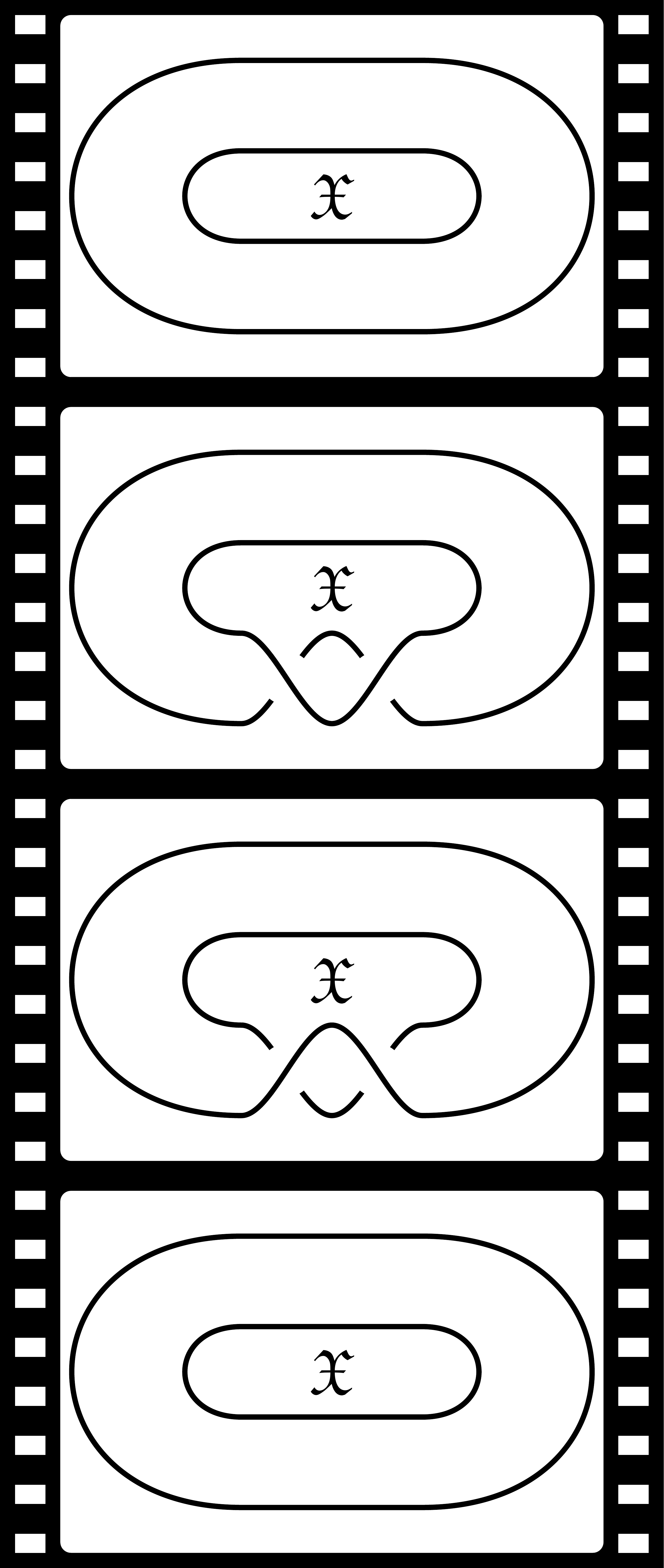}
    \end{equation}
\end{minipage}
\vspace{1em}

We start by considering the $n$-cable of the 0-framed unknot with the blackboard framing $U$, whose movie is shown in diagram \eqref{dig:6:strandSwapUnknot}.  Now the movie is generated by two Reidemeister II type cobordisms and an isotopy.  Consider the following diagram representing the support of the strand-swap cobordism in which we have forgotten all of the signs.

\begin{equation}
    \includeFig{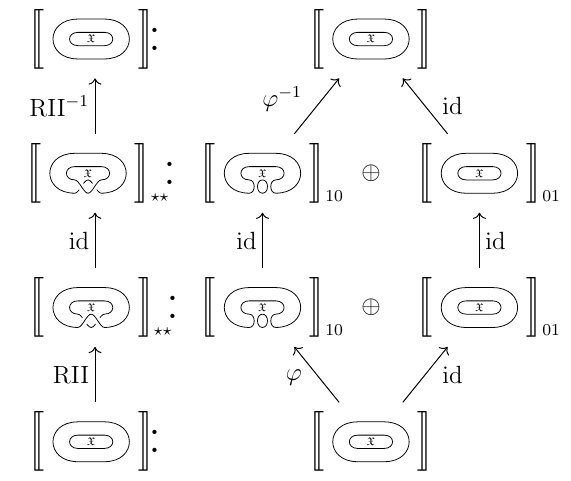}
\end{equation}

If we look at the right-hand side, it is clear that the composition produces the identity cobordism. If we look at the left-hand side, the map $\varphi$ is the one coming from the Reidemeister II cobordism chain map. This generates the following movie.

\begin{equation}
    \includeFig{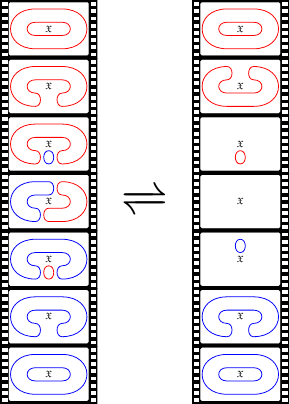}
\end{equation}

If we put these together, we see that the action decomposes in the following manner on the unknot.

\begin{equation}
    \left\llbracket\mysym{01}_i\upU\right\rrbracket_\alpha=\pm\left\llbracket\mysym{03}_i\upU\right\rrbracket_\alpha\mp\left\llbracket\mysym{04}_i\upU\right\rrbracket_\alpha
\end{equation}

Now, we can return to the case of the general knot.  The claim is that we arrive at the following diagram.

\begin{equation}
    \includeFig{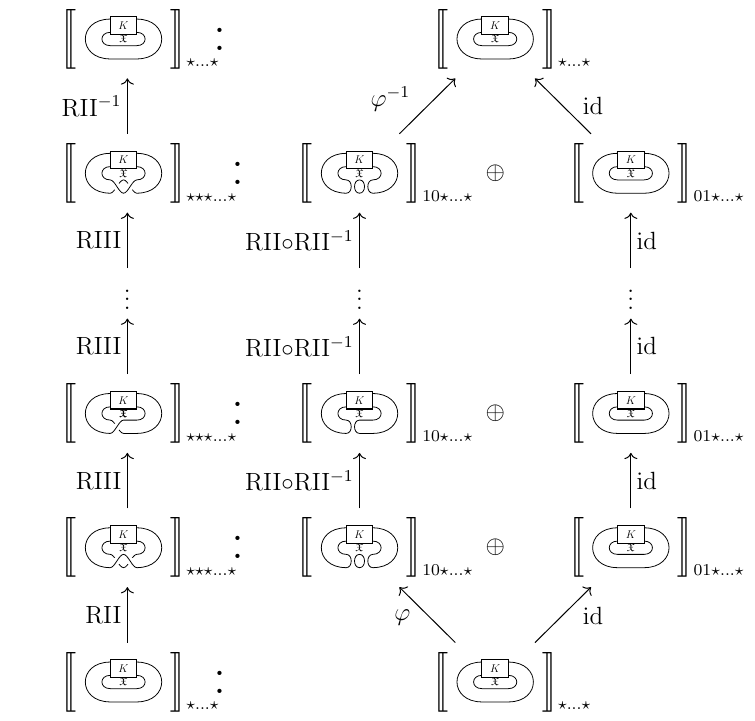}
\end{equation}

The Reidemeister III cobordisms cannot change the smoothing at the crossing that is not traveling. A grading argument is employed in \cite{GLW2017} to argue that if we restrict the Reidemeister III chain maps to the diagram in which the traveling crossing has the correct smoothing, then the cobordisms restrict to a pair of Reidemeister II type cobordisms.  This behavior aligns with the Reidemeister III chain map we previously defined, which is built from an identity cobordism and a pair of Reidemeister II type cobordisms.  As the gradings and the underlying cobordisms in the chain maps are the same in the even and odd settings, the same argument applies here as well. Therefore, we have the following equation.

\begin{equation}
    \left\llbracket\mysym{01}_i\upK\right\rrbracket_\alpha=\pm\left\llbracket\mysym{03}_i\upK\right\rrbracket_\alpha\mp\left\llbracket\mysym{04}_i\upK\right\rrbracket_\alpha
\end{equation}

\paragraph{Fixing the Sign Ambiguity}\phantom{.}

{\noindent}At this point it is not clear what sign shows up on the two terms.  In fact, these signs may depend on additional choices, since odd Khovanov homology is only functorial up to sign.  We therefore write

\begin{equation}\label{eq:6:decomp}
    \left\llbracket\mysym{01}_i\upK\right\rrbracket=a_i\left\llbracket\mysym{03}_i\upK\right\rrbracket+b_i\left\llbracket\mysym{04}_i\upK\right\rrbracket
\end{equation}

where $a_i$ and $b_i$ are either 1 or $-1$. Consider the induced map on the odd Khovanov homology after the odd Khovanov TQFT is applied.

\begin{equation}
     \mathcal{F}\left(\left\llbracket\mysym{01}_i\upK\right\rrbracket\right):\okh(K^n)\rightarrow\okh(K^n)
\end{equation}

Particularly consider rational odd Khovanov homology.  As the induced map is a linear operator on a finite-dimensional vector space over the perfect field $\QQ$, it uniquely decomposes into a semisimple part and a commuting nilpotent part via its Jordan-Chevalley decomposition. Lemma~\ref{lem:6:toruszero} below further shows that, under some assumptions, the cap-cup component in \eqref{eq:6:decomp} induces a nilpotent map. Since the first component in \eqref{eq:6:decomp} comes from an identity cobordism, it thus follows that the decomposition in \eqref{eq:6:decomp} corresponds to the Jordan-Chevalley decomposition of the induced linear operator. We will fix the sign of the crossing in \eqref{eq:6:decomp} so that a negative sign appears on the semisimple part of the Jordan-Chevalley decomposition. Since odd Khovanov homology---and hence the identity map induced by a crossing---is always nonzero over any coefficient ring, fixing the sign over rational coefficients also fixes the sign over integral coefficients.

\begin{equation}
    \left\llbracket\mysym{01}_i\upK\right\rrbracket=-\left\llbracket\mysym{03}_i\upK\right\rrbracket+b_i\left\llbracket\mysym{04}_i\upK\right\rrbracket
\end{equation}

We can now fix the sign on the cap-cup piece so that it is equal to the sum of the crossing and the identity tangles.

\begin{equation}\label{eq:6:skein}
    \left\llbracket\mysym{01}_i\upK\right\rrbracket=-\left\llbracket\mysym{03}_i\upK\right\rrbracket+\left\llbracket\mysym{04}_i\upK\right\rrbracket
\end{equation}

It follows for the opposite crossing type we assign the following:

\begin{equation}
    \left\llbracket\mysym{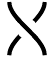}_i\upK\right\rrbracket=-\left\llbracket\mysym{03}_i\upK\right\rrbracket-\left\llbracket\mysym{04}_i\upK\right\rrbracket
\end{equation}

\paragraph{Evaluating Tori}\phantom{.}

{\noindent}In the even setting, a torus embedded in $\RR^4$ evaluates to 2, even when the torus is knotted. In the odd setting, it is not always the case that tori evaluate to 0 when knotted with other components of a link cobordism.  Under specific conditions, we can guarantee that the tori that arise from cobordisms of the form $T\times K$ evaluate to 0.

\begin{lemma}\label{lem:6:toruszero}
    For a natural number $n$ and a framed knot $K$, if either $n$ is even or $K$ has even framing, then the following holds.
    \begin{equation}
        \left\llbracket\mysym{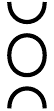}_i\upK\right\rrbracket\simeq0
    \end{equation}
\end{lemma}

\begin{proof}
The product of our tangle with $K$ contains a torus of the form $S^1\times K$, which can be represented by the movie in \eqref{eq:6:torusmovie}. As in our previous movies, we do not show all the strands of the cable, and the box labeled $K$ represents the $n$-cable of a $(1,1)$-tangle diagram whose closure is $K$. The movie starts with a birth followed by a sequence of Reidemeister II cobordisms knotting the circle with the remainder of the link, then a saddle connecting the circle with itself before the entire process is reversed.
\begin{equation}\label{eq:6:torusmovie}
        \includeMov{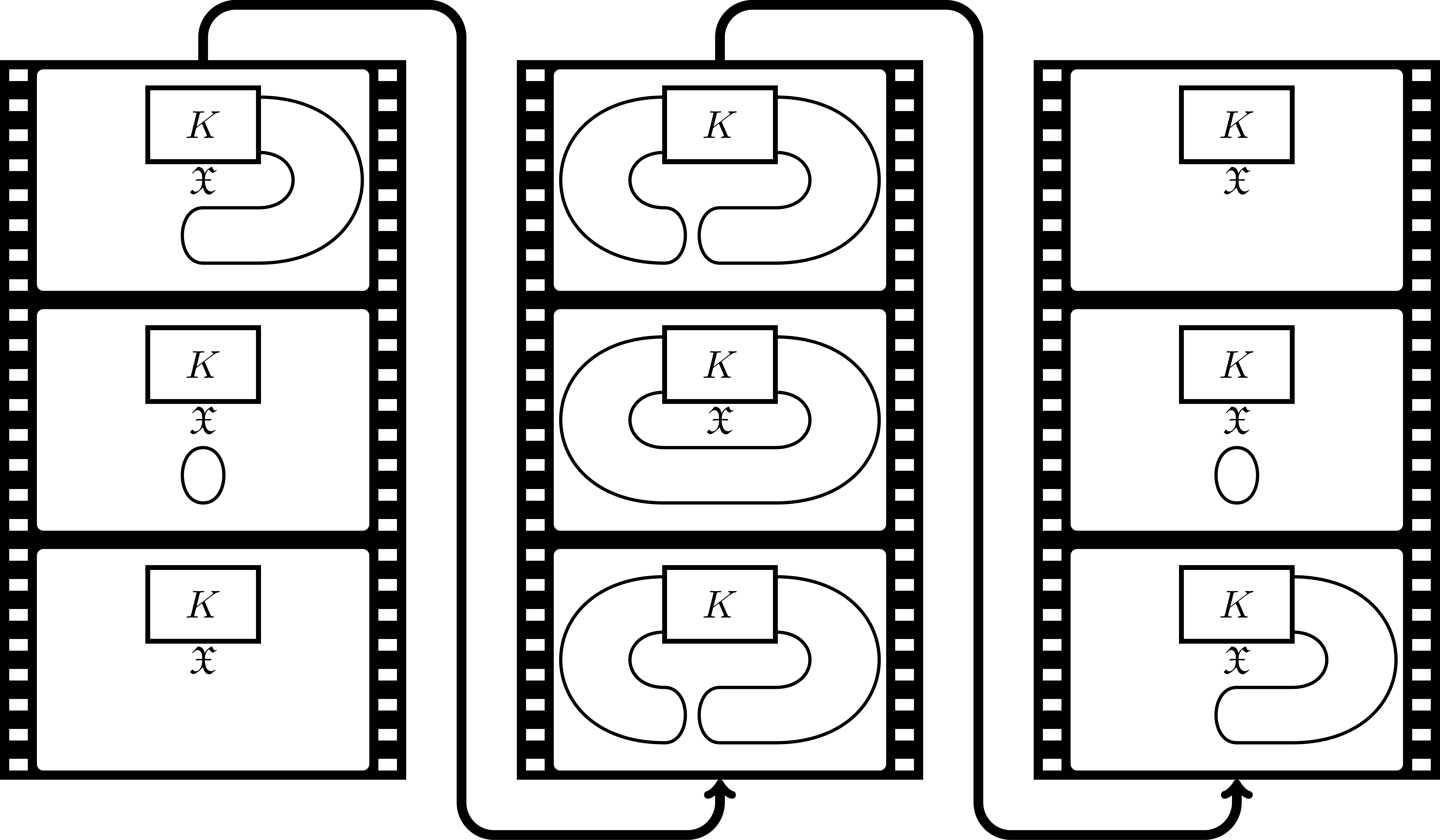}
\end{equation}
The sequence of the two saddles chronologically in the middle of the cobordism forms a horizontal neck, which we will cut using Lemma~\ref{lem:6:horizontalNeck}.  This turns our torus $S^1\times K$ into a sphere knotted with the remainder of the cobordism.  The sphere resembles a \enquote{pill}, with a birth and death and a \enquote{tentacle} projecting from the side reaching through the link.  After the horizontal neck cutting, we have  a difference of two dotted cobordisms where in one the dot is at the end of the projection, and in the other the dot is on the body of the sphere (near the birth and death in \eqref{eq:6:torusmovie}).  We will pull this projection back through the link, one time undotted and one time dotted, so that in both cases we end up with some multiple of an unknotted dotted sphere next to our cobordism.  Throughout the computations, the dot will live on the central link diagram of the cobordisms, i.e., on the link diagram which is chronologically in the middle. To simplify computations, we will therefore only depict that dotted link diagram. 
    
As we pull the projection back, we need to pull it through vertical \enquote{curtains} that come from identity cobordisms of other strands near the crossings of $K$. Consider such a curtain, and suppose that the projection is undotted (so that the dot lies near the birth and death), or that the projection is dotted, but we already moved the dot through the curtain. To pull the projection back through the curtain, we must then move the dot up out of the way in chronological direction, which does not incur a sign as dots commute with Reidemeister II cobordisms. Then, we can simplify the cobordism by movie move 3 and the dot returns to the central link diagram. 
  
As we pull the projection back when it is not dotted, we do not incur a change of sign as we are---at most---moving a dot past a Reidemeister II cobordism, or performing a movie move 3, which---as it is in the first five moves---does not incur any change in sign.  Now we will consider what happens when we pull the dotted projection back. Dots can slide over crossings, so as we pull the projection over other strands of the link near a crossing of $K$, no sign will be incurred.

Now consider what happens as we pull the dotted projection under a crossing of $K$. Let $m$ be either $n$ or $n-2$, depending on whether or not the projection had already been pulled through that crossing the other way. Near the given crossing of $K$, we then see the local picture shown on the left-hand side of \eqref{eq:6:crossing}:

\begin{equation}\label{eq:6:crossing}
    \includeFig{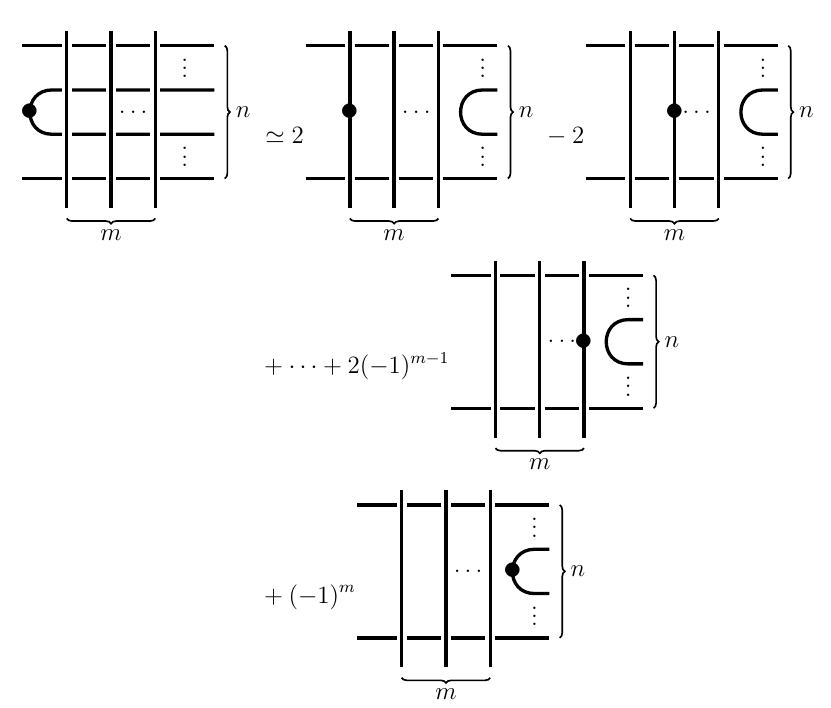}
\end{equation}

As we pull the projection under each successive strand, it leaves behind twice a term with the dot on that strand and negates the term with the dot pulled past.  This leaves us with an alternating sum of twice each term with the dot on the crossing strands, and a single term where the dot is on the projection past all of the strands.  After this point, all terms that have the dot on a strand that is not part of the projection evaluate to zero, as the projection can be freely retracted yielding an undotted sphere.  If $m=n$ three terms now remain, the term with the dot on the projection and twice the difference of terms where the dot is on the crossing strands, which are part of the projection coming back through the crossing.  The latter two terms cancel one another, as the projection can now be retracted until it reaches the crossing again, and we are left with the difference of identical terms.  In either case $m=n$ or $m=n-2$ we are left with only the term with the dot on the retracted projection multiplied by $(-1)^m$.

If $n$ is even, then $m$ is always even and pulling the projection under crossings also incurs no sign. Therefore, we can retract both the dotted projection and the undotted projection without incurring signs, yielding a difference of identical terms that evaluates to zero.

If $K$ has even framing, then there is an even number of crossings, and as the projection must pass under each crossing, the signs incurred with each undercrossing will be overall raised to an even power. This again allows us to retract the dotted projection without incurring a sign, and---as before---cancel with the term featuring the undotted projection.
\end{proof}

\subsection{The Hecke Algebra Action}

To the generator $g_i$ we make the following association

\begin{equation}
    g_i=\left\llbracket\mysym{01}_i\upK\right\rrbracket
\end{equation}

\begin{theorem}\label{thm:6:HeckeAction}
    For a natural number $n$ and a framed knot $K$, if either $n$ is even or $K$ has even framing, then the Hecke algebra $\mathcal{H}(q^2,n)$ at $q=i$ acts on the odd Khovanov homology of the $n$-cable of $K$.
\end{theorem}

\begin{proof}
The Hecke algebra relations in our setting correspond with the following three relations

\vspace{1em}
.

\begin{itemize}
    \item $\left\llbracket\mysym{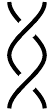}_i\upK\right\rrbracket\simeq-2\,\left\llbracket\mysym{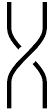}_i\upK\right\rrbracket-\left\llbracket\mysym{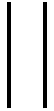}_i\upK\right\rrbracket$
    \item $\left\llbracket\mysym{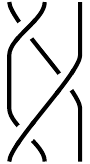}_i\upK\right\rrbracket \simeq \left\llbracket\mysym{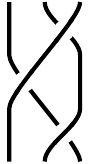}_i\upK\right\rrbracket$
    \item $\left\llbracket\left[\mysym{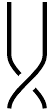}_i\dots\mysym{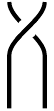}_j\right]\upK\right\rrbracket \simeq\left\llbracket\left[\mysym{12}_i\dots\mysym{11}_j\right]\upK\right\rrbracket$
\end{itemize}
\vspace{1em}

We will now prove these relations one after the other.

\paragraph{Hecke Relation}

\begin{equation}
    \begin{split}
        \left\llbracket\mysym{05}_i\upK\right\rrbracket
        &= \left\llbracket\mysym{09}_i\upK\right\rrbracket-\left\llbracket\mysym{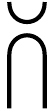}_i\upK\right\rrbracket-\left\llbracket\mysym{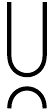}_i\upK\right\rrbracket+\left\llbracket\mysym{08}_i\upK\right\rrbracket\\
        &\simeq \left\llbracket\mysym{03}_i\upK\right\rrbracket-2\,\left\llbracket\mysym{04}_i\upK\right\rrbracket\\
        &= \left\llbracket\mysym{03}_i\upK\right\rrbracket-2\,\left\llbracket\mysym{04}_i\upK\right\rrbracket +\left\llbracket\mysym{03}_i\upK\right\rrbracket-\left\llbracket\mysym{03}_i\upK\right\rrbracket\\
        &= -2\left(-\left\llbracket\mysym{03}_i\upK\right\rrbracket+\left\llbracket\mysym{04}_i\upK\right\rrbracket\right)-\left\llbracket\mysym{03}_i\upK\right\rrbracket\\
        &= -2\,\left\llbracket\mysym{01}_i\upK\right\rrbracket-\left\llbracket\mysym{03}_i\upK\right\rrbracket
    \end{split}
\end{equation}

\paragraph{Cubic Relation (Yang-Baxter Relation or Reidemeister III Relation)}\phantom{.}

{\noindent}As the two cobordisms in equation \eqref{eq:6:YB1} are ambient isotopic, the functoriality of the odd Khovanov bracket implies that the associated chain maps are homotopic up to sign:

\begin{equation}\label{eq:6:YB1}
    \left\llbracket\mysym{13}_i\upK\right\rrbracket  \simeq\pm \left\llbracket\mysym{14}_i\upK\right\rrbracket
\end{equation}

To see that the sign on the right-hand side is a plus, we pre- and postcompose both sides of this relation with chain maps induced by cup and cap tangles. This yields \eqref{eq:6:YB2}, where the sign is the same as in \eqref{eq:6:YB1}:

\begin{equation}\label{eq:6:YB2}
    \left\llbracket\mysym{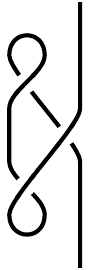}_i\upKnmt\right\rrbracket  \simeq\pm \left\llbracket\mysym{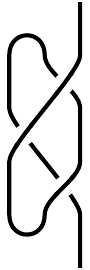}_i\upKnmt\right\rrbracket
\end{equation}

Note that while we arbitrarily fix the signs of the chain maps induced by the cap and the cup at the top and the bottom, we use the same signs (and the same chain maps) throughout this proof. For cap-cup tangles and crossing tangles in that occur in the center of tangle diagrams, we fix the signs as described earlier, and we assume that equation \eqref{eq:6:skein} holds on the nose. Moreover, we assume that the identity tangle in this equation is assigned the identity chain map, not just up sign and homotopy.

With this in mind, we can simplify the left-hand side of \eqref{eq:6:YB2} as follows:

\begin{equation}\label{eq:6:YB3}
    \begin{split}
         \left\llbracket\mysym{15}_i\upKnmt\right\rrbracket 
         &= -\left\llbracket\mysym{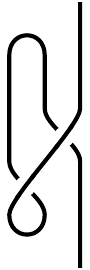}_i\upKnmt\right\rrbracket
         +\left\llbracket\mysym{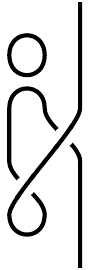}_i\upKnmt\right\rrbracket
         \simeq-\left\llbracket\mysym{17}_i\upKnmt\right\rrbracket \\
         &= \left\llbracket\mysym{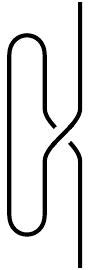}_i\upKnmt\right\rrbracket
         -\left\llbracket\mysym{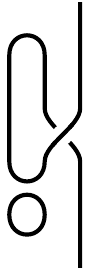}_i\upKnmt\right\rrbracket
         \simeq\left\llbracket\mysym{19}_i\upKnmt\right\rrbracket \\
         &=-\left\llbracket\mysym{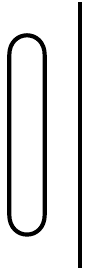}_i\upKnmt\right\rrbracket
         +\left\llbracket\mysym{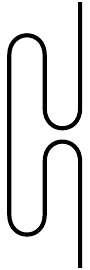}_i\upKnmt\right\rrbracket
         \simeq\left\llbracket\mysym{22}_i\upKnmt\right\rrbracket\\
         &\simeq\pm\left\llbracket\mysym{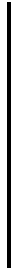}_i\upKnmt\right\rrbracket \\
    \end{split}
\end{equation}

Now consider the term on the right-hand side of \eqref{eq:6:YB2}:

\begin{equation}
    \left\llbracket\mysym{16}_i\upKnmt\right\rrbracket = -\left\llbracket\mysym{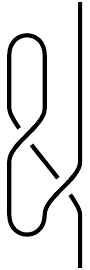}_i\upKnmt\right\rrbracket+\left\llbracket\mysym{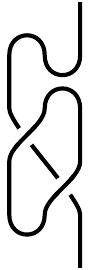}_i\upKnmt\right\rrbracket
\end{equation}

\begin{equation}
    \begin{split}
         \left\llbracket\mysym{24}_i\upKnmt\right\rrbracket
         &= -\left\llbracket\mysym{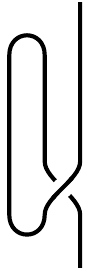}_i\upKnmt\right\rrbracket
         +\left\llbracket\mysym{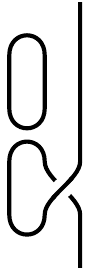}_i\upKnmt\right\rrbracket
         \simeq-\left\llbracket\mysym{26}_i\upKnmt\right\rrbracket \\
         &= \left\llbracket\mysym{21}_i\upKnmt\right\rrbracket
         -\left\llbracket\mysym{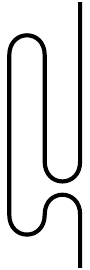}_i\upKnmt\right\rrbracket
         \simeq-\left\llbracket\mysym{22}_i\upKnmt\right\rrbracket
    \end{split}
\end{equation}

\begin{equation}
    \begin{split}
         \left\llbracket\mysym{25}_i\upKnmt \right\rrbracket
         &= \left\llbracket\mysym{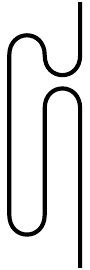}_i\upKnmt\right\rrbracket
         -\left\llbracket\mysym{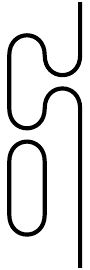}_i\upKnmt\right\rrbracket\\
         &-\left\llbracket\mysym{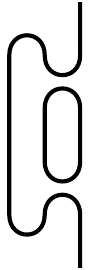}_i\upKnmt\right\rrbracket
         + \left\llbracket\mysym{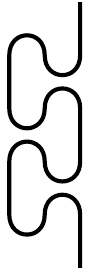}_i\upKnmt\right\rrbracket\\
         &\simeq\left\llbracket\mysym{22}_i\upKnmt\right\rrbracket
         +\left\llbracket\mysym{32}_i\upKnmt\right\rrbracket
    \end{split}
\end{equation}

Overall, for the right-hand tangle we arrive at equation \eqref{eq:6:YB4}.

\begin{equation}\label{eq:6:YB4}
    \left\llbracket\mysym{16}_i\upKnmt\right\rrbracket \simeq 2\,\,\left\llbracket\mysym{22}_i\upKnmt\right\rrbracket+\left\llbracket\mysym{32}_i\upKnmt\right\rrbracket
\end{equation}

From equations \eqref{eq:6:YB2} and \eqref{eq:6:YB3}, the left-hand tangle must also be equal to plus or minus identity.  As the only multiples of identity, which are equal to plus or minus identity, are plus or minus identity themselves, the equation \eqref{eq:6:YB6} follows.

\begin{equation}\label{eq:6:YB6}
    \left\llbracket\mysym{22}_i\upKnmt\right\rrbracket = -\left\llbracket\mysym{32}_i\upKnmt\right\rrbracket
\end{equation}

It then follows that the left-hand and right-hand tangles are homotopic.

\begin{equation}
    \left\llbracket\mysym{15}_i\upKnmt\right\rrbracket \simeq \left\llbracket\mysym{22}_i\upKnmt\right\rrbracket \simeq \left\llbracket\mysym{16}_i\upKnmt\right\rrbracket
\end{equation}

This shows that the overall signs were equal on the original tangles related to the cubic relation.

\begin{equation}
    \left\llbracket\mysym{13}_i\upK\right\rrbracket \simeq \left\llbracket\mysym{14}_i\upK\right\rrbracket
\end{equation}

\paragraph{Braid Commutativity Relation}\phantom{.}

{\noindent}Let $i$ and $j$ be the indices of distant saddles. Without loss of generality let $i<j-1$.  As the two cobordisms in equation \eqref{eq:6:BC1} are ambient isotopic, the functoriality of odd Khovanov bracket implies that the associated chain maps are equal up to sign and homotopy:

\begin{equation}\label{eq:6:BC1}
    \left\llbracket\left[\mysym{11}_i\dots\mysym{12}_j\right]\upK\right\rrbracket \simeq\pm \left\llbracket\left[\mysym{12}_i\dots\mysym{11}_j\right]\upK\right\rrbracket
\end{equation}

If we decompose each side we arrive at the following.

\begin{equation}
    \begin{split}
        \left\llbracket\left[\mysym{11}_i\dots\mysym{12}_j\right]\upK\right\rrbracket &= 
        \left\llbracket\left[\mysym{09}_i\dots\mysym{09}_j\right]\upK\right\rrbracket
        -\left\llbracket\left[\mysym{07}_i\dots\mysym{09}_j\right]\upK\right\rrbracket\\
        &-\left\llbracket\left[\mysym{09}_i\dots\mysym{06}_j\right]\upK\right\rrbracket
        +\left\llbracket\left[\mysym{07}_i\dots\mysym{06}_j\right]\upK\right\rrbracket\\
        \left\llbracket\left[\mysym{12}_i\dots\mysym{11}_j\right]\upK\right\rrbracket &= 
        \left\llbracket\left[\mysym{09}_i\dots\mysym{09}_j\right]\upK\right\rrbracket
        -\left\llbracket\left[\mysym{06}_i\dots\mysym{09}_j\right]\upK\right\rrbracket\\
        &-\left\llbracket\left[\mysym{09}_i\dots\mysym{07}_j\right]\upK\right\rrbracket
        +\left\llbracket\left[\mysym{06}_i\dots\mysym{07}_j\right]\upK\right\rrbracket
    \end{split}
\end{equation}

Both terms have an identity piece and three components involving cap-cup terms.  Again, we can refer to the Jordan-Chevalley decomposition. The semisimple part is the identity component, and the nilpotent part is everything else.  Therefore, as the identity term shows up in both cobordisms without a sign, the overall cobordisms must be equal and not negatives of each other.
\end{proof}

\begin{appendices}

\newpage

\singlespacing

\section{Commutativity \& Associativity Relations} \label{apdx:asscomm}
\vspace{1em}

\paragraph{Canonical Cobordisms}\phantom{.}\bigskip

\begin{equation}
    \includeFig{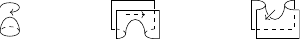}
\end{equation}

\paragraph{Orientation Reversal Relations}\phantom{.}\bigskip

{\noindent}\begin{minipage}{0.5\textwidth}
    \begin{equation}
        \includeFig{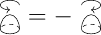}
    \end{equation}
\end{minipage}%
\begin{minipage}{0.5\textwidth}
    \begin{equation}
        \includeFig{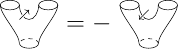}
    \end{equation}
\end{minipage}
\vspace{1em}

\paragraph{Associativity and Frobenius Relations}\phantom{.}\bigskip

{\noindent}\begin{minipage}{0.5\textwidth}
    \begin{equation}
        \includeFig{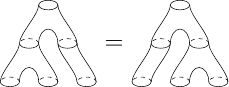}
    \end{equation}
\end{minipage}%
\begin{minipage}{0.5\textwidth}
    \begin{equation}
        \includeFig{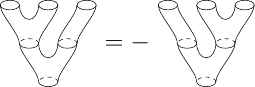}
    \end{equation}
\end{minipage}
\vspace{1em}

\begin{equation}
    \includeFig{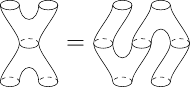}
\end{equation}

\paragraph{Commutativity Relations}\phantom{.}\bigskip

{\noindent}\begin{minipage}{0.5\textwidth}
    \begin{equation}
        \includeFig{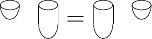}
    \end{equation}
\end{minipage}%
\begin{minipage}{0.5\textwidth}
    \begin{equation}
        \includeFig{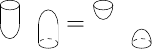}
    \end{equation}
\end{minipage}
\vspace{1em}

{\noindent}\begin{minipage}{0.5\textwidth}
    \begin{equation}
        \includeFig{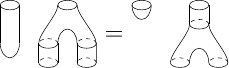}
    \end{equation}
\end{minipage}%
\begin{minipage}{0.5\textwidth}
    \begin{equation}
        \includeFig{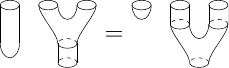}
    \end{equation}
\end{minipage}
\vspace{1em}

{\noindent}\begin{minipage}{0.5\textwidth}
    \begin{equation}
        \includeFig{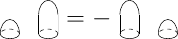}
    \end{equation}
\end{minipage}%
\begin{minipage}{0.5\textwidth}
    \begin{equation}
        \includeFig{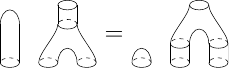}
    \end{equation}
\end{minipage}
\vspace{1em}

{\noindent}\begin{minipage}{0.5\textwidth}
    \begin{equation}
        \includeFig{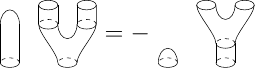}
    \end{equation}
\end{minipage}%
\begin{minipage}{0.5\textwidth}
    \begin{equation}
        \includeFig{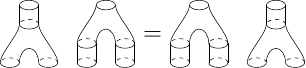}
    \end{equation}
\end{minipage}
\vspace{1em}

{\noindent}\begin{minipage}{0.5\textwidth}
    \begin{equation}
        \includeFig{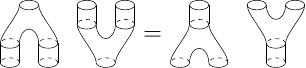}
    \end{equation}
\end{minipage}%
\begin{minipage}{0.5\textwidth}
    \begin{equation}
        \includeFig{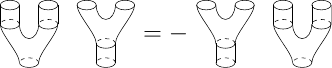}
    \end{equation}
\end{minipage}
\vspace{1em}

\paragraph{Cross and Diamond Relations}\phantom{.}\bigskip

{\noindent}\begin{minipage}{0.5\textwidth}
    \begin{equation}
        \includeFig{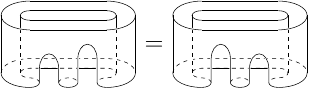}
    \end{equation}
\end{minipage}%
\begin{minipage}{0.5\textwidth}
    \begin{equation}
        \includeFig{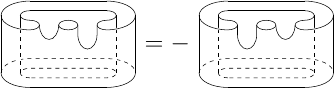}
    \end{equation}
\end{minipage}
\vspace{1em}

\begin{equation}
    \includeFig{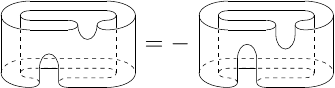}
\end{equation}

\newpage

\paragraph{Commutativity of Faces in the Odd Khovanov Cube}\phantom{.}\bigskip

{\noindent}All the relations above which do not feature deaths or births can appear in a face of the odd Khovanov cube.  Faces in the odd Khovanov cube correspond with two-crossing link diagrams that either commute or anticommute.  This section reinterprets the same relations above as odd Khovanov cubes of two crossing links.  In the following table, all planar circles are omitted.  If a crossing's orientation is not stipulated either choice can be made.  Diagrams are considered up to planar isotopy.  Tangle diagrams can be closed with any crossingless tangle that connects endpoints that have opposite decorations.
\vfill
\begin{center}
\begin{tabular}{ C{1.5in} c }
\begin{tabular}{c}Commuting\\$\sigma_{i,j} = 1$\end{tabular} &
\begingroup
\renewcommand{\arraystretch}{3.25}
\begin{tabular}{ C{0.75in} C{3in} }
type i & \includeTang{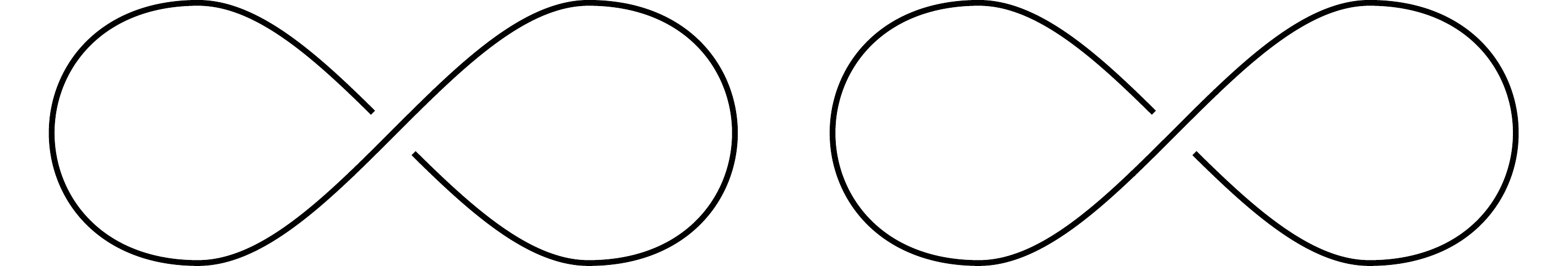} \\
type ii & \includeTang{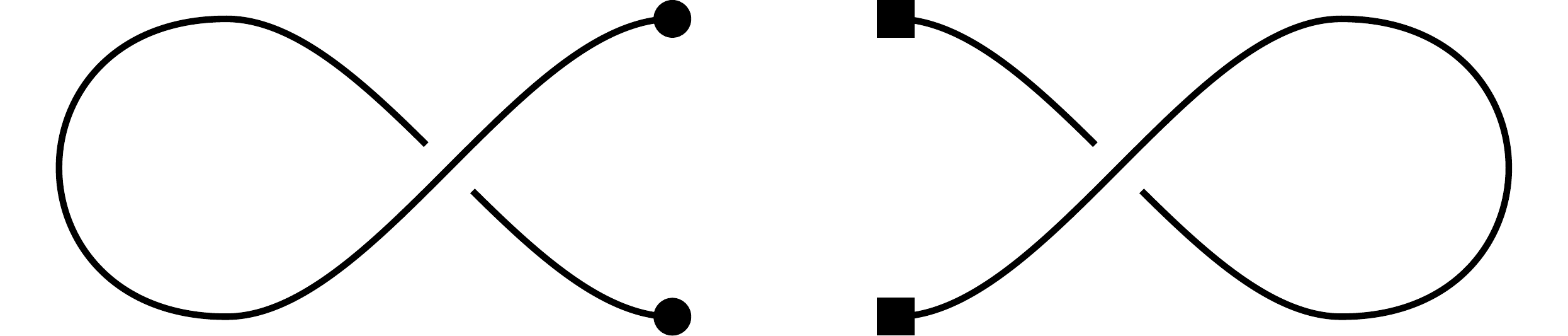} \\
type iii & \includeTang{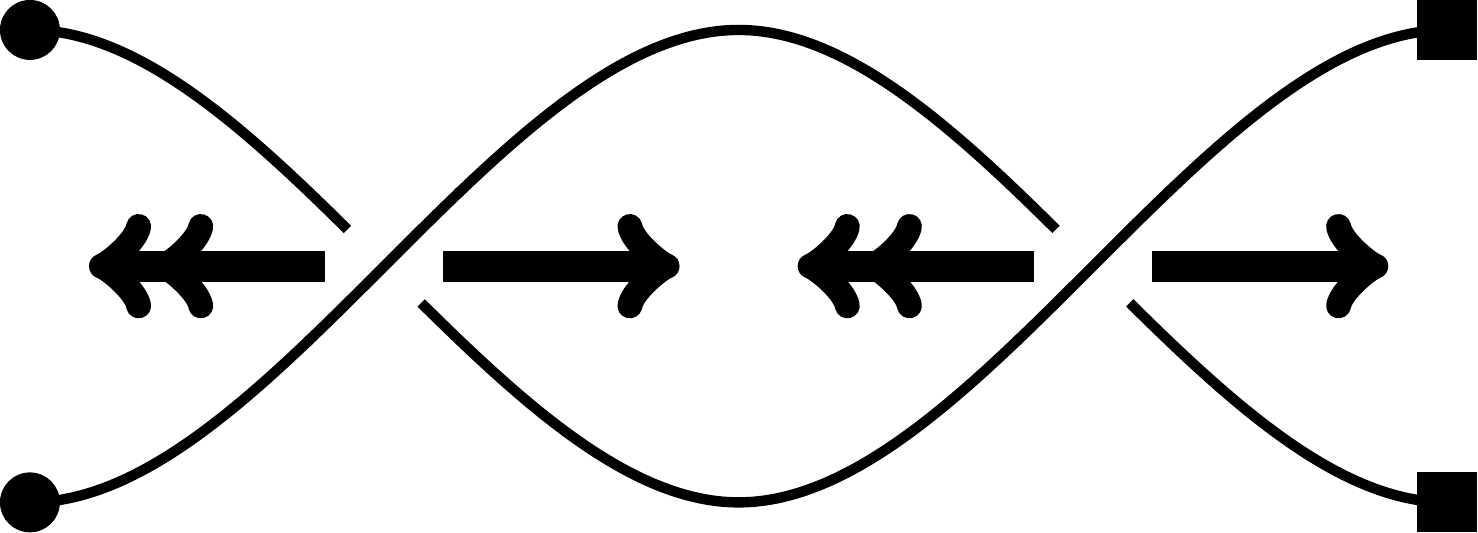} or \includeTang{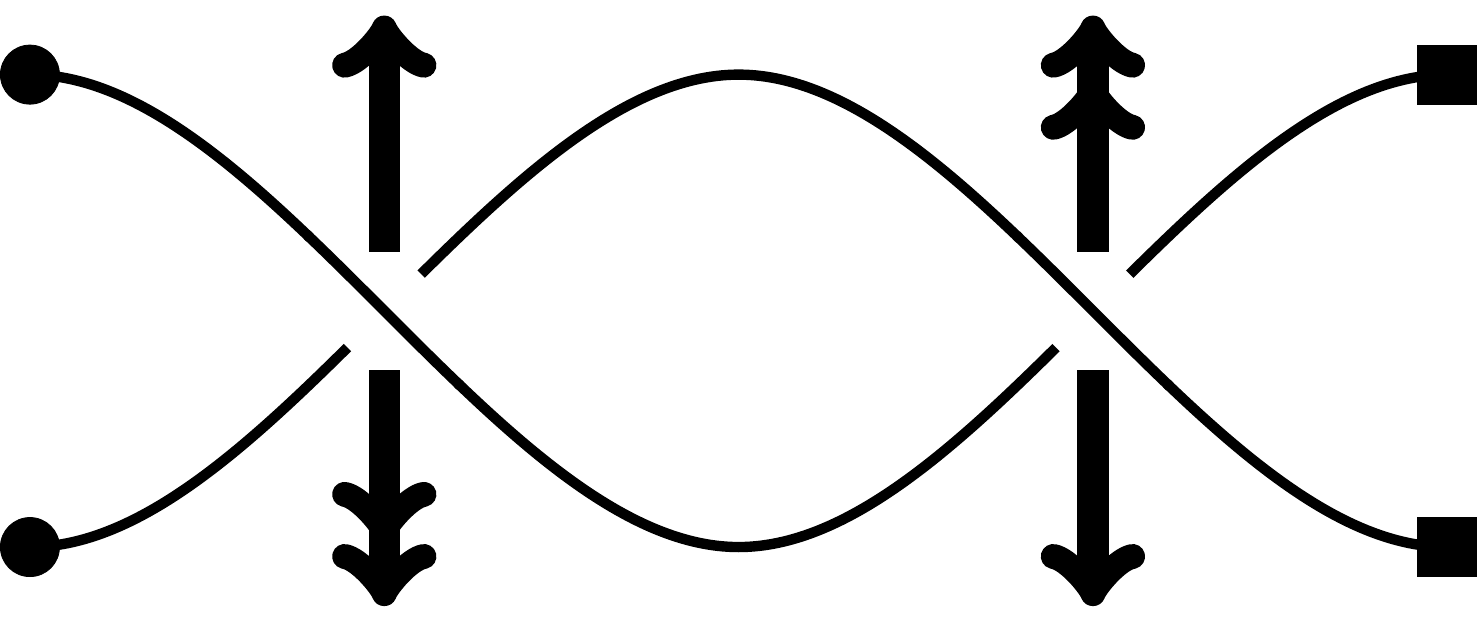}\\
type iv & \includeTang{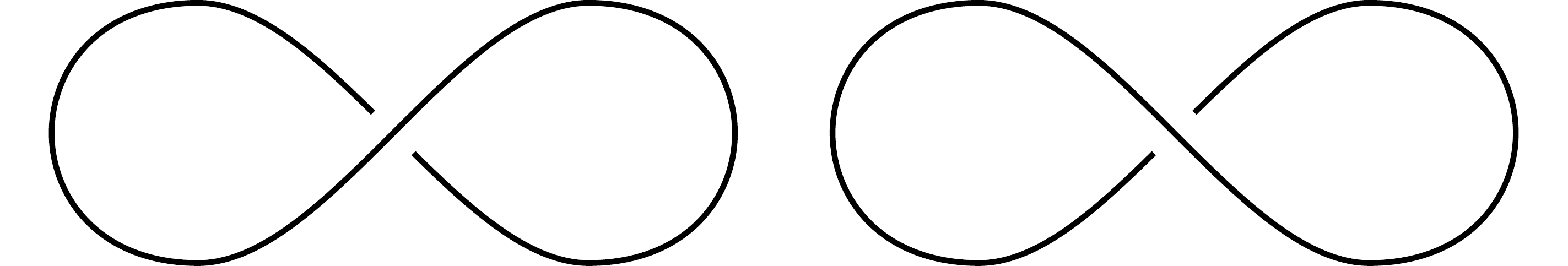} \\
type v & \includeTang{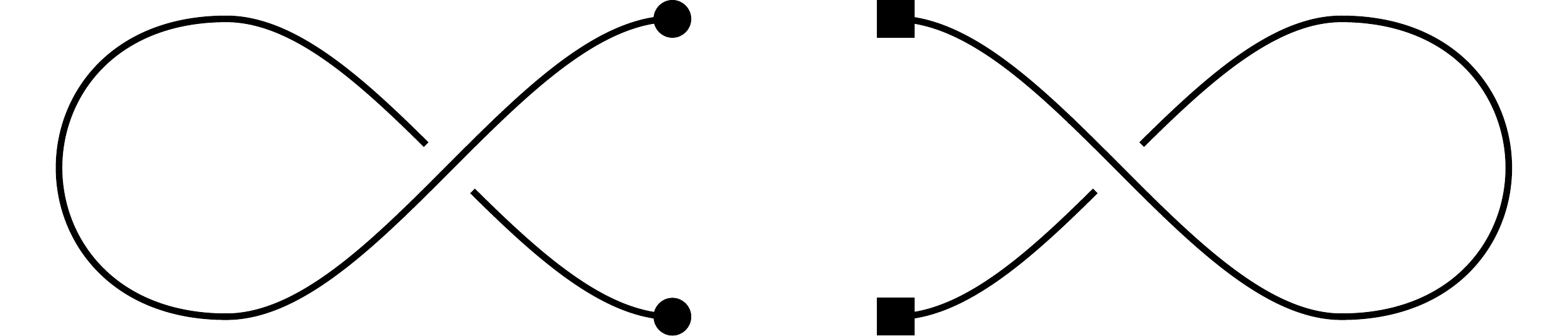} \\
type vi & \includeTang{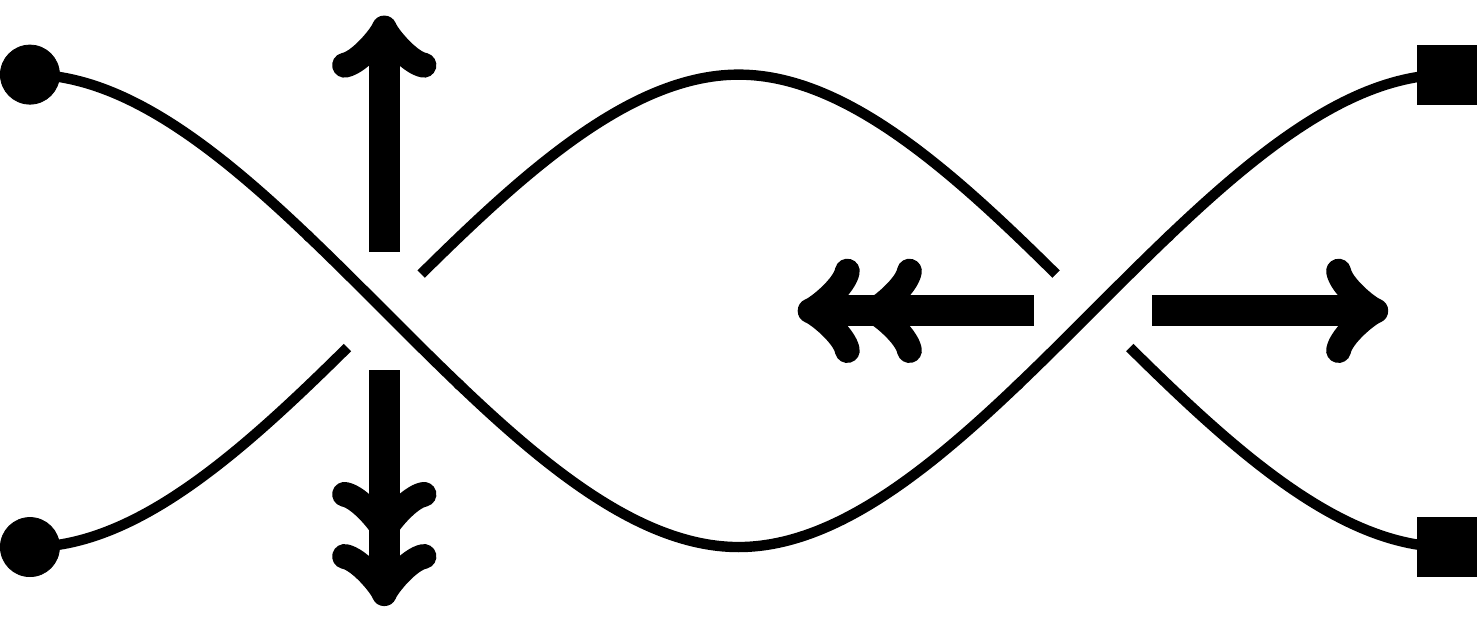} or \includeTang{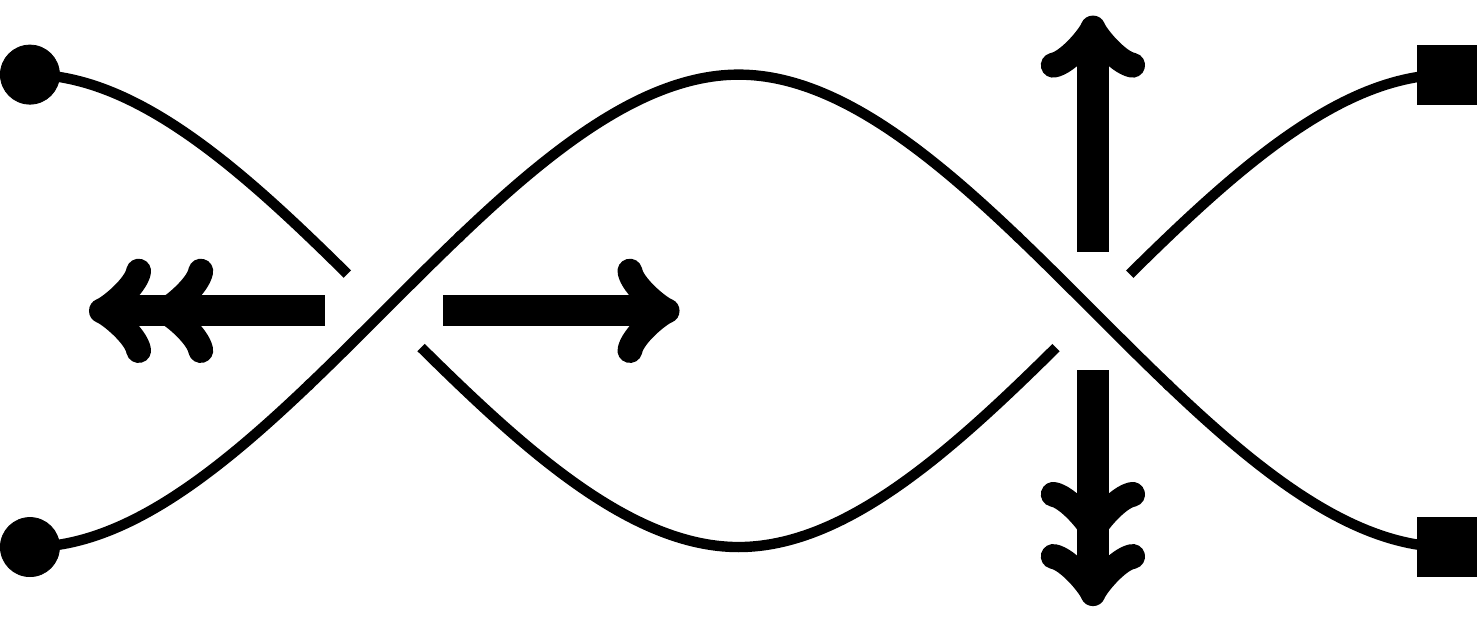}
\end{tabular}
\endgroup
\end{tabular}
\vfill
\vspace{1em}
\rule[0pt]{\textwidth}{2pt}
\vfill
\begin{tabular}{ C{1.5in} c }
\begin{tabular}{c}Anticommuting\\$\sigma_{i,j} = -1$\end{tabular} &
\begingroup
\renewcommand{\arraystretch}{3.25}
\begin{tabular}{ C{0.75in} C{3in} }
type vii & \includeTang{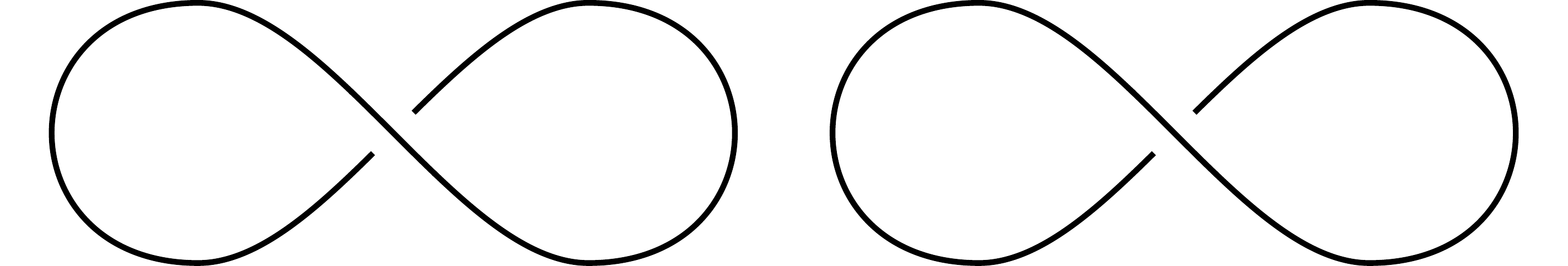} \\
type viii & \includeTang{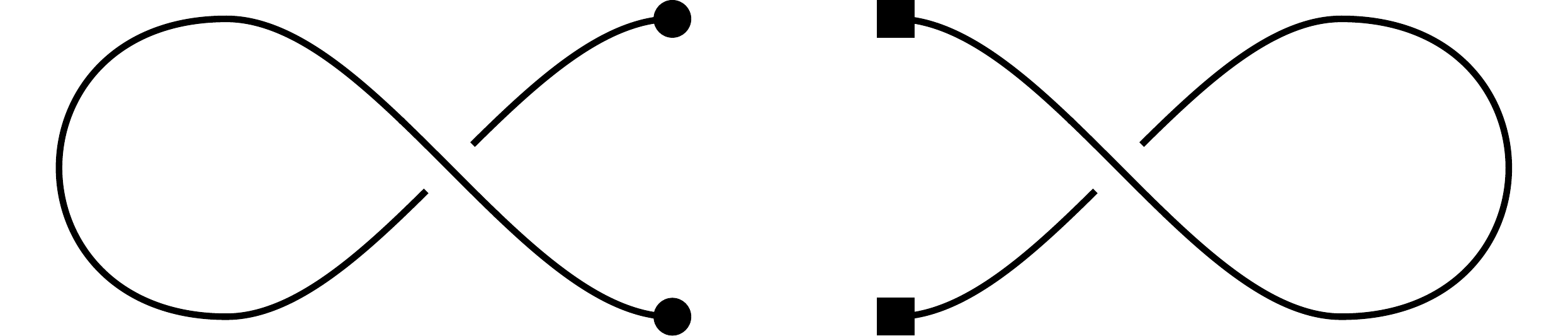} \\
type ix & \includeTang{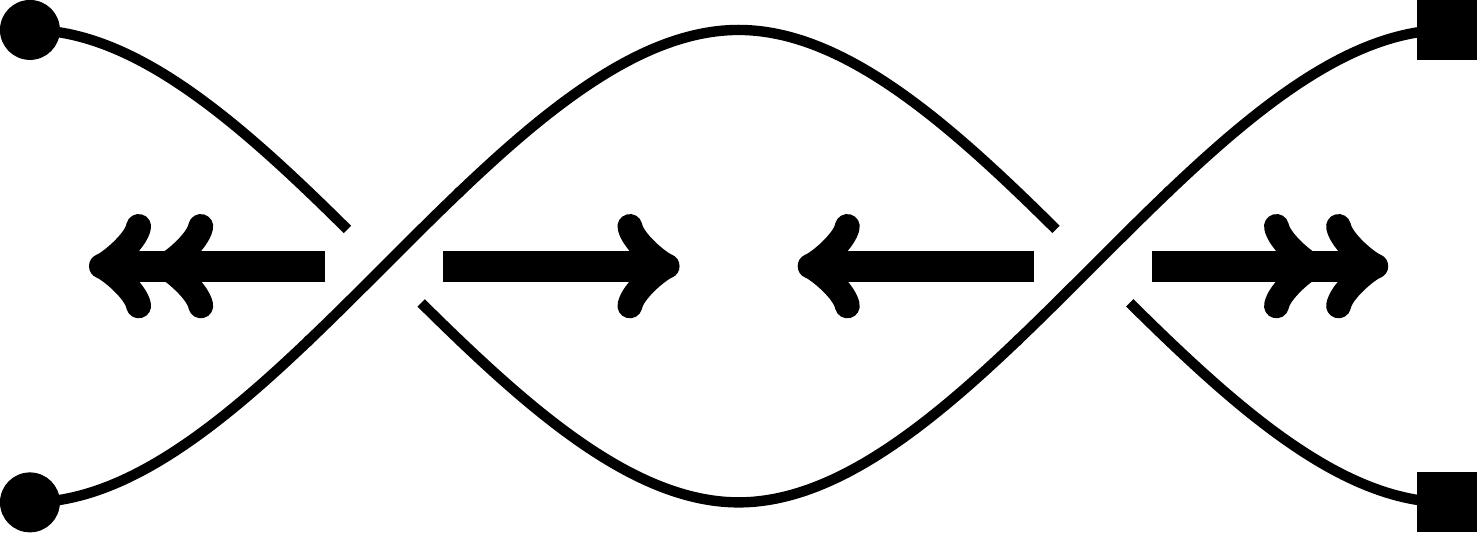} or \includeTang{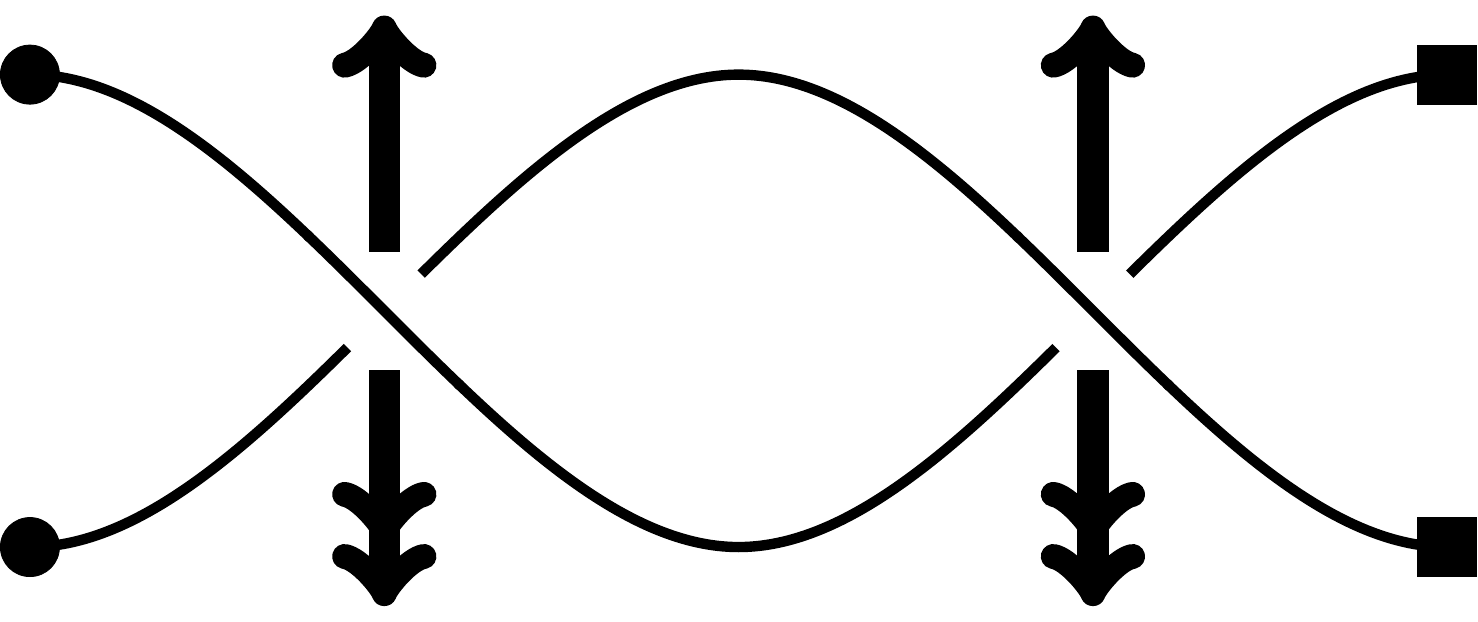}\\
type x & \includeTang{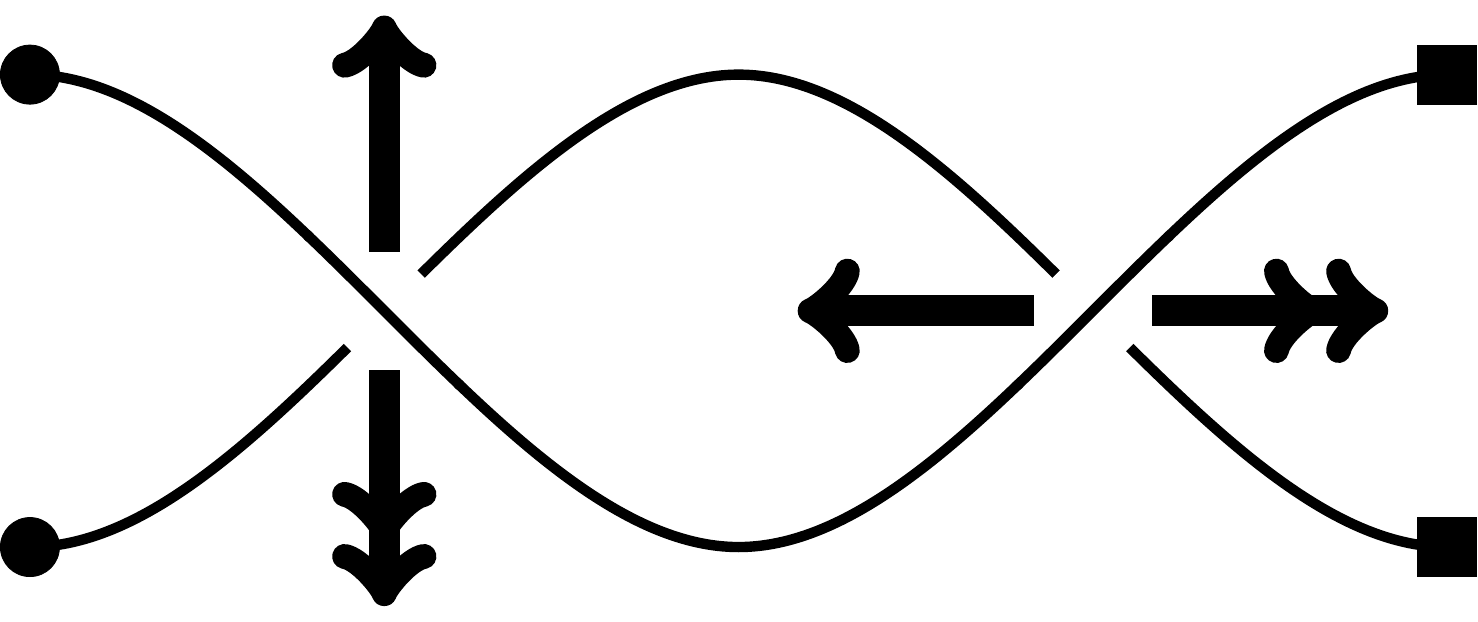} or \includeTang{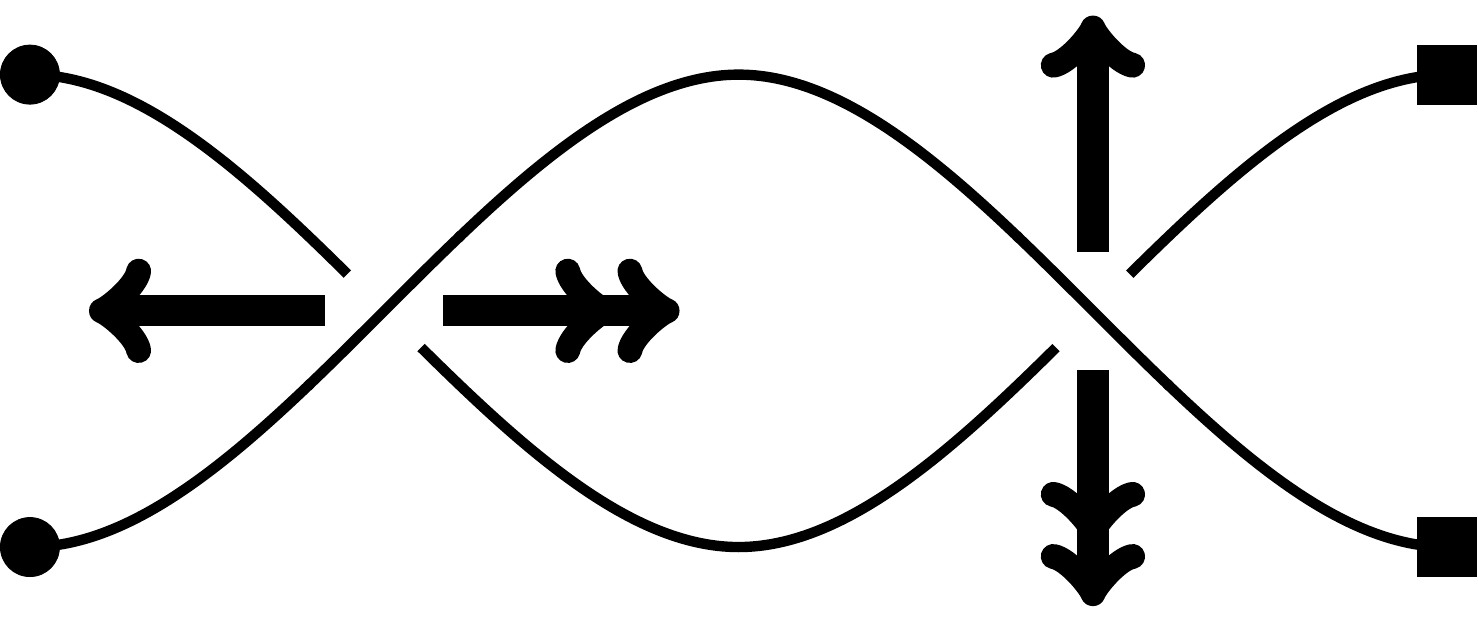}
\end{tabular}
\endgroup
\end{tabular}
\end{center}
\vfill

\section{Four-Tube Relation Variants} \label{apdx:4tuvariants}

In this section, we work out a few results of the (4Tu) relation which are used in the arguments on the invariance of the odd Khovanov functor under movie moves 12 and 13.
\begin{itemize}
    \item The first result is shown in equation \eqref{eq:A:11} with supporting computations shown in equations \eqref{eq:A:12}-\eqref{eq:A:13}.
    \item The second result is shown in equation \eqref{eq:A:21} with supporting computations shown in equations \eqref{eq:A:22}-\eqref{eq:A:23}.
    \item The third result is shown in equation \eqref{eq:A:31} with supporting computations shown in equations \eqref{eq:A:32}-\eqref{eq:A:33}.
\end{itemize}
\vfill
\begin{equation}
\label{eq:A:11}
\includeCobEq{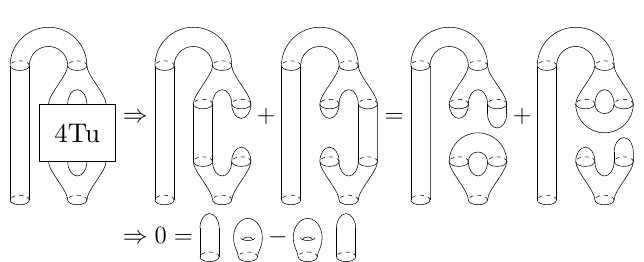}
\end{equation}
\vfill
\begin{equation}
\label{eq:A:21}
\includeCobEq{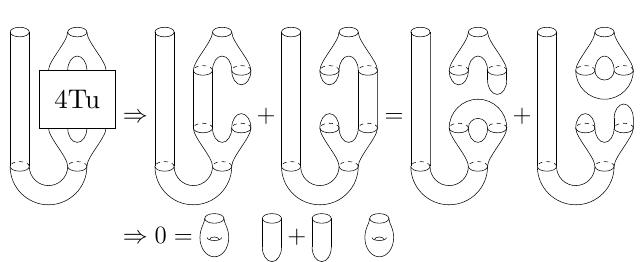}
\end{equation}
\vfill
\begin{equation}
\label{eq:A:31}
\includeCobEq{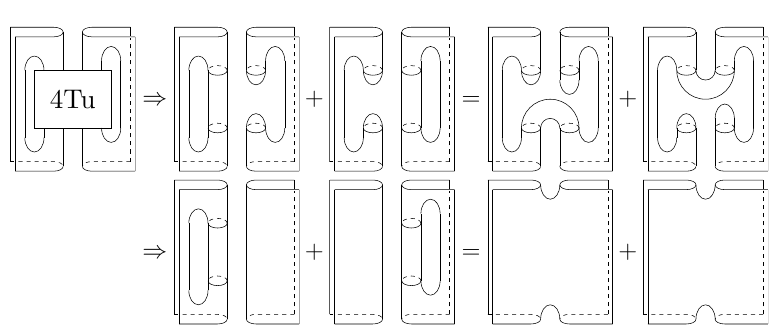}
\end{equation}
\newpage
\begin{equation}
\label{eq:A:12}
\includeCobEq{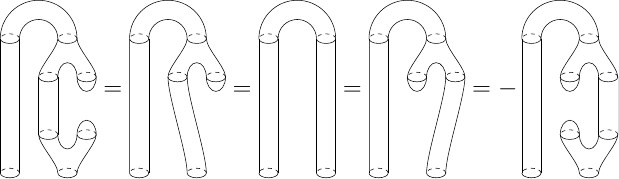}
\end{equation}
\vfill
\begin{equation}
\includeCobEq{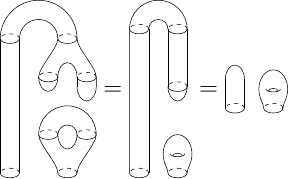}
\end{equation}
\vfill
\begin{equation}
\label{eq:A:13}
\includeCobEq{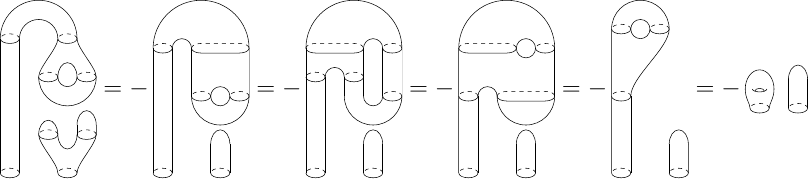}
\end{equation}
\vfill
\begin{equation}
\label{eq:A:22}
\includeCobEq{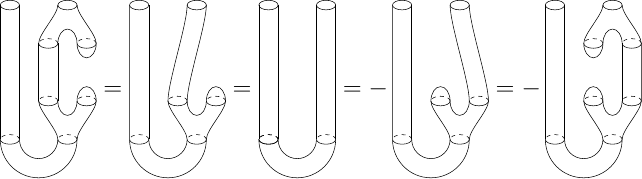}
\end{equation}
\vfill
\begin{equation}
\includeCobEq{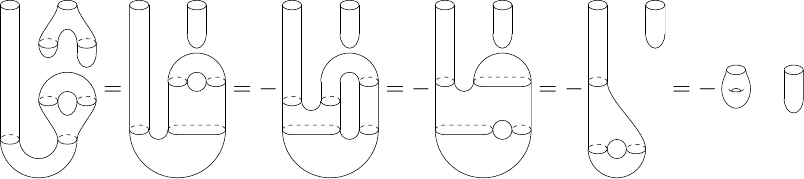}
\end{equation}
\vfill
\begin{equation}
\label{eq:A:23}
\includeCobEq{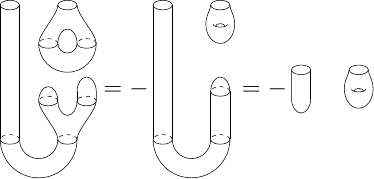}
\end{equation}
\newpage
\hspace{0pt}
\vfill
\begin{equation}
\label{eq:A:32}
\includeCobEq{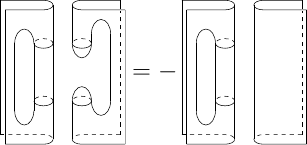}
\end{equation}
\vfill
\begin{equation}
\includeCobEq{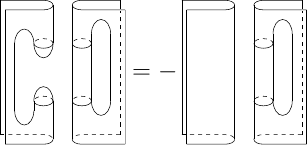}
\end{equation}
\vfill
\begin{equation}
\includeCobEq{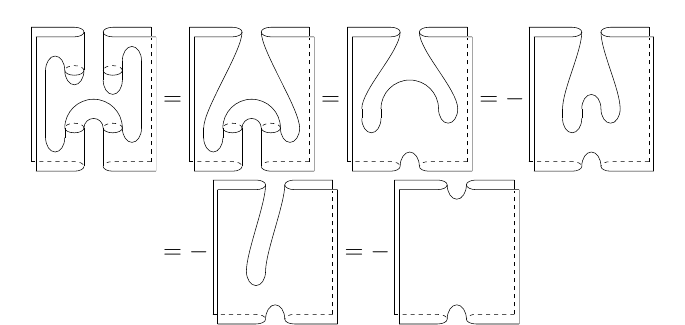}
\end{equation}
\vfill
\begin{equation}
\label{eq:A:33}
\includeCobEq{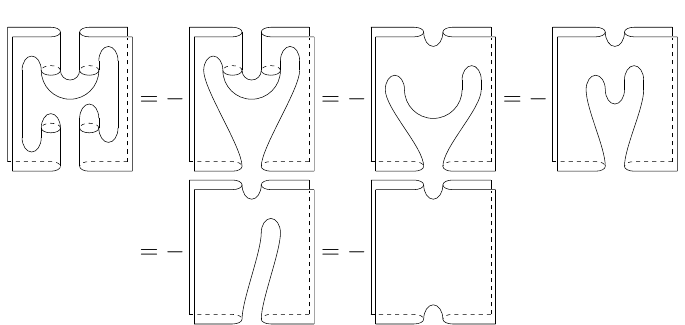}
\end{equation}
\vfill
\end{appendices}
\newpage

\addcontentsline{toc}{section}{References}
\printbibliography

\end{document}